\documentclass[a4paper, 11pt]{amsart}
\usepackage[utf8]{inputenc}
\usepackage[T1]{fontenc}
\usepackage[english]{babel}
\usepackage{graphicx}
\usepackage{stmaryrd}
\usepackage{tikz}
\usepackage{tikz-cd} 
\usepackage[all]{xy}
\usepackage{comment}
\usepackage{bm}
\usepackage{comment}
\usepackage{amsmath,amsfonts,amssymb}
\usepackage[a4paper]{geometry}
\usepackage{amsthm}
\usepackage{float}
\usepackage{enumitem}
\usepackage{hyperref}
\usepackage{xcolor}
\hypersetup{
    colorlinks=true,
    linkcolor={blue},
    citecolor={blue},
    urlcolor={blue}}
\newtheorem{te}{Theorem}[section]
\newtheorem{st}[te]{Standing Assumption}

\newtheorem{prop}[te]{Proposition}

\newtheorem{co}[te]{Corollary}

\newtheorem{qu}[te]{Question}
\newtheorem{lemme}[te]{Lemma}
\newtheorem*{lemme2}{Lemma}
\theoremstyle{definition}

\newtheorem{de}[te]{Definition}
\newtheorem{ex}[te]{Example}

\theoremstyle{remark}
\newtheorem{rque}[te]{Remark}

\newlength{\plarg}
\usetikzlibrary{cd}
\calclayout
\newcommand{\norm}[1]{\lvert\lvert#1\rvert\rvert}
\newcommand{\abs}[1]{\lvert#1\rvert}
\title{Formal solutions and the first-order theory of acylindrically hyperbolic groups}
\author{Simon André and Jonathan Fruchter}
\begin{document}

\begin{minipage}{\linewidth}

\maketitle

\vspace{0mm}

	\begin{abstract}
We generalise Merzlyakov's theorem about the first-order theory of non-abelian free groups to all acylindrically hyperbolic groups. As a corollary, we deduce that if $G$ is an acylindrically hyperbolic group and $E(G)$ denotes the unique maximal finite normal subgroup of $G$, then $G$ and the HNN extension $G\dot{\ast}_{E(G)}$, which is simply the free product $G\ast\mathbb{Z}$ when $E(G)$ is trivial, have the same $\forall\exists$-theory. As a consequence, we prove the following conjecture, formulated by Casals-Ruiz, Garreta and de la Nuez Gonz\'alez: acylindrically hyperbolic groups have trivial positive theory. In particular, one recovers a result proved by Bestvina, Bromberg and Fujiwara, stating that, with only the obvious exceptions, verbal subgroups of acylindrically hyperbolic groups have infinite width.\end{abstract}
	
\end{minipage}

\thispagestyle{empty}

\vspace{5mm}

\section{Introduction}

Given a group $G$, a natural model-theoretic question is whether or not $G$ and $G\ast\mathbb{Z}$ have the same first-order theory. This problem was first considered by Tarski in the case of free groups. Around 1945, he posed the following question: are all non-abelian free groups elementarily equivalent? A positive answer to this question was given by Sela in \cite{Sel06} (see also \cite{KM06} by Kharlampovich and Myasknikov). Then, Sela generalised this result in two directions: first, he proved in \cite{Sel09} that every torsion-free non-elementary hyperbolic group $G$ is elementarily equivalent to $G\ast\mathbb{Z}$. A few years later, he established the same result in the case where $G$ is a non-trivial free product, different from the infinite dihedral group $D_{\infty}=\mathbb{Z}/2\mathbb{Z}\ast\mathbb{Z}/2\mathbb{Z}$ (see \cite{Sel10}). More precisely, he proved the following stronger result: $G$ is elementarily embedded into $G\ast\mathbb{Z}$.

\smallskip

All these groups (namely non-elementary hyperbolic groups and non-elementary free products) have in common the property of being \emph{acylindrically hyperbolic}, meaning that they admit a non-elementary \emph{acylindrical action} on a hyperbolic space (for details, we refer the reader to Section \ref{prelim_acyl}). The main result of this paper is a partial generalisation of the above-mentioned theorems of Sela to all acylindrically hyperbolic groups (see Theorems \ref{th1} and \ref{th11} below). This wide class of groups, introduced by Osin in \cite{Osi16} in order to unify several classes of negatively-curved groups considered by different authors (in particular, see \cite{DGO17}), has been intensively studied in the past few years. Examples of acylindrically hyperbolic groups include, notably, all non-elementary (relatively) hyperbolic groups, all but finitely many mapping class groups of surfaces of finite type, $\mathrm{Out}(F_n)$ for $n\geq 2$, most 3-manifold groups, all non-cyclic and directly indecomposable right-angled Artin groups, and more generally any group acting geometrically on a $\mathrm{CAT}(0)$ space and containing a rank-one isometry, many fundamental groups of graphs of groups, and many other groups.

\smallskip

Despite an intense activity around acylindrically hyperbolic groups in geometric group theory, very little is know about the first-order theory of these groups. Dahmani, Guirardel and Osin proved that acylindrically hyperbolic groups are not superstable (see Theorem 8.1 in \cite{DGO17}). Recently, Groves and Hull adapted some of Sela's techniques to the context of acylindrically hyperbolic groups and initiated the study of solutions of systems of equations over such groups (see \cite{GH19}). Last, building on Groves' and Hull's version of Sela's \emph{shortening argument} (for further details, see \ref{shortening}), the second-named author of the present paper proved a generalisation of Merzlyakov's celebrated theorem \cite{Mer66} for torsion-free acylindrically hyperbolic groups (see \cite{Fru19}). An important part of our paper is devoted to an extension of Merzlyakov's theorem to all acylindrically hyperbolic groups, possibly with torsion; this involves techniques used in \cite{And19a} by the first-named author in the setting of hyperbolic groups. 

\smallskip

An \emph{$\forall\exists$-sentence} is a first-order sentence of the form $\forall\bm{x} \ \exists\bm{y} \ \psi(\bm{x},\bm{y})$, where $\bm{x}$ and $\bm{y}$ are two tuples of variables, and $\psi$ is a quantifier-free formula in these variables. The set of such sentences satisfied by a group $G$ is called the \emph{$\forall\exists$-theory} of $G$. Before stating our main result, recall that every acylindrically hyperbolic group $G$ admits a unique maximal finite normal subgroup, denoted by $E(G)$ (see \cite{DGO17}, Theorem 2.24). In what follows, $G\dot{\ast}_{E(G)}$ denotes the HNN extension where the stable letter acts trivially, that is the group \[G\ast_{E(G)}(\mathbb{Z}\times E(G))=\langle G,t \ \vert \ [t,g]=1, \ \forall g\in E(G)\rangle.\]

\begin{te}\label{th1}If $G$ is an acylindrically hyperbolic group, then $G$ and $G\dot{\ast}_{E(G)}$ have the same $\forall\exists$-theory.\end{te}

\begin{rque}Note that if the finite group $E(G)$ is trivial, the group $G\dot{\ast}_{E(G)}$ is simply the free product $G\ast\mathbb{Z}$. If $E(G)$ is non-trivial, one easily sees that $G\ast\mathbb{Z}$ cannot have the same $\forall\exists$-theory as $G$, since the existence of a non-trivial normal finite subgroup is expressible by means of a $\forall\exists$-sentence.\end{rque}

In fact, we prove a slightly stronger result. We say that the inclusion $i$ of a group $G$ into an overgroup $G'$ is an \emph{$\exists\forall\exists$-elementary embedding} if the following condition is satisfied: for every first-order formula of the form \[\phi(\bm{t}):\exists \bm{x} \ \forall\bm{y} \ \exists \bm{z} \ \psi(\bm{x},\bm{y},\bm{z},\bm{t}),\] where $\psi(\bm{x},\bm{y},\bm{z},\bm{t})$ is a quantifier-free formula, and for every tuple $\bm{g}$ of elements of $G$ of the same arity as $\bm{t}$, if the statement $\phi(\bm{g})$ holds in $G$, then $\phi(i(\bm{g}))$ holds in $G'$.

\begin{te}\label{th11}Let $G$ be an acylindrically hyperbolic group. The canonical inclusion of $G$ into $G\dot{\ast}_{E(G)}$ is an $\exists\forall\exists$-elementary embedding. In particular, $G$ and $G\dot{\ast}_{E(G)}$ have the same $\forall\exists$-theory.\end{te}

\begin{rque}This result was proved by the first author in \cite{And19a} under the stronger assumption that the group $G$ is hyperbolic (possibly with torsion).
\end{rque}

\begin{rque}As an immediate consequence of Theorem \ref{th11}, one recovers a result of Hull and Osin stating that acylindrically hyperbolic groups are mixed identity free, see \cite{HO16}.\end{rque}

For now, it is an open question whether Theorems \ref{th1} and \ref{th11} above remain true if one considers the whole first-order theories of $G$ and $G\dot{\ast}_{E(G)}$ instead of the $\forall\exists$ or $\exists\forall\exists$-fragments of these theories. This question can be viewed as a broad generalisation of Tarski's problem about elementary equivalence of non-abelian free groups.

\begin{qu}\label{question}Let $G$ be an acylindrically hyperbolic group. 
\begin{enumerate}
\item Are $G$ and $G\dot{\ast}_{E(G)}$ elementarily equivalent? 
\item Is $G$ elementarily embedded into $G\dot{\ast}_{E(G)}$?
\end{enumerate}
\end{qu}

As mentioned before, Sela proved that the answer to both of these questions is `Yes' under the stronger assumption that $G$ is a torsion-free non-elementary hyperbolic group or a non-trivial and non-dihedral free product. In all other cases, the answer is not known. 

\smallskip

Moreover, let us note that we do not know of any example of a finitely generated group $G$ that is not acylindrically hyperbolic and that has the same first-order theory, or even the same $\forall\exists$-theory, as $G\ast\mathbb{Z}$. The question of the existence of such a group is closely related to that of the preservation of acylindrical hyperbolicity under elementary equivalence among finitely generated groups (see Section \ref{comments_section} for further comments). It is worth mentioning the following corollary of Theorem \ref{th11} (see Proposition \ref{acyl_preservation}).

\begin{co}
Let $G$ be an acylindrically hyperbolic group, and let $H$ be a group that admits a non-trivial splitting over a virtually abelian group. Suppose that $G$ and $H$ are elementarily equivalent (or simply that they have the same $\exists\forall\exists$-theory). Then the group $H$ is acylindrically hyperbolic.
\end{co}

\begin{rque}
As a consequence, if there exists a group $G$ that is not acylindrically hyperbolic and such that $G$ and $G\ast\mathbb{Z}$ are elementarily equivalent, then all non-trivial splittings of $G$ (if they exist) have sufficiently complicated edge groups. For instance, if $G$ is a generalized Baumslag-Solitar group, then $G$ and $G\ast\mathbb{Z}$ are not elementarily equivalent.
\end{rque}

\subsection*{Positive theory, verbal subgroups}A first-order sentence is called \emph{positive} if it does not involve inequalities. We say that a group $G$ has \emph{trivial positive theory} if every positive sentence satisfied by $G$ is satisfied by all groups. In \cite{Mer66}, Merzlyakov proved that non-abelian free groups have trivial positive theory. As a consequence, $G$ has trivial positive theory if and only if it has the same positive theory as $F_n$, for any $n\geq 2$. Recently, in \cite{CGN19} and \cite{CGKN19}, Casals-Ruiz, Garreta, Kazachkov and de la Nuez Gonz\'alez proved that many groups acting non-trivially on trees have trivial positive theory. In particular, they showed that every acylindrically hyperbolic group that acts hyperbolically and irreducibly on a tree has trivial positive theory (see \cite[Corollary 8.2]{CGN19}). They also established the following quantifier elimination result (see \cite[Theorem 6.3]{CGN19}): a group has trivial positive theory if and only if it has trivial positive $\forall\exists$-theory. Using this fact and relying on Theorem \ref{th1}, we prove the following result (see Section \ref{sectionpos}), which was conjectured in \cite{CGN19} (Conjecture 9.1).

\begin{co}\label{pos}Acylindrically hyperbolic groups have trivial positive theory.\end{co}

Let $G$ be a group, and let $w$ be an element of the free group $F(x_1,\ldots,x_k)$. This element $w$ induces a map $\bm{g}\in G^k \mapsto w(\bm{g})\in G$, also denoted by $w$. We say that the \emph{verbal subgroup} $w(G)=\langle\lbrace w(\bm{g}), \ \bm{g}\in G^k\rbrace\rangle$ has \emph{finite width} if there exists an integer $m\in\mathbb{N}$ such that any $g\in w(G)$ can be represented as a product of at most $m$ values of $w$ and their inverses, and the smallest such integer $m$ is called the \emph{width} of $w(G)$. For instance, for $n\geq 3$, there exists a constant $C(n)$ such that every element of $\mathrm{SL}_n(\mathbb{Z})$ is a product of $C(n)$ commutators (see \cite{AM18}); in other words, the derived subgroup of $\mathrm{SL}_n(\mathbb{Z})$ (that is $\mathrm{SL}_n(\mathbb{Z})$ itself) has finite width. Otherwise, one says that $w(G)$ has \emph{infinite width}. 

\smallskip

Let $e_i$ be the sum of the exponents of $x_i$ in $w$. If they are all $0$, define $d(w) = 0$. Otherwise, let $d(w)$ be their greatest common divisor. The following holds (see Section \ref{sectionpos}, Lemma \ref{lemmeposintro}): if $G$ has trivial positive theory, then $w(G)$ has infinite width, except if $w$ is trivial or $d(w)=1$ (in which cases the width is equal to $1$). As a consequence of Corollary \ref{pos}, one recovers the main result of \cite{BBF19}, due to Bestvina, Bromberg and Fujiwara.

\begin{co}\label{BBF}Let $G$ be an acylindrically hyperbolic, let $k\geq 1$ be an integer and let $w$ be a non-trivial element of $F_k$. If $d(w)\neq 1$, then $w(G)$ has infinite width.\end{co}

It is worth noting that we do not know any group $G$ with non-trivial positive theory and such that $w(G)$ has infinite width for every non-trivial $w$ satisfying $d(w)\neq 1$ (see \cite[Section 9.8]{CGN19} for further discussion).

\smallskip

\subsection*{Merzlyakov's theorem}\label{merz}In Section \ref{proof_main_th}, we deduce Theorem \ref{th11} from a generalisation of Merzlyakov's theorem \cite{Mer66}. Assuming that a non-abelian free group $F$ satisfies the positive first-order sentence \[\forall \bm{x} \ \exists\bm{y} \ \Sigma(\bm{x},\bm{y})=1,\] where $\Sigma(\bm{x},\bm{y})=1$ denotes a finite system of equations, Merzlyakov's theorem asserts that there exists a retraction from $\langle \bm{x},\bm{y} \ \vert \ \Sigma(\bm{x},\bm{y})=1\rangle$ onto the free group $F(\bm{x})$ on $\bm{x}$. Upon closer inspection, this result resembles the classical implicit function theorem in the sense that it enables one to convert the relations between the tuples $\bm{x}$ and $\bm{y}$ into a function. This is why Merzlyakov's theorem is sometimes referred to as an implicit function theorem for groups. This fundamental result was one of the first steps in Sela's positive answer to Tarski's question about the elementary equivalence of non-abelian free groups.

\smallskip

Let us mention that previous generalisations of Merzlyakov's theorem were proved for torsion-free hyperbolic groups, for hyperbolic groups with torsion, and for $\pi$-groups (that is pairs of the form $(F,\pi)$ where $\pi : F\rightarrow G$ is a homomorphism), respectively by Sela (see \cite{Sel09}),  by Heil (see \cite{Hei18}), and by de la Nuez Gonz\'alez (see \cite{Nuez17}).

\smallskip

Given a group $G$ and an element $g\in G$, we denote by $\mathrm{ad}(g)$ the inner automorphism $x\in G \mapsto gxg^{-1}$. Before stating our generalisation of Merzlyakov's theorem to all acylindrically hyperbolic groups, let us introduce the following definition.

\begin{de}Let $G$ be a group, and let $H$ be a subgroup of $G$. We define the subgroup $\mathrm{Aut}_{G}(H)$ of $\mathrm{Aut}(H)$ as follows:
\[\mathrm{Aut}_G(H)=\lbrace \sigma\in\mathrm{Aut}(H) \ \vert \ \exists g\in G, \ \mathrm{ad}(g)_{\vert H}=\sigma\rbrace.\]
\end{de}

We prove the following version of Merzlyakov's theorem (in Section \ref{Merz_section}, we give a more general statement allowing us to deal with finite disjunctions of finite systems of equations and inequations). In the case where $G$ is torsion-free and the first-order sentence considered in the theorem is positive, this result was proved by the second author (see \cite{Fru19}).

\begin{te}\label{th0bis}Let $G$ be an acylindrically hyperbolic group, and let $\bm{a}$ be a tuple of elements of $G$ (called constants). Fix a presentation $\langle \bm{a} \ \vert \ R(\bm{a})=1\rangle$ for the subgroup of $G$ generated by $\bm{a}$. Let \[\Sigma(\bm{x},\bm{y},\bm{a})=1 \ \wedge \ \Psi(\bm{x},\bm{y},\bm{a})\neq 1\] be a finite system of equations and inequations over $G$, where $\bm{x}$ and $\bm{y}$ are two tuples of variables. Let $G_{\Sigma}$ denote the following finitely generated group, finitely presented relative to $\langle \bm{a} \ \vert \ R(\bm{a})=1\rangle$: \[\langle \bm{x},\bm{y},\bm{a} \ \vert \ R(\bm{a})=1, \ \Sigma(\bm{x},\bm{y},\bm{a})=1\rangle.\]Let $p=\vert \bm{x}\vert$ be the arity of $\bm{x}$, and let $x_i$ denote the $i$th component of $\bm{x}$. Suppose that $G$ satisfies the following first-order sentence: \[\forall \bm{x} \ \exists \bm{y} \ \Sigma(\bm{x},\bm{y},\bm{a})=1 \ \wedge \ \Psi(\bm{x},\bm{y},\bm{a})\neq 1.\]Then, for every $p$-tuple $\bm{\sigma}=(\sigma_1,\ldots,\sigma_p)\in \mathrm{Aut}_G(E(G))^p$, there exists a morphism \[\pi_{\bm{\sigma}} : G_{\Sigma}\rightarrow G_{\bm{\sigma}}=G\ast_{E(G)}\left\langle  \bm{x}, E(G) \ \vert \ \mathrm{ad}(x_i)_{\vert E(G)}={\sigma_i}, \ \forall i\in \llbracket 1,p\rrbracket\right\rangle,\]called a \emph{formal solution}, enjoying the following properties:
\begin{itemize}
    \item[$\bullet$] $\pi_{\bm{\sigma}}(\bm{x})=\bm{x}$,
    \item[$\bullet$] $\pi_{\bm{\sigma}}(\bm{a})=\bm{a}$,
    \item[$\bullet$] $\Psi(\bm{x},\pi_{\bm{\sigma}}(\bm{y}),\bm{a})\neq 1$.
\end{itemize}
Moreover, the image of $\pi_{\bm{\sigma}}$ is a subgroup of $G_{\bm{\sigma}}$ of the form \[\left\langle \bm{g},\bm{a}\right\rangle\ast_{E(G)}\left\langle \bm{x}, E(G) \ \vert \ \mathrm{ad}(x_i)_{\vert E(G)}={\sigma_i}, \ \forall i\in \llbracket 1,p\rrbracket\right\rangle\] for some tuple $\bm{g}$ of elements of $G$.\end{te}

\begin{rque}Note that $G_{\bm{\sigma}}$ is isomorphic to the group $G\ast_{E(G)}(F_p\times E(G))$ obtained from $G$ by adding $p$ stable letters commuting with $E(G)$. Indeed, by definition of $\mathrm{Aut}_G(E(G))$, for every $1\leq i\leq p$, there exists an element $g_i\in G$ such that $\mathrm{ad}(x_i)_{\vert E(G)}=\mathrm{ad}(g_i)_{\vert E(G)}$. It follows that $t_i=x_ig_i^{-1}$ commutes with $E(G)$.\end{rque}

\begin{rque}This theorem captures the spirit of Merzlyakov's original theorem, in the following sense: let $\bm{g}=(g_1,\ldots,g_p)$ be a tuple of elements of $G$, of the same arity as $\bm{x}$. Let $\bm{\sigma}=(\mathrm{ad}(g_1)_{\vert E(G)},\ldots,\mathrm{ad}(g_p)_{\vert E(G)})$, and let $\varphi : G_{\bm{\sigma}}\twoheadrightarrow G$ be the retraction that maps $x_i$ to $g_i$ and coincides with the identity on $G$. The homomorphism $\varphi\circ \pi_{\bm{\sigma}}$ from $G_{\Sigma}$ to $G$ maps $\bm{x}$ to $\bm{g}$. Denote by $\bm{h}$ the image of $\bm{y}$ under this homomorphism. The equalities $\Sigma(\bm{g},\bm{h},\bm{a})=1$ hold in $G$. In other words, just as with Merzlyakov's original theorem, the theorem above gives a mechanism for associating to every tuple $\bm{g}\in G^p$ another tuple $\bm{h}$ of the same arity as $\bm{y}$ such that the equalities $\Sigma(\bm{g},\bm{h},\bm{a})=1$ hold in $G$. However, note that the image of $\Psi(\bm{x},\pi_{\bm{\sigma}}(\bm{y}),\bm{a})\neq 1$ by $\varphi$ may be trivial.\end{rque}

\begin{ex}Let $G=\langle g_1,g_2,a \ \vert \ a^3=1, \ g_1ag_1^{-1}=a, \ g_2ag_2^{-1}=a^2\rangle\simeq \mathbb{Z}/3\mathbb{Z}\rtimes F_2$. Let $\sigma$ be the automorphism of $\langle a\rangle$ that maps $a$ to $a^2$, and let us consider the following first-order sentence, which is clearly satisfied by $G$: $\forall x \ \exists y \ ([x,a]=[y,a]) \ \wedge \ (x\neq y)$. By definition, one has:
\begin{itemize}
    \item[$\bullet$] $G_{\Sigma}=\langle x,y,a \ \vert \  a^3=1, \ [x,a]=[y,a]\rangle$,
    \item[$\bullet$] $G_{\mathrm{id}}=G\ast_{\langle a\rangle}\langle x,a \ \vert \ xax^{-1}=a\rangle$,
    \item[$\bullet$] $G_{\sigma}=G\ast_{\langle a\rangle}\langle x,a \ \vert \ xax^{-1}=a^2\rangle$.
\end{itemize}
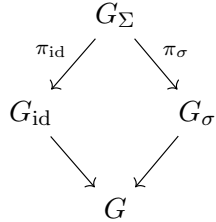
\begin{center}
\begin{tikzcd}[column sep=tiny]& G_{\Sigma}\ar[dr, "\pi_{\sigma}"]\ar["\pi_{\mathrm{id}}",swap,dl]&&[1.5em] \\G_{\mathrm{id}}\ar[dr]&& G_{\sigma}\ar[dl]& \\& G&&\end{tikzcd}
\end{center}
The morphism $\pi_{\mathrm{id}}$ can be defined by $\pi_{\mathrm{id}}(x)=x$, $\pi_{\mathrm{id}}(a)=a$ and $\pi_{\mathrm{id}}(y)=g_1$. The morphism $\pi_{\sigma}$ can be defined by $\pi_{\sigma}(x)=x$, $\pi_{\sigma}(a)=a$ and $\pi_{\sigma}(y)=g_2$. Note that $\pi_{\mathrm{id}}(G_{\Sigma})$ and $\pi_{\sigma}(G_{\Sigma})$ are both isomorphic to $G$.
\end{ex}



The structure of the proof of Theorem \ref{th0bis} given in this paper, which is quite different from Merzlyakov's original combinatorial proof, is inspired from Sela's geometric proof of Merzlyakov's theorem (we refer to \cite{Sel03}). Nevertheless, both proofs rely crucially on small cancellation theory (combinatorial in one case, geometric in the other case). We also took inspiration from Sacerdote's paper \cite{Sac73}. 

\smallskip

\subsection*{An outline of the proof of Theorem \ref{th1}}\label{shortening_strength}

In order to illustrate the main ideas and to highlight the difficulties encountered, we sketch a proof of Theorem \ref{th1} in the particular case where the maximal normal finite subgroup $E(G)$ is trivial. 

\smallskip

Suppose that $G$ satisfies a first-order sentence \[\theta:\forall \bm{x} \ \exists \bm{y} \ \Sigma(\bm{x},\bm{y})=1\wedge \Psi(\bm{x},\bm{y})\neq 1.\]Let $\Gamma=G\ast\langle t\rangle\simeq G\ast\mathbb{Z}$. Observe that the following two assertions are equivalent, where $p$ denotes the arity of $\bm{x}$.
\begin{itemize}
\item[$\bullet$]$\Gamma$ satisfies the sentence $\theta$.
\item[$\bullet$]For every $\bm{\gamma}\in\Gamma^p$, there exists a retraction $r$ from $\Gamma_{\Sigma,\bm{\gamma}}=\langle \Gamma,\bm{y} \ \vert \ \Sigma(\bm{\gamma},\bm{y})=1\rangle$ onto $\Gamma$ such that no component of $\Psi(\bm{\gamma},\bm{y})$ is killed by $r$, i.e.\ the inequations remain valid in the image of $r$.\end{itemize}

\smallskip

In order to prove that $\Gamma$ satisfies the sentence $\theta$, we will construct such a retraction $r : \Gamma_{\Sigma,\bm{\gamma}}\twoheadrightarrow\Gamma$, for any $\bm{\gamma}\in\Gamma^p$. The very first step of the construction of this retraction relies on the existence of a quasi-convex free subgroup $F(a,b)\subset G$ (see \cite[Theorem 6.14]{DGO17} combined with \cite[Lemma 3.1]{AMS13}). From a sequence of elements $(w_n(a,b))_{n\in\mathbb{N}}\in F(a,b)^{\mathbb{N}}$ satsifying certain small cancellation conditions in the free group $F(a,b)$, one defines a \emph{test sequence} $(\varphi_n : \Gamma\twoheadrightarrow G)_{n\in \mathbb{N}}$ by ${\varphi_n}_{\vert G}=\mathrm{id}_G$ and $\varphi_n(t)=w_n(a,b)$. Since, by assumption, the sentence $\theta$ is true in the group $G$, each morphism $\varphi_n$ extends to a morphism $\psi_n : \Gamma_{\Sigma,\bm{\gamma}}\twoheadrightarrow G$ mapping $\bm{y}$ to a tuple $\bm{g}_n$ such that $\Sigma(\bm{\gamma},\bm{g}_n)=1$ and $\Psi(\bm{\gamma},\bm{g}_n)\neq 1$. 

\smallskip

The fact that $F(a,b)$ is quasi-isometrically embedded into $G$ enables us to prove that the sequence of elements $(\psi_n(t)=w_n(a,b))_{n\in\mathbb{N}}$ satisfies nice geometric conditions in $G$, which, in some sense, encapsulate the first-order sentence $\exists \bm{y} \ \Sigma(\bm{\gamma},\bm{y})=1\wedge \Psi(\bm{\gamma},\bm{y})\neq 1.$

\smallskip

Then, using the non-elementary acylindrical action of $G$ on a hyperbolic space, one can show that the sequence $(\psi_n)_{n\in\mathbb{N}}$ converges to an action of $\Gamma_{\Sigma,\bm{\gamma}}$ on a limiting real tree in the Gromov-Hausdorff topology, via the well-known Bestvina-Paulin method. This tree $T$ comes equipped with an isometric action of a quotient $L$ of $\Gamma$, called a divergent limit group. The action of $L$ on this tree can be analysed using the \emph{Rips machine}, adapted by Groves and Hull in \cite{GH19} to the setting of acylindrically hyperbolic groups, which converts the action $L\curvearrowright T$ into an action of $L$ on a simplicial tree, i.e.\ a splitting of $L$. The properties of the test sequence $(\psi_n)_{n\in\mathbb{N}}$ are reflected in this splitting of $L$, and the rest of the proof consists in constructing a retraction from $L$ onto $\Gamma$, using this splitting.

\smallskip

One key ingredient in this construction is a generalisation of Sela's \emph{shortening argument}. Due to the lack of \emph{equational Noetherianity} of acylindrically hyperbolic groups (see below), the sequence $(\psi_n : \Gamma_{\Sigma,\bm{\gamma}}\rightarrow G)_{n\mathbb{N}}$ does not factor through the quotient epimorphism $\psi_{\infty} : \Gamma_{\Sigma,\bm{\gamma}}\twoheadrightarrow L$ in general, which is the source of difficulties. Note that our version of the shortening argument is slightly different from the one proved by Groves and Hull, see \cite[Theorem 5.29 and Lemma 6.5]{GH19} and Section \ref{section_LG_SA} (Remark \ref{single_ap}) for further details.

\smallskip

Recall that a group is said to be equationally Noetherian if the set of solutions of any system of equations in finitely many variables coincides with the set of solutions of a certain finite subsystem of this system. As a consequence of the Hilbert Basis Theorem, linear hyperbolic groups are equationally Noetherian, and it was proved by Sela in \cite{Sel09} (torsion-free case) and by Reinfeldt and Weidmann in \cite{RW14} (general case) that the linearity assumption can be dropped. Equational Noetherianity has proved extremely useful in the study of the first-order theory of hyperbolic groups, notably because \emph{limit groups} over hyperbolic groups are not finitely presentable in general, which constrains us to deal with infinite systems of relations. Unfortunately, since equational Noetherianity is inherited by subgroups, and since for instance $H\ast\mathbb{Z}$ is acylindrically hyperbolic for any non-trivial group $H$, acylindrically hyperbolic groups are typically not equationally Noetherian. This is a major obstacle to constructing the desired retraction from $\Gamma_{\Sigma,\bm{\gamma}}$ onto $\Gamma$. 

\smallskip

We overcome this problem by introducing a method of approximating, in a precise sense, limit groups over acylindrically hyperbolic groups by finitely presented groups relative to a subgroup. The idea of approximating limit groups by finitely presented groups already appears, in a slightly different form, in \cite[Theorem 3.2]{Sel01} (see also \cite{Gro05}, \cite[Lemma 6.1]{RW14} and \cite[Lemma 6.3]{GH19}). More precisely, in the present case, there exists a quotient $A$ of $\Gamma_{\Sigma,\bm{\gamma}}$, called an \emph{approximation} of $L$, which is finitely presented relative to $G$, maps onto $L$, and has a splitting that mimics the splitting of $L$ outputted by the Rips machine. Since $A$ is finitely presented relative to $G$, each morphism $\psi_n : \Gamma_{\Sigma,\bm{\gamma}}\rightarrow G $ factors through the quotient epimorphism $\Gamma_{\Sigma,\bm{\gamma}}\twoheadrightarrow A$, for $n$ sufficiently large, as shown in the commutative diagram below.

\smallskip

\begin{align*}
\xymatrix{
\Gamma_{\Sigma,\bm{\gamma}} \ar@{->>}[drr] \ar[rrrr]^{\psi_n} \ar@{->>}[ddrr]_{\psi_\infty} && && G \\
&& A \ar[rru]_{\rho_n} \ar@{->>}[d] \\
&& L}
\end{align*}
We prove that the shortening argument applies to the resulting sequence $(\rho_n : A \rightarrow G)_{n\in\mathbb{N}}$, together with the splitting of $A$ mimicking the splitting of $L$. We refer the reader to Section \ref{section_LG_SA} for further details. Note, however, that the sequence $(\rho_n : A \rightarrow G)_{n\in\mathbb{N}}$ is not discriminating as soon as $L$ is a strict quotient of $A$; in other words, the stable kernel of this sequence is not trivial, and $A$ is not a $G$-limit group \emph{a priori}, which leads to new technical difficulties.

\smallskip

A further bad consequence of the lack of equational Noetherianity is that there is no descending chain condition for limit groups over acylindrically hyperbolic groups in general. This is another obstacle to the construction of the retraction. Fortunately, it is proved in \cite{GH19} (Convention 4.6 and Lemma 4.7) that the size of the finite edge groups appearing in the splitting of $L$ outputted by the Rips machine is bounded from above by a constant that depends only on the acylindrical action of $G$ on a fixed hyperbolic space, and on the hyperbolicity constant of this space. This result is remarkable since the order of a finite subgroup of $G$ is not bounded in general, and allows us to appeal to accessibility results and to prove that our construction, which is iterative, eventually terminates.

\smallskip

\subsection*{Acknowledgments}The first author thanks Denis Osin for interesting discussions. The second author would like to thank Martin Bridson and Zlil Sela for their guidance and support. 

\tableofcontents

\section{Preliminaries}
\subsection{Conventions}

For a group $G$ generated by a (not necessarily finite) set $S$, the word length $\abs{g}_S$ of an element $g \in G$ is the length of the shortest word in $S \cup S^{-1}$ representing $g$ in $G$. We usually denote the Cayley graph of $G$ (with respect to $S$) by $X$ and regard $X$ as a metric space by setting $d(g,h)=\abs{g^{-1}h}_S$. Throughout this paper, all groups acting on metric spaces act by isometries, and all metric spaces are geodesic.

\subsection{Equations over groups}

An \emph{equation} in variables $\bm{x}=(x_1,\ldots,x_p)$ is an equality $w(\bm{x})=1$ for $w(\bm{x}) \in F(\bm{x})$ (where $F(\bm{x})$ is the free group on $\bm{x}$); an equation \emph{over} a group $G$ in variables $\bm{x}$ is an equality of the form $w(\bm{x},\bm{a})=1$ where $w(\bm{x},\bm{a})\in F(\bm{x})*G$ and $\bm{a}$ is a tuple of elements from $G$. A \emph{solution} to the equation $w(\bm{x},\bm{a})=1$ over a group $G$ consists of a tuple $\bm{g}\in G^p$ for which the element $w(\bm{g},\bm{a})$, obtained by replacing every occurrence of $x_i^{\pm 1}$ with $g_i^{\pm 1}$, is trivial. Given a subset $\Sigma(\bm{x},\bm{a})=\lbrace w_i(\bm{x},\bm{a})\rbrace_{i\in I}\subset F(\bm{x})*G$, we refer to the conjunction $\bigwedge_{i\in I}w_i(\bm{x},\bm{a})=1$ as a \emph{system} of equations. We abbreviate and write $\Sigma(\bm{x},\bm{a})=1$, and say that a tuple $\bm{g}\in G^p$ is a solution to $\Sigma(\bm{x},\bm{a})=1$ if for every $w_i(\bm{x},\bm{a})\in \Sigma(\bm{x},\bm{a})$, one has $w_i(\bm{g},\bm{a})=1$. 

\smallskip

Similarly, an \emph{inequation} in variables $\bm{x}=(x_1,\ldots,x_p)$ is an inequality $w(\bm{x})\ne 1$ for $w(\bm{x}) \in F(\bm{x})$ (and an inequation over a group $G$ is an inequality $w(\bm{x},\bm{a}) \ne 1$ where $w(\bm{x},\bm{a})\in F(\bm{x})*G$ and $\bm{a}$ is a tuple of elements from $G$). Just like systems of equations, systems of inequations are conjunctions of inequations; we say that a tuple $\bm{g}\in G^p$ satisfies the system of inequations $\Phi(\bm{x},\bm{a}) \ne 1$ in $G$ if for every $w_i(\bm{x},\bm{a})\in \Phi(\bm{x},\bm{a})$, $w_i(\bm{g},\bm{a})\ne 1$ holds.

\smallskip

Note that there is a one-to-one correspondence between the set of solutions to the system of equations $\Sigma(\bm{x},\bm{a})=1$ over a group $G$ and the set of homomorphisms 
\begin{align*}
    \varphi: G_\Sigma = \langle \bm{x}, \bm{a} \ \vert \ R(\bm{a}) \cup \Sigma(\bm{x},\bm{a}) \rangle \rightarrow G
\end{align*}
(where $R(\bm{a})$ is a set of relations for which $\langle \bm{a} \ \vert \ R(\bm{a})\rangle$ is a presentation of the subgroup of $G$ generated by $\bm{a}$). If $\bm{g}$ is a solution to $\Sigma(\bm{x},\bm{a})=1$, there exists a homomorphism $\varphi:G_\Sigma \rightarrow G$ mapping $\bm{x}$ to $\bm{g}$ and $\bm{a}$ to $\bm{a}$; on the other hand, given such a homomorphism $\varphi$, the tuple $\varphi(\bm{x})\in G^p$ is a solution to $\Sigma(\bm{x},\bm{a})=1$ over $G$. In addition, a solution $\bm{g}$ to the system of equations $\Sigma(\bm{x},\bm{a})=1$ satisfies the system of inequations $\Phi(\bm{x},\bm{a})\ne 1$ if and only if there exists a homomorphism $\varphi:G_\Sigma \rightarrow G$ which maps $\bm{x}$ to $\bm{g}$ and $\bm{a}$ to $\bm{a}$, and such that for every $w_i(\bm{x},\bm{a})\in \Phi(\bm{x},\bm{a})$, $\varphi(w_i(\bm{x},\bm{a}))\ne 1$. Thus we regard the study of equations (and their solutions) over $G$, as the study of homomorphisms from the group $G_\Sigma$ to $G$. 

\subsection{Acylindrically hyperbolic groups}\label{prelim_acyl}
The aim of this subsection is to familiarize the reader, in a rather shallow manner, with acylindrically hyperbolic groups. 

\begin{de}
A geodesic metric space $(X,d)$ is called \emph{$\delta$-hyperbolic} if every geodesic triangle $\Delta=(x,y,z)$ in $X$ is \emph{$\delta$-slim}: every side of $\Delta$ is contained in the closed $\delta$-neighborhood of the union of the two other edges. The space $(X,d)$ is called \emph{hyperbolic} if it is $\delta$-hyperbolic for some $\delta$.
\end{de}

Recall that if a group $G$ acts on a hyperbolic space $(X,d)$ by isometries, an isometry $g\in G$ is called \emph{elliptic} if some (equivalently, any) orbit of $g$ is bounded. An isometry $g \in G$ is called \emph{hyperbolic} if for some (equivalently, any) $x \in X$, the map $\mathbb{Z}\rightarrow X$ defined via $m \mapsto g^m x$ is a quasi-isometric embedding; we call the image of such a quasi-isometric embedding a \emph{quasi-geodesic axis} of $g$ (or in other words, a quasi-geodesic axis of $g$ is an orbit of $g$ in $X$). The \emph{Gromov boundary} of $X$, denoted by $\partial X$, is defined as the collection of equivalence classes of quasi-isometric embeddings $\mathbb{N} \rightarrow X$ (where two embeddings are equivalent if their images lie at bounded Hausdorff distance from one another). A hyperbolic element $g \in G$ has therefore exactly two limit points $g^{+\infty}$ and $g^{-\infty}$ on $\partial X$, represented by the quasi-isometric embeddings $n \mapsto g^n x$ and $n \mapsto g^{-n} x$ (for some $x \in X$) respectively. Two hyperbolic elements $g$ and $h$ are called \emph{independent} if $\{g^{\pm \infty}\} \cap \{h^{\pm \infty}\}=\varnothing$. We call the action of $G$ on $X$ \emph{non-elementary} if there are two (or equivalently, infinitely many) independent hyperbolic elements in $G$. 

\smallskip

The notion of an \emph{acylindrical} group action on a metric space was first introduced by Bowditch in \cite{Bow08}, and was inspired by Sela's notion of a \emph{$k$-acylindrical} group action on a tree: a group action on a tree is called $k$-acylindrical if it contains no arcs of length greater than $k$ which are fixed by a non-trivial element of the group (and hence, the tree contains no "cylinders"). This notion was later generalized by imposing a bound on the cardinality of a subgroup which fixes an arc of length greater than $k$ in the tree, and coarsified in the following manner.

\begin{de}\label{acylindricity_constants}
    A group action on a metric space $G \curvearrowright (X,d)$ is called \emph{acylindrical} if for every $\varepsilon \ge 0$ there exist $N>0$ and $R>0$ such that for every $x,y \in X$ satisfying $d(x,y) \ge R$,
    \begin{align*}
        \abs{\lbrace g \in G \ \vert \ d(x,gx)\le \varepsilon \text{ and } d(y,gy)\le \varepsilon \rbrace}\le N.
    \end{align*}
\end{de}

The following lemma which appears in \cite{DGO17} will play an important role throughout this paper (note that in \cite{DGO17}, this result is stated for hyperbolic (also called loxodromic) WPD elements, but it turns out that all hyperbolic elements are automatically WPD when the action of the group is acylindrical).

\begin{lemme}
    \cite[Lemma 6.5, Corollary 6.6]{DGO17}\label{lemmefin} Let $G$ be a group acting acylindrically on a hyperbolic space $X$ and let $g \in G$ be a hyperbolic element. Then $g$ is contained in the unique maximal and virtually cyclic subgroup $\Lambda (g)$ which consists of all $h \in H$ for which the Hausdorff distance between $\ell$ and $h\ell$ is finite, where $\ell$ is some quasi-geodesic axis of $g$ in $X$. In addition, the following are equivalent for any $h \in G$:
    \begin{enumerate}
        \item $h \in \Lambda (g)$.
        \item $h^{-1}g^mh=g^k$ for some $0 \ne m,k \in \mathbb{Z}$.
        \item $h^{-1}g^nh=g^{\pm n}$ for some $n \in \mathbb{N}^{\ast}$.
    \end{enumerate}
    In addition, there exists $r \in \mathbb{N}$ such that the centralizer of $g^r$ is given by
    \begin{align*}
        C_G(g^r)=\lbrace h \in G \ \vert \ \exists n \in \mathbb{N}, \ h^{-1}g^nh=g^n \rbrace \subset \Lambda (g).
    \end{align*}
\end{lemme}

Suppose now that a group $G$ acts acylindrically on a hyperbolic space $(X,d)$. The action of $G$ on $X$ falls into exactly one of three categories \cite[Theorem 1.1]{Osi16}:
\begin{enumerate}
    \item the action of $G$ is elliptic, that is every $G$-orbit is bounded.
    \item $G$ is virtually cyclic and contains a hyperbolic element.
    \item $G$ contains two (equivalently, infinitely many) pairwise independent hyperbolic elements.
\end{enumerate}

If an action falls into category (1) or (2) above, it is termed \emph{elementary}.

\begin{de}
    A group $G$ is said to be \emph{acylindrically hyperbolic} if it admits a non-elementary and acylindrical action on a hyperbolic space.
\end{de}

As a matter of fact, we can always choose the hyperbolic space on which $G$ acts to be a simplicial graph, as shown by the following result.

\begin{te}
    \cite[Theorem 1.2]{Osi16}\label{ahcay} If $G$ is acylindrically hyperbolic then there exists a (not necessarily finite) generating set $S$ of $G$ such that the Cayley graph $X$ of $G$ with respect to $S$ is hyperbolic and such that the natural action of $G$ on $X$ is non-elementary and acylindrical.
\end{te}

\begin{rque}
    If the group $G$ is not hyperbolic, then the generating set $S$ mentioned in Theorem \ref{ahcay} above is necessarily infinite.
\end{rque}

\section{Limit groups, approximations and the shortening argument}\label{section_LG_SA}
In this section, we define \emph{limit groups} over acylindrically hyperbolic groups and discuss some of their prominent properties following the recent work of Groves and Hull \cite{GH19}. We focus our interest on \emph{divergent} limit groups (see Definition \ref{limit}) as these come armed with a limiting action on a real tree. Under some conditions, a divergent limit group splits as a \emph{graph of actions} (see Definition \ref{graphofactions}); this splitting is famously known as the output of the \emph{Rips machine}, which was introduced in unpublished work of Rips in the early 1990's. 

\smallskip

The lack of equational Noetherianity in acylindrically hyperbolic groups imposes a great obstacle to exploiting Sela's \emph{shortening argument} (see Subsection \ref{shortening}) in our setting. Hence, we conclude this section by introducing a method of approximating limit groups by finitely presented groups (relative to a subgroup) in a sense that captures the structure of the aforementioned splitting, and which will allow us to establish a generalisation of Sela's shortening argument to acylindrically hyperbolic groups. Note that a similar method of approximating limit groups appears in \cite{GH19}, however the construction of the approximation is different, and the version of the shortening argument we prove is slightly different from the one proved in \cite{GH19} (see Remark \ref{single_ap} for further details).

\smallskip

Throughout this section, assume that $(X,d)$ is a $\delta$-hyperbolic simplicial graph on which $G$ acts acylindrically and non-elementarily (note that, by Theorem \ref{ahcay}, this assumption is not restrictive). We also stick to the language of \emph{ultrafilters} as in \cite{GH19} and as is common in non-standard analysis; this enables us to phrase statements with relative ease, rather than often passing to subsequences. Recall that a \emph{non-principal ultrafilter} is a finitely additive probability measure $\omega:2^{\mathbb{N}}\rightarrow\{0,1\}$ satisfying $\omega(F)=0$ for every finite $F\subset \mathbb{N}$. For a statement $P$ depending on an index $n \in \mathbb{N}$, we say that $P$ holds $\omega$\emph{-almost-surely} if 
\begin{align*}
    \omega(\lbrace n\in \mathbb{N} \ \vert \ P \text{ holds for }n\rbrace)=1.
\end{align*}
The $\omega$\emph{-limit} of a sequence $(x_n)_{n\in\mathbb{N}}$ in $\mathbb{R}$ is $x \in \mathbb{R}$ if for every $\varepsilon>0$,
\begin{align*}
    \omega(\lbrace n\in \mathbb{N} \ \vert \ \abs{x-x_n}<\varepsilon \rbrace)=1.
\end{align*}
In this case we denote $\lim_\omega(x_n)=x$. We say that $\lim_\omega(x_n)=\infty$ if $\omega(\lbrace n\in \mathbb{N} \ \vert \ x_n>N \rbrace )=1$ holds for every $N \in \mathbb{N}$. Every sequence of real numbers has a unique $\omega$-limit in $\mathbb{R}\cup\{\infty\}$.

\subsection{Limit groups over acylindrically hyperbolic groups}\label{limit_groups_over_AH}
We define limit groups over acylindrically hyperbolic groups as in the standard case over free groups; for a more detailed description of the construction, and for additional properties of such limit groups, we refer the reader to \cite{GH19}.
\begin{de}
    \label{limit}
    Let $H$ be a finitely generated group, and let $(\varphi_n)_{n\in\mathbb{N}}$ be a sequence in $\mathrm{Hom}(H,G)^{\mathbb{N}}$. The \emph{stable} kernel of $(\varphi_n)_{n\in\mathbb{N}}$ (with respect to $\omega$) is
    \begin{align*}
        \underleftarrow{\ker}_\omega((\varphi_n)_{n\in\mathbb{N}})=\lbrace g \in H \ \vert \ g \in \ker (\varphi_n) \, \omega \text{-almost-surely} \rbrace.
    \end{align*}
    Fixing a finite generating set $S$ of $H$, we associate a \emph{scaling factor} to every element in the sequence $(\varphi_n)_{n\in\mathbb{N}}$, defined by
    \begin{align*}
        \norm{\varphi_n}=\inf_{y \in X} \max_{s \in S} d(y,\varphi_n(s)y).
    \end{align*}
\end{de}
Using the notion of a stable kernel of a sequence of homomorphisms, we can define limit groups over acylindrically hyperbolic groups.
\begin{de}
    Keeping the notation from Definition \ref{limit}, a $G$-\emph{limit group} is a group of the form $L=H / \underleftarrow{\ker}_\omega((\varphi_n)_{n\in\mathbb{N}})$. We call the limit group $L$ \emph{divergent} if $\lim_\omega(\norm{\varphi_n})=\infty$. 
\end{de}

The sequence $(\varphi_n)_{n\in\mathbb{N}}$ is called the \emph{defining sequence} of homomorphisms for $L$, and we denote by $\varphi_{\infty}:H\twoheadrightarrow L$ the natural quotient map and refer to $\varphi_{\infty}$ as the \emph{limit map} associated with the sequence $(\varphi_n)_{n\in\mathbb{N}}$. As previously mentioned, every divergent limit group $L$ comes equipped with a non-trivial and minimal action on a real tree; the construction of this real tree is commonly referred to as the Bestvina-Paulin method. We briefly explain how the real tree on which a divergent limit group acts is constructed (a detailed proof appears in \cite[Theorem 4.4]{GH19}). Consider the sequence $(X_n,d_n,o_n)_{n\in\mathbb{N}}$ of pointed simplicial graphs, where $X_n=X$ for every $n$, $d_n=\frac{1}{\norm{\varphi_n}} \cdot d$ and $o_n$ is a point in $X_n$ chosen to satisfy
\begin{align*}
    \max_{s\in S}d_n(o_n,\varphi_n(s)o_n) \leq \norm{\varphi_n}+\frac{1}{n}.
\end{align*}
We can always choose a point $o_n$ which satisfies this inequality since the metric on $X$ is discrete. The \emph{ultra-limit} $\left(\prod_{n \in \mathbb{N}} X_n\right) / \omega$ of the sequence $(X_n,d_n,o_n)_{n\in\mathbb{N}}$ is given by
\begin{align*}
    \left(\prod_{n \in \mathbb{N}} X_n\right) / \omega = \frac {\lbrace  (x_n)_{n\in\mathbb{N}} \in \prod_{n \in \mathbb{N}} X_n   \ \vert \ \lim_\omega (d_n(o_n,x_n))<\infty \rbrace}{\sim_\omega}
\end{align*}
where the equivalence relation $\sim_\omega$ on $\prod_{n \in \mathbb{N}}X_n$ is defined by setting $(x_n)_{n\in\mathbb{N}}\sim_\omega(y_n)_{n\in\mathbb{N}}$ if and only if $\lim_\omega(d_n(x_n,y_n))=0$. The ultra-limit $\left(\prod_{n \in \mathbb{N}}X_n\right) / \omega$ is equipped with a complete metric $d_\omega$ defined by $d_\omega((x_n)_{n\in\mathbb{N}},(y_n)_{n\in\mathbb{N}})=\lim_\omega d_n(x_n,y_n)$; furthermore, note that in this case every $X_n$ is $\left(\delta_n=\frac{1}{\norm{\varphi_n}}\cdot \delta\right)$-hyperbolic and the ultra-limit is $(\lim_\omega(\delta_n)=0)$-hyperbolic or in other words a real tree. Note that every homomorphism $\varphi_n$ endows $H$ with an action on $X$ (and $X_n$) by setting $hx=\varphi_n(h)x$ for every $h \in H$. These actions enable us to define an action of $H$ on the ultra-limit $\left(\prod_{n \in \mathbb{N}} X_n\right) / \omega$ via $h(x_n)_{n\in\mathbb{N}}=(hx_n)_{n\in\mathbb{N}}$. \smallskip
Lastly, we choose a minimal and $H$-invariant subtree $T$ of $\left(\prod_{n \in \mathbb{N}} X_n\right) / \omega$. Since every element of $\underleftarrow{\ker}_\omega((\varphi_n)_{n\in\mathbb{N}})$ acts trivially on $T$, the action of $H$ on $T$ induces the desired action of the divergent limit group $L$ on the real tree $T$.

\subsection{Graphs of actions and the Rips machine}
Under certain conditions, a group acting on a real tree splits as a \emph{graph of actions}. This splitting endows the group with an action on a simplicial tree which is generally easier to understand than an action on a real tree. Groves and Hull proved in \cite{GH19} that divergent limit groups over acylindrically hyperbolic groups and their canonical actions on real trees satisfy the desired conditions which are required to invoke Guirardel's version of the Rips machine (see \cite{Gui08}). We present the relevant definitions and results from both works and assume that the reader is familiar with the standard terminology associated with the Rips machine, as in \cite{Gui08}.

\begin{de}
    \label{graphofactions}
    \cite[Definition~1.2]{Gui08}
    A \emph{graph of actions} $\mathcal{G}$ consists of:
    \begin{enumerate}
        \item an underlying graph of groups $\mathbb{A}=\left(A,(A_v)_{v \in \mathrm{V}(A)},(A_e)_{e \in \mathrm{E}(A)},(i_e)_{e \in \mathrm{E}(A)}\right)$,
        \item a collection of real trees $(T_v,d_v)_{v \in \mathrm{V}(A)}$ such that $A_v$ acts on $T_v$,
        \item a collection of points $(p_e \in T_{t(e)})_{e \in \mathrm{E}(A)}$ such that every $p_e$ is fixed by $i_e(A_e)$, called \emph{attaching points},
        \item a function $\ell:\mathrm{E}(A)\rightarrow \mathbb{R}_{\ge 0}$ assigning \emph{lengths} to the edges of $A$, and such that $\ell(e)=\ell(\overline{e})$ for every $e \in \mathrm{E}(A)$.
    \end{enumerate}
    We usually present the information above as a tuple and write
    \begin{align*}
        \mathcal{G}=\mathcal{G}(\mathbb{A})=\left(\mathbb{A},(T_v)_{v \in \mathrm{V}(A)},(p_e)_{e \in \mathrm{E}(A)},\ell \right).
    \end{align*}
\end{de}

A graph of actions $\mathcal{G}(\mathbb{A})$ enables one to canonically construct a real tree $T_\mathcal{G}$ on which $G=\pi_1(\mathbb{A})$ acts: replace each vertex $\tilde{v}$ of the Bass-Serre tree $T_A$ corresponding to $\mathbb{A}$ by a copy of $T_v$ (where $v$ is the image of $\tilde{v}$ under the quotient map $q:T_A\rightarrow G \backslash T_A = A$), and replace any edge $\tilde{e}$ of $T_A$ by a segment of length $\ell(e)$ (where $e=q(\tilde{e})$). We also ask that if $t(\tilde{e})=\tilde{v}$ in $T_A$ then $t(\tilde{e})=\tilde{p}_e$ in $T_\mathcal{G}$, that is we attach the tree $T_v$ via the attaching point $p_e$. The action of $\pi_1(\mathbb{A})$ on $T_A$ extends naturally to an action $\pi_1(\mathbb{A})\curvearrowright T_\mathcal{G}$.  We next define notions of stability concerning with group actions on real trees which will allow us to describe the output of the Rips machine.

\begin{de}\label{recall_definition_stable}
    Suppose that $L$ is a group acting on a real tree $T$. 
    \begin{enumerate}
        \item A subtree $T' \subset T$ is called \emph{stable} if for every non-degenerate subtree $T''\subset T'$, $\mathrm{Stab}_L(T')=\mathrm{Stab}_L(T'')$. Otherwise, $T'$ is called \emph{unstable}. An action on a real tree is \emph{stable} if any non-degenerate arc contains a non-degenerate stable subarc.
        \item The action $L\curvearrowright T$ is said to satisfy the \emph{ascending chain condition} if for any sequence of nested arcs $I_1 \supset I_2 \supset \cdots$ in $T$ whose lengths tend to $0$, the corresponding sequence of stabilizers $\mathrm{Stab}_L(I_1)\subset \mathrm{Stab}_L(I_2)\subset \cdots$ eventually stabilizes.
    \end{enumerate}
\end{de}

We are now ready to state a relative version of the Rips machine which appears in \cite{Gui08}.

\begin{te}
\label{ripsmachine}
    \cite[Main Theorem]{Gui08}
    Let $L$ be a group acting minimally and non-trivially on a real tree $T$ by isometries. Let $U$ be a subgroup of $L$ such that $L$ is finitely generated over $U$ and such that $U$ fixes a point in a real tree $T$ on which $L$ acts. Assume in addition that the action of $L$ on $T$ satisfies the ascending chain condition, and that for any unstable arc $I \subset T$,
    \begin{enumerate}
        \item $\mathrm{Stab}_L(I)$ is finitely generated, and
        \item $\mathrm{Stab}_L(I)$ is not a proper subgroup of any conjugate of itself.
    \end{enumerate}
    Then one of the following holds.
    \begin{enumerate}
        \item $L$ splits over the stabilizer of an unstable arc and $U$ is contained in one of the factors.
        \item $L$ splits over the stabilizer $N$ of an infinite tripod and $U$ is contained in one of the factors, and the normalizer of $N$ contains a non-abelian free group generated by two hyperbolic elements whose axes do not intersect.
        \item The action $L \curvearrowright T$ decomposes as a graph of actions $\mathbb{R}_L$ where each vertex action is
        \begin{enumerate}
            \item either \emph{simplicial}: a simplicial action on a simplicial tree,
            \item of \emph{Seifert-type}: the action of $L_v$ has kernel $N_v$ and the faithful action of $L_v/N_v$ is dual to an arational measured foliation on a compact $2$-orbifold with boundary,
            \item or \emph{axial}: $T_v$ is a line, and the image of $L_v$ in $\mathrm{Isom}(T_v)$ is a finitely generated group acting with dense orbits on $T_v$.
        \end{enumerate}
    \end{enumerate}
\end{te}

We are interested in decompositions of divergent limit groups over acylindrically hyperbolic groups; the following lemma that appears in \cite{GH19}, also known as the stability lemma, implies that such limit groups indeed satisfy the stability conditions required for applying Theorem \ref{ripsmachine}. Recall that $L$ is a divergent limit group with defining sequence of homomorphisms $(\varphi_n)_{n\in\mathbb{N}}\in \mathrm{Hom}(H,G)^{\mathbb{N}}$. In the following lemma, $\delta$ denotes the hyperbolicity constant of a hyperbolic space on which $G$ acts acylindrically and non-elementarily, and $N$ and $R$ denote the acylindricity constants appearing in Definition \ref{acylindricity_constants}.

\begin{lemme}
    \label{stabilitylemma}
    \cite[Lemma~4.7]{GH19}There is a constant $C$ depending only on $\delta$, $N$ and $R$ such that the action of $L$ on the real tree $T$ constructed in Subsection \ref{limit_groups_over_AH} satisfies the following conditions.
    \begin{enumerate}
        \item If $A\subset L$ stabilizes a non-trivial arc of $T$, or if $A$ preserves a line in $T$ and fixes its ends, then $A$ is an extension of an abelian group by a finite group of order $\le C$.
        \item The stabilizer of a tripod in $T$ is of order $\le C$.
        \item The stabilizer of an unstable arc $I\subset T$ is of order $\le C$.
        \item If $K\subset L$ is \emph{locally stably elliptic}, that is for every finitely generated subgroup $K' \subset K$, the action of $\varphi_n(\tilde{K}')$ (where $\tilde{K}'$ is a lift of $K'$ to $H$) on $X_n$ is elliptic $\omega$-almost-surely, then the order of $K$ is $\le C$.
    \end{enumerate}
\end{lemme}

\begin{co}
    The fact that stabilizers of unstable arcs are finite implies that the action of $L$ on $T$ satisfies the ascending chain condition and the rest of the conditions required for Theorem \ref{ripsmachine}. Hence if $L$ does not split non-trivially over a finite subgroup of order $\le C$, it must split as a graph of actions as in Theorem \ref{ripsmachine}. 
\end{co}

\subsection{Approximations of limit groups}
Recall that $G$ is a group that admits an acylindrical and non-elementary action on a $\delta$-hyperbolic simplicial graph $(X,d)$ and that $L$ is a limit group with defining sequence $(\varphi_n)_{n\in\mathbb{N}}\in\mathrm{Hom}(H,G)^{\mathbb{N}}$, where $H$ is a finitely generated group.

\begin{st}\label{StAs}In what follows, we assume that $H$ is finitely presented over an infinite finitely generated (but not necessarily finitely presented) subgroup $U\subset H$. We denote by $S$ a finite generating set of $H$. In addition, we suppose that $U$ acts elliptically on the limiting tree $T$, and that the restriction of the limit map $\varphi_\infty$ to $U$ is injective, which allows us to identify $U$ with its image under $\varphi_\infty$.\end{st}

\smallskip

In this subsection, we aim to prepare the grounds for proving a version of the shortening argument for acylindrically hyperbolic groups (slightly different from the version proved in \cite{GH19} by Groves and Hull). First, let us recall that the group $G$ is not equationally Noetherian in general. Therefore, the sequence $(\varphi_n : H \rightarrow G)_{n\in\mathbb{N}}$ defining the $G$-limit group $L$ does not factor through the quotient map $\varphi_{\infty}$ \emph{a priori}. This is a major obstacle to generalising the standard proof of the shortening argument, since we cannot rely on the splitting of $L$ as a graph of actions outputted by the Rips machine for shortening the morphisms of the sequence. Before we explain our approach for overcoming this difficulty (which is very similar to the approach taken in \cite[Lemma 6.3]{GH19}, coming from \cite[Theorem 3.2]{Sel01}), we begin by defining \emph{approximations} of limit groups.

\begin{de}\label{cor_approx}
    Given a finite set of relations $\mathcal{R}\subset \ker (\varphi_\infty)$, we define the \emph{$\mathcal{R}$-approximation} $A$ of $L$ as $A=H\big/ \langle \langle \mathcal{R} \rangle \rangle$. In general, we call a group $A$ obtained in this manner an $\emph{approximation}$ of $L$.
\end{de}

\begin{rque}
    Since the set $\mathcal{R}$ is a subset of $\ker(\varphi_\infty)$, the group $A$ acts on the limiting tree $T$. Both quotient maps $H\twoheadrightarrow A$ and $A\twoheadrightarrow L$ are equivariant with respect to the corresponding actions on $T$.
\end{rque}

Our motivations for introducing approximations of limit groups are the following.
\begin{enumerate}
    \item Since $H$ is finitely presented over $U$ (see Standing Assumption \ref{StAs}), every $\mathcal{R}$-approximation $A$ of $L$ is finitely presented over $U$. Therefore, the homomorphisms in the sequence $(\varphi_n : H \rightarrow G)_{n \in \mathbb{N}}$ factor $\omega$-almost-surely through the quotient map $H\twoheadrightarrow A$. This factorization will be crucial in our proof of the general version of the shortening argument.
    \item Suppose that $L$ admits a nice splitting as a graph of groups (see below for more details). We will see that, provided that the finite set of relations $\mathcal{R}\subset \ker (\varphi_\infty)$ is carefully chosen, the approximation $A$ admits a splitting that mimics the splitting of $L$, in a precise sense.
\end{enumerate}

In the proof of the shortening argument, as well as in the proof of Merzlyakov's theorem, we will consider different splittings of $L$: 
\begin{itemize}
    \item[$\bullet$]if $L$ splits non-trivially relative to $U$ over a finite subgroup of order $\leq C$  (the constant appearing in Lemma \ref{stabilitylemma}), a reduced JSJ splitting of $L$ relative to $U$ over finite subgroups of order $\le C$, denoted by $\mathbb{J}_L$ (see Proposition \ref{approx} and Corollary \ref{JSJ_app});
    \item[$\bullet$]if $L$ does not split non-trivially relative to $U$ over a finite subgroup of order $\leq C$, a splitting of $L$ as a graph of actions outputted by the Rips machine \ref{ripsmachine}, denoted by $\mathbb{R}_L$ (see Proposition \ref{approx} and Corollary \ref{rips_app}). This is our main motivation for approximating limit groups;
    \item[$\bullet$]more  generally, a splitting $\mathbb{RJ}_L$ of $L$ obtained from $\mathbb{J}_L$ by replacing the unique vertex $u$ fixed by $U$ with $\mathbb{R}_{L_u}$ (see Proposition \ref{approx} and Corollary \ref{comb_app}).
\end{itemize}

Before we construct approximations of $L$ equipped with splittings that mimic one of the aforementioned splittings, we define with more details the sense in which an approximation of $L$ mimics a certain splitting.

\begin{de}\label{core_approx2}
Let $A$ be an approximation of $L$ as in Definition \ref{cor_approx}. Let $\pi : A \twoheadrightarrow L$ be the natural epimorphism (obtained by quotienting out by the image of $\ker(\varphi_\infty)$ in $A$). Suppose that $L$ splits as a graph of groups $\mathbb{S}_L$. One says that $A$ is an $\mathbb{S}_L$-approximation of $L$ if the following four conditions hold:
\begin{enumerate}
    \item $A$ splits as a graph of groups $\mathbb{S}_A$ with the same underlying graph as $\mathbb{S}_L$, and in which all the edge groups are finitely presented and all the vertex groups are finitely presented (relative to $U$);
    \item $\pi$ induces an isomorphism of graphs, denoted by $f$, between the underlying graph of $\mathbb{S}_A$ and the underlying graph of $\mathbb{S}_L$;
    \item if $i : A_e \hookrightarrow A_v$ denotes the inclusion of an edge group of $\mathbb{S}_A$ into an adjacent vertex group, and $j : L_{f(e)} \hookrightarrow L_{f(v)}$ denotes the corresponding inclusion in the graph of groups $\mathbb{S}_L$, then the following diagram commutes:
\begin{center}
\begin{tikzcd}A_e\arrow[r, "i"]\arrow[d, "\pi"]& A_v\arrow[d, "\pi"] \\L_{f(e)}\arrow[r, "j"]& L_{f(v)}\end{tikzcd}
\end{center}
\item $\pi$ maps every edge group $A_e$ of $\mathbb{S}_A$ into the corresponding edge group $L_{f(e)}$ of $\mathbb{S}_L$.
\end{enumerate}
For readability, we omit the isomorphism $f$ and denote the vertex $f(v)$ and the edge $f(e)$ by $v$ and $e$ respectively.
\end{de}

\begin{rque}
The second and third conditions in the definition above can be phrased, equivalently, as follows: there exists a $\pi$-equivariant isomorphism of graphs between the Bass-Serre trees of the splittings $\mathbb{S}_A$ and $\mathbb{S}_L$.
\end{rque}

A similar method of approximating limit groups appears in \cite[Theorem 3.2]{Sel01} (see also \cite{Gro05}, \cite[Lemma 6.1]{RW14} and \cite[Lemma 6.3]{GH19}). Note that in these papers, in the process of approximating a limit group, one constructs countably many approximations; each of them approximates the limit group $L$ to a greater extent than its predecessors. Below we give an alternative construction, which approximates (only) the specific properties of the limit group $L$ required for the proofs of the shortening argument and Merzlyakov's theorem.

\begin{prop}\label{approx}
Suppose that $L$ splits as a graph of groups $\mathbb{S}_L$ in which all the edge groups are virtually abelian. Then there exists an $\mathbb{S}_L$-approximation $A$ of $L$, which in addition satisfies the following two properties.
\begin{enumerate}
    \item If $L_e$ is a finitely generated edge group of $\mathbb{S}_L$, then the quotient map $\pi:A \twoheadrightarrow L$ maps $A_e$ onto $L_e$ (and thus $\pi_{\vert A_e}:A_e \rightarrow L_e$ is an isomorphism).
    \item Let $L_v$ be a vertex group of $\mathbb{S}_L$; if all the edge groups of $\mathbb{S}_L$ adjacent to $L_v$ are finitely generated, then the quotient map $\pi:A \twoheadrightarrow L$ maps $A_v$ onto $L_v$. Moreover, if $L_v$ is finitely presented (relative to $U$), then the map $\pi_{\vert A_v}:A_v \rightarrow L_v$ is an isomorphism.
\end{enumerate}
In addition, for any finite set of relations $\mathcal{F}\subset\ker(\varphi_\infty)$, we can choose the approximation $A$ such that the image of $\mathcal{F}$ in $A$ is trivial.
\end{prop}

\begin{proof}
    We choose to construct the vertex and edge groups of $\mathbb{S}_A$ before constructing the group $A$ itself. By doing so, we hope to give the reader a better understanding of the structure of the approximation $A$. Since the proof of this proposition is quite intricate, we divide it into four steps.
    
    \smallskip
    
    \textbf{Step 1.} We will fix explicit presentations of the edge and vertex groups of the graph of groups $\mathbb{S}_L$ that will be used throughout the proof. Since $L$ is finitely generated, every vertex group $L_v$ of $\mathbb{S}_L$ is finitely generated relative to its adjacent edge groups. Fix a presentation $\langle S \  \vert \ P_U \cup P\rangle$ of $H$, where $P_U$ consists of relations involving only elements from $U$ and $P$ is finite. We also fix a finite generating set $X_U$ of $U$.
    
     \smallskip
     
    For each edge $e\in\mathrm{E}(\mathbb{S}_L)$, fix a presentation $\langle X_{e} \ \vert \ R_{e} \rangle$ of the edge group $L_e$. If $L_e$ is finitely generated (and hence finitely presented, as a virtually abelian group), we choose this presentation to be finite. Note that otherwise, $X_e$ and $R_e$ are both infinite. Then, for each vertex $v \in \mathrm{V}(\mathbb{S}_L)$, fix a presentation of the vertex group $L_v$ of the form 
    \begin{align*}
        \langle X_v=Y_v\cup X^v_{e_1} \cup \cdots \cup X^v_{e_{n_v}} \left. \right | \ R_v=Q_v\cup R^v_{e_1} \cup \cdots \cup R^v_{e_{n_v}} \rangle,
    \end{align*}
    where
    \begin{itemize}
        \item $e_1,\ldots,e_{n_v}$ are the edges adjacent to $v$ in the underlying graph of $\mathbb{S}_L$;
        \item each $\langle X^v_{e_i} \ \vert \ R^v_{e_i}\rangle$ is a copy of the corresponding edge group within $L_v$.
        \item Recall that $L$ is finitely generated; therefore, $L_v$ is finitely generated relative to its adjacent edge groups. Hence, the set $Y_v$ can be chosen finite (and if $L_v$ contains $U$ we choose $Y_v$ to be the union of $X_U$ and a finite set).
        \item $Q_v$ is a (possibly infinite) set of relations.
    \end{itemize}
    In addition, we fix a presentation of $L$ as the fundamental group of $\mathbb{S}_L$, that is
    \begin{align*}
    L=\left\langle \bigcup_{v \in \text{V}(\mathbb{R}_L)} X_v \cup \{t_e, \ e \in \text{E}\} \ \Bigg\vert \ \bigcup_{v \in \text{V}(\mathbb{J}_K)} R_v \cup R \right\rangle
    \end{align*}
    where $\text{E}$ is a subset of the set of edges $\text{E}(\mathbb{R}_L)$ and $R$ is a (possibly infinite) set of relations that identify a set of generators for each edge group with their images in the adjacent vertex groups.
    
    \smallskip
    
    \textbf{Step 2.}
    After having fixed the relevant presentations, we seek to pick a finite set $X_L \subset  \bigcup_{v \in \text{V}(\mathbb{R}_L)} X_v \cup \{t_e, \ e \in \text{E}\}$ that generates $L$. The elements in $X_L$ will be used to define the vertex and edge groups of $\mathbb{S}_A$. We choose $X_L$ to be extensive enough so that each of the relations in the finite set $\varphi_\infty(P \cup \mathcal{F})$ can be written as a product of conjugates of relations from the presentation of $L$ above as the fundamental group of $\mathbb{S}_L$, involving only elements from $X_L$.
    
    \smallskip
    
    For every $s\in S$, write $\varphi_\infty(s)$ as a product of generators appearing in the presentation of $L$ above. Let $X_S$ be the finite subset of the generating set $\bigcup_{v \in \text{V}(\mathbb{R}_L)} X_v \cup \{t_e, \ e \in \text{E}\}$ of $L$ composed of the generators appearing in these products.
    
    \smallskip
     
    Similarly, each relation $r$ in the finite set $\varphi_\infty(P \cup \mathcal{F})$ can be written as a product of conjugates of relations appearing in the presentation of $L$ above. Let $R_L$ be the finite set of relations that participate in such products, and let $X_R$ be the finite subset of $ \bigcup_{v \in \text{V}(\mathbb{R}_L)} X_v \cup \{t_e, \ e \in \text{E}\}$ which consists of all the generators of $L$ participating in the products of conjugates of relations from $R_L$ described above.
    
    \smallskip
    
    Finally, let $X_L=X_S \cup X_R$.
    
    \smallskip
    
    \textbf{Step 3.}
    We finally construct the edge and vertex groups of the splitting $\mathbb{S}_A$. For every $e \in \mathrm{E}(\mathbb{S}_L)$ we define $A_e$ as follows: if $L_e$ is finitely generated (and hence, finitely presented), let $A_e=L_e$. We fix an alternative notation for the finite presentation of $A_e$: $\langle X_e' \ \vert \ R'_e\rangle$. If $L_e$ is not finitely presented, let $A_e$ be the subgroup of $L_e$ generated by $L_e \cap X_L$; note that $A_e$ is finitely presented (as a finitely generated virtually abelian group) and fix a finite presentation $\langle X_e' \ \vert \ R'_e \rangle$ of $A_e$. Up to modifying the original presentation of $L_e$, we may assume that $X_e'$ is a subset of $X_e$.
    
    \smallskip

    For every $v \in \mathrm{V}(\mathbb{S}_L)$ we define $A_v$ as follows: if $L_v$ is finitely presented over $U$, we let $A_v=L_v$ (and fix an alternative notation for the presentation of $A_v$: $\langle X'_v \ \vert \ R'_v \rangle$); otherwise, we set $A_v$ to be the group admitting the following presentation:
    \begin{align*}
        \big\langle X'_v=Y_v\cup (X'_{e_1})^v \cup \cdots \cup (X'_{e_{n_v}})^v \left. \right | \ R'_v=Q'_v \cup \cdots \cup (R_{e_{n_v}}')^v \big\rangle,
    \end{align*}
    where
    \begin{itemize}
        \item $e_1,\ldots,e_{n_v}$ are the edges adjacent to $v$ in the underlying graph of $\mathbb{R}_L$;
        \item $Y_v$ is as in the presentation of $L_v$;
        \item each $\langle (X'_{e_i})^v \ \vert \ (R_{e_i}')^v \rangle$ is a copy of $A_{e_i}$ within $A_v$;
        \item if $L_v$ does not contain $U$, one has $Q'_v=Q_v\cap R_L$ if $Q_v$ is infinite, and $Q'_v=Q_v$ otherwise. If $L_v$ contains $U$, we pick $Q'_v$ in the same manner, but include in $Q'_v$ the (possibly infinite) subset of $Q_v$ which consists of relations involving only elements from $U$.
    \end{itemize}
    
    \smallskip
    
    Recall that $R$ is the set of relations from the presentation of $L$ which identify a set of generators for each edge group with their images in the adjacent vertex groups. Let $R'$ be the finite set identifying the generators $X'_e$ of $A_e$ with their images $(X'_e)^v$ in $A_v$, whenever $v\in \mathrm{V}(\mathbb{R}_L)$ is adjacent to $e \in \mathrm{E}(\mathbb{S}_L)$. Let $\mathcal{R}=\bigcup_{v \in \mathrm{V}(\mathbb{R}_L)}R'_v \cup R'$, where all of the relations in the union are written with the letters of the generating set $S$ of $H$. Note that $\mathcal{R}$ is finite. Now, let $A$ be the $\mathcal{R}$-approximation of $L$, that is define $A=H/\langle \langle \mathcal{R} \rangle \rangle$.
    
    \smallskip
    
    \textbf{Step 4.}
    We now show that $A$ satisfies the desired properties. First, note that $A$ admits the presentation $\langle S \ \vert \ P_U \cup P \cup \mathcal{R} \rangle$ in the generators of $H$. By expressing this presentation in terms of the generators of $X_L$, we obtain the following presentation of $A$:
    \begin{align*}
        A=\left\langle \bigcup_{v \in \text{V}(\mathbb{S}_L)} X'_v \cup \{t_e, \ e \in \text{E}\} \ \Bigg\vert \ R_L \cup \mathcal{R} \right\rangle.
    \end{align*}
    But since $R_L$ is contained in $\mathcal{R}$ by definition of the $R'_v$, one can omit $R_L$ in the previous presentation of $A$. Hence, $A$ is simply the fundamental group of the graph of groups $\mathbb{S}_A$ obtained from $\mathbb{S}_L$ by replacing each vertex group $L_v$ with the group $A_v$, and each edge group $L_e$ with the group $A_e$. In addition, all of the relations in $\mathcal{F}$ hold in $A$. This shows that condition (1) of Definition \ref{core_approx2} holds.
    
    \smallskip
    
    Next, note that the map $\pi'$ defined by mapping each generator in the presentation of $A$ above to the corresponding generator in the presentation of $L$ as the fundamental group of $\mathbb{S}_L$ coincides with the natural epimorphism $\pi:A \twoheadrightarrow L$ obtained by quotienting out the image of $\ker(\varphi_\infty)$ in $A$. Indeed, denote by $q$ the quotient map $H\twoheadrightarrow A$ and observe that for every $s \in S$ one has $\pi'\circ q(s)=\varphi_\infty(s)$; this implies that $\pi'=\pi$. Last, properties (2), (3) and (4) appearing in Definition \ref{core_approx2} are clearly satisfied.
    
    \smallskip
    
    To finish, let us check that properties \emph{(1)} and \emph{(2)} of Proposition \ref{approx} hold: for property \emph{(1)}, recall that whenever an edge group $L_e$ of $\mathbb{S}_L$ is finitely generated, we defined $A_e$ to be $L_e$. For property \emph{(2)}, recall that if all the edge groups adjacent to a vertex group $L_v$ of $\mathbb{S}_L$ are finitely generated, then the generators $X'_v$ of $A_v$ correspond to the generators $X_v$ of $L_v$. This implies that the restriction of the map $\pi:A\twoheadrightarrow L$ to $A_v$ is a surjection. If in addition $L_v$ is finitely presented (over $U$), then $A_v$ and $L_v$ admit the same presentation and $\pi$ maps $A_v$ isomorphically to $L_v$.
\end{proof}

\begin{rque}\label{trick}Suppose that $L$ admits a splitting $\mathbb{S}_L$, and let $\lbrace h_1,\ldots ,h_k\rbrace$ be a finite set of elements of $L$. Write each element $h_i$ as a product $s_{i,1}\cdots s_{i,m_i}$ of generators appearing in the presentation of $L$ as the fundamental group of $\mathbb{S}_L$. By choosing the finite set of relations $\mathcal{F}$ in Proposition \ref{approx} above wisely, we can make sure that $L$ has an $\mathbb{S}_L$-approximation $A$ such that each $h_i$ has a primage $a_i$ in $A$ that admits the same decomposition as $h_i$ as a product of generators. More precisely, let $\tilde{h}_i,\tilde{s}_{i,1},\ldots,\tilde{s}_{i,m_i}$ be lifts of $h_i,s_{i,1},\ldots,s_{i,m_i}$ to $H$, for $1\leq i\leq k$. Let 
    \begin{align*}
        \mathcal{F}=\{\tilde{h}_i^{-1}\tilde{s}_{i,1}\cdots\tilde{s}_{i,m_i}, \ 1\le i \le k\}\subset \ker(\varphi_\infty),
    \end{align*}
    and let $A$ be an $\mathbb{S}_L$-approximation of $L$ in which the relations in $\mathcal{F}$ hold. Then by Proposition \ref{approx} above, all of the generators $s_{i,j}$ appear in the presentation of $A$ as the fundamental group of $\mathbb{S}_A$, and the image of $h_i$ in $A$ can be simply written as $s_{i,1}\cdots s_{i,m_i}$. We will use this method in our proof of the general version of the shortening argument.
\end{rque}

We now deduce from Proposition \ref{approx} a series of three corollaries. 

\begin{co}\label{JSJ_app}
    Suppose that $L$ admits a splitting $\mathbb{S}_L$ in which all the edge groups are finite (for instance, $\mathbb{S}_L$ can be a reduced JSJ splitting of $L$ relative to $U$ over finite subgroups of order $\le C$, denoted by $\mathbb{J}_L$). In this case, all the vertex groups of $\mathbb{S}_L$ are finitely generated. Then there exists an $\mathbb{S}_L$-approximation $A$ of $L$, whose splitting is denoted by $\mathbb{S}_A$, such that $\mathbb{S}_A$ and $\mathbb{S}_L $ share the same edge groups, and the vertex groups of $\mathbb{S}_A$ surject onto those of $\mathbb{S}_L$.
\end{co}

\begin{proof}Since edge groups of $\mathbb{S}_L$ are finite, they are virtually abelian and finitely generated. Thus the existence of $A$ is an immediate consequence of Proposition \ref{approx}.\end{proof}

As mentioned earlier, the main motivation for defining approximations of limit groups is to approximate in an accurate manner splittings of $L$ as a graph of actions outputted by the Rips machine \ref{ripsmachine}. The following lemma proves that the approximations given by Proposition \ref{approx} capture many of the properties of such splittings.

\begin{co}\label{rips_app}
Suppose that $L$ does not split non-trivially over a finite subgroup of order $\leq C$, and let $\mathbb{R}_L$ be a splitting of $L$ as a graph of actions outputted by the Rips machine. Let $A$ be an $\mathbb{R}_L$-approximation of $L$ given by Proposition \ref{approx} and let $\mathbb{R}_A$ denote its splitting. Then the following hold:
\begin{enumerate}
    \item every edge group of $\mathbb{R}_A$ is finitely presented and virtually abelian, and the quotient map $\pi: A\twoheadrightarrow L$ maps every edge group of $\mathbb{R}_A$ into the corresponding edge group of $\mathbb{R}_L$; 
    \item if $L_v$ is a simplicial vertex group of $\mathbb{R}_L$ and $A_v$ is the corresponding vertex group of $\mathbb{R}_A$, then $\pi$ maps $A_v$ onto a finitely generated subgroup of $L_v$;
    \item if $L_v$ is a Seifert-type vertex group of $\mathbb{R}_L$ and $A_v$ is the corresponding vertex group of $\mathbb{R}_A$, then $\pi$ maps $A_v$ isomorphically to $L_v$;
    \item if $L_v$ is an axial vertex group of $\mathbb{R}_L$ and $A_v$ is the corresponding group of $\mathbb{R}_A$, then $A_v$ is finitely presented (relative to $U$) and virtually abelian, and $\pi$ maps $A_v$ into $L_v$.
\end{enumerate}
\end{co}

\begin{proof}Recall that the edge groups of $\mathbb{R}_L$ are virtually abelian, and that each edge group $A_e$ of $\mathbb{R}_A$ is finitely presented (condition (1) in Definition \ref{core_approx2}), and that $\pi$ maps $A_e$ into the corresponding edge group $L_e$ (condition (4) in Definition \ref{core_approx2}). Hence, $A_e$ is virtually abelian and the first assertion above holds. 

\smallskip

Next, note that the second assertion is an immediate consequence of \emph{(2)} in Proposition \ref{approx}. 

\smallskip

For \emph{(3)}, let $L_v$ be a Seifert-type vertex group of $\mathbb{R}_L$. Note that $L_v$ is finitely presented (relative to $U$). We will prove that $\pi$ maps the corresponding vertex group $A_v$ of $\mathbb{R}_A$ isomorphically to $L_v$. By the second assertion of Proposition \ref{approx}, it is enough to show that $L_v$ does not contain infinitely generated abelian subgroups, and thus that all the edge groups adjacent to $L_v$ are finitely generated. This follows from the following easy observation: since $L_v$ is a Seifert-type vertex group, it is hyperbolic, and therefore all of its abelian subgroups are virtually cyclic. 

\smallskip

Lastly, for (4), note that any subgroup of $L_v$ which is finitely generated (relative to $U$) is finitely presented (relative to $U$) since $L_v$ is virtually abelian. This implies that, as part of the construction of $A_v$ in the proof of Proposition \ref{approx}, $A_v$ is in fact a subgroup of $L_v$ which is finitely presented (relative to $U$) and $\pi$ maps $A_v$ injectively to $L_v$.
\end{proof}

The last corollary is a combination of Corollaries \ref{JSJ_app} and \ref{rips_app}, and it can be proved in a similar way.

\begin{co}\label{comb_app}
    Let $\mathbb{J}_L$ be a reduced JSJ splitting of $L$ over finite groups of order $\leq C$ relative to $U$, and let $\mathbb{RJ}_L$ be a splitting of $L$ obtained from $\mathbb{J}_L$ by replacing the unique vertex $u$ fixed by $U$ with $\mathbb{R}_{L_u}$. Note that $L_u$ can be viewed as a limit group whose defining sequence of homomorphisms is $({\varphi_n}_{\vert H_u})_{n\in\mathbb{N}}$, where $H_u$ denotes a finitely generated subgroup of $H$ that contains $U$ and such that $\varphi_\infty(H_u)=L_u$. Then any $\mathbb{RJ}_L$-approximation $A$ of $L$ outputted by Proposition \ref{approx} admits a splitting $\mathbb{RJ}_A$ satisfying the following conditions.
\begin{enumerate}
    \item Denote by $\mathbb{R}_{A_u}$ the subgraph of $\mathbb{RJ}_A$ which corresponds to the subgraph $\mathbb{R}_{L_u}$ of $\mathbb{RJ}_L$. Denote by $A_u$ the fundamental group of $\mathbb{R}_{A_u}$ and note that $A_u$ is a lift of $L_u$ to $A$. Then $A_u$ is an $\mathbb{R}_{L_u}$-approximation of $L_u$. Furthermore, the splitting $\mathbb{R}_{A_u}$ of $A_u$ enjoys the properties described in Corollary \ref{rips_app}.
    \item Let $\mathbb{J}_A$ be the splitting of $A$ obtained by collapsing to a point the subgraph $\mathbb{R}_{A_u}$ of $\mathbb{RJ}_A$. Then $A$, equipped with the splitting $\mathbb{J}_A$, is a $\mathbb{J}_L$-approximation of $L$.
\end{enumerate}
\end{co}

\begin{rque}
Note that the splitting $\mathbb{RJ}_L$ is not unique in general since the finite edge groups adjacent to the vertex $u$ in $\mathbb{J}_L$ may fix several vertex groups in the splitting $\mathbb{R}_{L_u}$ of $L_u$.
\end{rque}

\subsection{The shortening argument}\label{shortening}

The \emph{shortening argument} encompasses a wide array of results, all of which share a similar nature: shortening homomorphisms. The classical result asserts that given a sequence of homomorphisms from a finitely generated group to another group from a certain class, either one can shorten the homomorphisms (in some sense), or the stable kernel of the sequence is non-trivial. For the class of acylindrically hyperbolic groups, a version of the shortening argument is proven in \cite[Theorem 5.29]{GH19}. In this section, we provide two additional versions of this result which will help us deal with inequalities in the proof of Merzlyakov's theorem. We first define the notion of a \emph{short} homomorphism.

\begin{de}
    The \emph{length} (or \emph{scaling factor}) of a homomorphism $\varphi:H\rightarrow G$ is defined by
    \begin{align*}
        \norm{\varphi}=\inf_{y \in X} \max_{s \in S} d(y,\varphi(s)y),
    \end{align*}
    where $S$ is a finite generating set of $H$. We call $\varphi$ \emph{short relative to} $U$ if for every homomorphism $\phi:H\rightarrow G$ whose restriction to $U$ coincides with $\varphi_{\vert U}$ up to conjugation (that is, there exists $g \in G$ such that $\phi(h)=g \varphi(h) g^{-1}$ for every $h \in U$), one has:
    \begin{align*}
        \norm{\varphi} \le \norm{\phi}.
    \end{align*}
\end{de}

We would like to point out once again the main difficulty which stands in our way: unlike hyperbolic groups, acylindrically hyperbolic groups are not equationally Noetherian in general. Therefore, the sequence of homomorphisms $(\varphi_n)_{n\in\mathbb{N}}$ does not necessarily factor via the limit group $L$ ($\omega$-almost-surely) and one can not use automorphisms of $L$ in order to shorten the homomorphisms in the sequence $(\varphi_n)_{n\in\mathbb{N}}$. To combat this, we use approximations of $L$ which were defined in the previous subsection (see Definitions \ref{cor_approx} and \ref{core_approx2}). Since the homomorphisms in the sequence $(\varphi_n)_{n \in \mathbb{N}}$ do factor via an approximation $A$ of $L$ $\omega$-almost-surely, we can use automorphisms of $A$ in order to shorten the sequence $(\varphi_n)_{n\in\mathbb{N}}$. The automorphisms which we use are lifts of a certain type of \emph{modular} automorphisms of $L$ (see Definitions \ref{mod} and \ref{mod*}) to $A$. 

\smallskip

Before stating and proving our two versions of the shortening argument, we begin by collecting a few definitions and results. 

\begin{de}
\label{natural_ext}
\cite[Definition 3.13]{RW14}
Let $G$ be a group which splits as a graph of groups $\mathbb{S}$ and let $G_v$ be one of its vertex groups. Suppose that $\alpha_v\in \mathrm{Aut}(G_v)$ satisfies the following property: for every edge group $G_{e_i}$ adjacent to $G_v$ there exists an element $c_{e_i}\in G_v$ such that $\alpha_v$ restricts to conjugation by $c_{e_i}$ on $G_{e_i}$. Recall that each element of $G$ can be realized as a loop in the graph of groups $\mathbb{S}$. The homomorphism $\alpha : G \rightarrow G$ defined by
\begin{align*}
    [a_0,e_1,a_1,\ldots,e_k,a_k]\mapsto [b_0,e_1,b_1,\ldots,e_k,b_k]
\end{align*}
where
\begin{align*}
        b_i=\begin{cases} a_i & a_i \notin A_v \\ c_{e_i^{-1}}^{-1}\alpha_v(a_i)c_{e_{i+1}} & a_i \in A_v \end{cases}.
\end{align*}
is called a \emph{natural extension} of $\alpha_v$.\end{de}

\begin{rque}Note that the elements $c_{e_i}$ are not unique in general, and hence the morphism $\alpha$ above is not uniquely defined by $\alpha_v$.\end{rque}

The following short lemma shows that such a natural extension $\alpha$ is an automorphism of $G$.

\begin{lemme}\label{extension_lemma}Let $G$ be a group that splits as a graph of groups $\mathbb{S}$; let $G_v$ be one of its vertex groups. Let $\alpha_v\in \mathrm{Aut}(G_v)$ satisfy the properties appearing in Definition \ref{natural_ext} above and let $\alpha:G\rightarrow G$ be a natural extension of $\alpha_v$. Then $\alpha$ is a well-defined automorphism of $G$ whose restriction to $G_v$ is $\alpha_v$, and whose restriction to every edge group of $\mathbb{S}$ is a conjugation (by some element, depending on the edge).\end{lemme}

\begin{proof}
Since $G$ can be realized as a sequence of amalgamated products followed by a sequence of HNN extensions, it is enough to prove the lemma in the case where $\mathbb{S}$ has only one edge.

\smallskip

\emph{First case.} Suppose that $G=A\ast_C B$, and assume that $\alpha_v$ is an automorphism of $A$ such that $\alpha_{v \vert C}=\mathrm{ad}(a)$ for some $a\in A$. Define $\alpha$ as in Definition \ref{natural_ext}, that is: $\alpha_{\vert A}=\alpha_v$ and $\alpha_{\vert B}=\mathrm{ad}(a)$. This endomorphism is well-defined, and it is clearly surjective since its image contains $\alpha_v(A)=A$ and $aBa^{-1}$, which generate $G$. Let us prove that $\alpha$ is injective. Consider a non-trivial element $g=a_1b_1a_2b_2\cdots a_nb_n\in G$ written in normal form. The elements $a_i$ and $b_i$ do not belong to $C$, except maybe $a_1$ or $b_n$. One can write $\alpha(g)=a'_1b_1a'_2b_2\cdots a'_nb_na'_{n+1}$ with $a'_1=\alpha_v(a_1)a$, $a'_i=a^{-1}\alpha_v(a_i)a$ for $1< i\leq n$ and $a'_{n+1}=a^{-1}$. Observe that $a'_i$ does not belong to $C$ for $1< i\leq n$, otherwise $a'_i=a^{-1}\alpha_v(a_i)a=c\in C$, thus $\alpha_v(a_i)=aca^{-1}=\alpha_v(c)$. It follows that $a_i=c$; this is a contradiction. Hence, the previous decomposition of $\alpha(g)$ is in normal form, which proves that $\alpha(g)$ is not trivial. 

\smallskip

\emph{Second case.} Suppose that $G=\langle A,t \ \vert \ tct^{-1}=\sigma(c), \ \forall c\in C_1\rangle$, where $\sigma$ denotes an isomorphism between two subgroups $C_1$ and $C_2=\sigma(C_1)$ of $A$. Suppose that $\alpha_v$ is an automorphism of $A$ such that $\alpha_{v \vert C_i}=\mathrm{ad}(a_i)$ for some $a_i\in A$, for $1\leq i\leq 2$. Define $\alpha$ as follows: $\alpha_{\vert A}=\alpha_v$ and $\alpha(t)=a_2ta_1^{-1}$. As in the first case, one easily sees that $\alpha$ is well-defined and surjective. The injectivity follows from Britton's lemma, by a similar argument as above.\end{proof}

We next define the modular group of a limit group (see \cite[Definition 5.22]{GH19}).

\begin{de}
\label{mod}
Suppose that $L$ admits a splitting as a graph of actions $\mathbb{R}_L$ outputted by the Rips machine. The \emph{modular group} $\mathrm{Mod}_{\mathbb{R}_L}(L)$ associated with the splitting $\mathbb{R}_L$ is the subgroup of $\mathrm{Aut}(L)$ generated by the following automorphisms.
\begin{enumerate}
    \item Inner automorphisms.
    \item Dehn twists over the virtually abelian edge groups of $\mathbb{R}_L$: if $L_e$ is an edge group of $\mathbb{R}_L$ and $c \in Z(L_e)$ then the \emph{Dehn twist by} $c$ is the automorphism of $L$ given by
    \begin{align*}
        \begin{cases}
        \tau_c(a)=a,\, \tau_c(b)=cbc^{-1} & \text{if }L=A_1\ast_{A_e} A_2,\, a\in A_1 \text{ and } b\in A_2\\
        \tau_c(a)=a,\, \tau_c(t)=tc & \text{if }L=A\ast_{L_e}\text{with stable letter }t \text{ and } a\in A.
        \end{cases}
    \end{align*}
    \item Natural extensions of automorphisms of Seifert-type vertex groups which are induced by homeomorphisms of the underlying $2$-orbifold and which fix the boundary and conical points.
    \item Natural extensions (in the sense of Definition \ref{natural_ext}) of automorphisms of axial vertex groups, which satisfy the condition below. Denote by $L_v$ an axial vertex group of $\mathbb{R}_L$, then by \cite[Lemma 5.1]{GH19} every subgroup $B \le L_v$ is virtually abelian, and has a unique maximal subgroup $B^+$ of index at most $2$ which is finite-by-abelian. Denote by $E(L_v)$ the subgroup of $L_v$ generated by its adjacent edge groups. We allow natural extensions of automorphisms $\alpha_v$ of $L_v$ for which:
    \begin{enumerate}
        \item $\alpha_v$ fixes the subgroup $P_v^+$ of $L_v$ which consists of all $g \in L_v$ such that $g \in \ker(\phi)$ for every homomorphism $\phi:L_v\rightarrow \mathbb{Z}$ satisfying $E(L_v)\cap A_v^+ \subset \ker(\phi)$, and
        \item $\alpha_v$ restricts to conjugation on every subgroup $B \le L_v$ for which $B^+=P_v^+$.
    \end{enumerate}
\end{enumerate}
\end{de}

\begin{rque}
    \label{conj_fin}
    Note that every modular automorphism of one of the types (1)-(4) above restricts to conjugation on every finite subgroup of $L$, and hence modular automorphisms always restrict to conjugation on finite subgroups of $L$.
\end{rque}

Our next goal is to show that given a modular automorphism $\alpha$ of $L$, under some restrictions, one can find an approximation $A$ of $L$ and a lift $\beta \in \mathrm{Aut}(A)$ of $\alpha$. As Lemma \ref{mod_approx} will show, every modular automorphism of types (1)-(3) above admits a lift to an approximation of $L$. This follows from the extent to which one can approximate the edge groups and the Seifert-type vertex groups of the splitting of $L$ as a graph of actions, as evident in Corollary \ref{rips_app}. However, dealing with axial vertex groups is slightly more complicated. We therefore discuss further the structure of axial vertex groups of $L$ and describe a few properties of modular automorphisms of type (4) used in the proof of the shortening argument. We follow \cite[Subsection 4.2.1]{RW14} and refer the reader to \cite{RW14} for further details.

\smallskip

Suppose that $L$ admits a splitting as a graph of actions $\mathbb{R}_L$ outputted by the Rips machine and that $L_v$ is an axial vertex group of $L$. Denote by $E\le L_v$ the torsion subgroup of $L_v$ and by $H$ the quotient $L_v/E$; let $\pi_E$ be the quotient map $L_v\twoheadrightarrow H$. Recall that $L_v$ has a subgroup $L_v^+$ of index at most $2$ which is finite-by-abelian, and let $H^+$ be the image of $L_v^+$ in $H$. The group $H^+$ admits a decomposition $H^+=A \oplus B$ where $A$ is a finitely generated free abelian group and $B$ is the torsion-free (and abelian) kernel of the action of $H^+$ on the line $T_v\subset T$. Let $\tilde{A}=\pi_E^{-1}(A)$ and $\tilde{B}=\pi_E^{-1}(B)$. As in \cite[Subsection 4.2.1]{RW14}, there exists an element $s\in L_v$ such that $L_v=\langle A,B,s \rangle$ and every $g\in L_v$ can be written as a product of the form
\begin{align*}
    g=abs^{\eta}
\end{align*}
where $a\in \tilde{A}$, $b\in \tilde{B}$ and $\eta \in \{0,1\}$.

\smallskip

We now define the subgroup $\mathrm{Aut}^*(L_v)$ of $\mathrm{Aut}(L_v)$ to be the subgroup which consists of all the automorphisms $\alpha_v \in \mathrm{Aut}(L_v)$ which satisfy the following three properties:
\begin{enumerate}
    \item $\alpha_v$ preserves $\tilde{A}$;
    \item $\alpha_v$ restricts to the identity on $\langle \tilde{B},s \rangle$;
    \item consider the action of $L_v$ on the line $T_v \subset T$. Then for every $x\in T_v$, $\alpha_v$ restricts to conjugation on the stabilizer $(L_v)_x$ of $x$. 
\end{enumerate}

This leads us to define the following subgroup of $\mathrm{Mod}_{\mathbb{R}_L}(L)$.
\begin{de}
\label{mod*}
The group $\mathrm{Mod}^*_{\mathbb{R}_L}(L)$ is the subgroup of $\mathrm{Mod}_{\mathbb{R}_L}(L)$ generated by:
\begin{enumerate}
    \item modular automorphisms of types (1)-(3) (see Definition \ref{mod});
    \item modular automorphisms $\alpha$ of type (4) which satisfy the following: if $\alpha$ is a natural extension of $\alpha_v\in \mathrm{Aut}(L_v)$ for an axial vertex group $L_v$ of $L$, then $\alpha_v \in \mathrm{Aut}^*(L_v)$.
\end{enumerate}
\end{de}

The motivation behind Definition \ref{mod*} comes from the fact that, by \cite[Subsection 4.2.1]{RW14}, it is enough to use modular automorphisms which lie in $\mathrm{Mod}^*_{\mathbb{R}_L}(L)$ in the proof of the shortening argument. We finish this discussion with the following easy observation.

\smallskip

\noindent \textbf{Observation:} suppose that $L_v$ is an axial vertex group of $L$; since the torsion subgroup $E$ of $L_v$ is finite (by Lemma \ref{stabilitylemma}) and since the subgroup $A$ of $H^+$ is finitely generated, there are finitely generated subgroups of $L_v$ which contain $\tilde{A}$. In addition, let $\alpha_v\in \mathrm{Aut}^*(L_v)$ and suppose that $L'_v$ is any subgroup of $L_v$ which contains $\tilde{A}$, then the restriction of $\alpha_v$ to $L'_v$ is an automorphism.

\begin{prop}\label{mod_approx}
    Suppose that $L$ does not split non-trivially over a subgroup of order $\le C$ and let $\mathbb{R}_L$ be the graph of actions decomposition of $L$ outputted by the Rips machine. Let $\alpha\in \mathrm{Mod}^*_{\mathbb{R}_L}(L)$. Then there exists an $\mathbb{R}_L$-approximation $A$ of $L$ and $\beta \in \mathrm{Aut}(A)$ such that the following diagram commutes $\omega$-almost-surely.
    \begin{align*}
        \xymatrix{H \ar[rrrrrr]^{\varphi_n} \ar@{->>}[drr]^q \ar@{->>}[ddrr]_{\varphi_\infty} && && && G \\
        && A \ar[rr]_\beta \ar@{->>}[d]^{\theta_\infty} \ar[urrrr]^{\theta_n}&& A \ar@{->>}[d]^{\theta_\infty} && \\
        && L \ar[rr]^\alpha && L &&
        }
    \end{align*}
    In particular, for every $h\in H$ such that $\varphi_n(h)\ne 1$ $\omega$-almost-surely, $\theta_n \circ \beta \circ q(h)\ne 1$ $\omega$-almost-surely.
\end{prop}

\begin{proof}
    Write $\alpha=\alpha_k\circ \cdots \circ \alpha_1$ where for every $1\le i \le k$, $\alpha_i\in \mathrm{Mod}_{\mathbb{R}_L}(L)$ is a modular automorphism of $L$ of one of the types (1)-(4) appearing in Definition \ref{mod}. Furthermore, if $\alpha_i$ is a modular automorphism of type (4), we assume that it satisfies the condition appearing in Definition \ref{mod*}. It is enough to find an $\mathbb{R}_L$-approximation $A$ of $L$ and an automorphism $\beta \in \mathrm{Aut}(A)$ for which
    \begin{align*}
        \theta_\infty \circ \beta = \alpha \circ \theta_\infty.
    \end{align*}
    In fact, it is enough to show that there is an $\mathbb{R}_L$-approximation $A$ of $L$ and automorphisms $\beta_1,\ldots,\beta_k$ of $L$ such that the following holds for every $1\le j \le k$:
    \begin{align*}
        \theta_\infty \circ \beta_j = \alpha_j \circ \theta_\infty.
    \end{align*}
    
    We begin with the construction of the approximation $A$ of $L$. We define a finite subset $C$ of $L$ as follows: 
    \begin{enumerate}
        \item Whenever $\alpha_i$ is a Dehn twist for $1\le i\le k$, we add to $C$ an element $c$ which lies in an edge group of $\mathbb{R}_L$ and such that $\alpha_i$ is a Dehn twist by $c$.
        \item We keep the notations from the discussion appearing before Definition \ref{mod*}. Suppose now that $\alpha_i$ is a modular automorphism of type (4); denote by $\alpha_v\in \mathrm{Aut}^*(L_v)$ an automorphism of an axial vertex group $L_v$ of $L$ such that $\alpha_i$ is a natural extension of $\alpha_v$. Let $L_{e_1},\ldots,L_{e_k}$ be the edge groups adjacent to $L_v$ and recall that $\alpha_v$ restricts to conjugation by elements $c_1,\ldots,c_k \in L_v$ on $L_{e_1},\ldots,L_{e_k}$ respectively. Let $L'_v$ be a finitely generated subgroup of $L_v$ (over $U$) which contains $\tilde{A}$ and $c_1,\ldots,c_k$; such a group exists by the observation following Definition \ref{mod*}. Denote by $S'_v$ a finite set of generators of $L'_v$. We add the elements in $S'_v$ to $C$.
    \end{enumerate}
     Let $A$ be an $\mathbb{R}_L$-approximation of $L$ which satisfies the following condition: for every $c\in C$ belonging to an edge group $L_e$ of $\mathbb{R}_L$, there is an element $c' \in A_e$ such that $\theta_\infty(c')=c$. Such an approximation $A$ of $L$ exists by Remark \ref{trick}. Recall that by Lemma \ref{core_approx2}, $A$ admits a splitting $\mathbb{R}_A$ which satisfies the following: $\theta_\infty$ maps every vertex group or edge group of $A_v$ to the corresponding vertex or edge group of $L_v$. In addition, $\theta_\infty$ maps every stable letter in the presentation of $A$ as the fundamental group of $\mathbb{R}_A$ to the corresponding stable letter in the presentation of $L$ as the fundamental group of $\mathbb{R}_L$. 
    
    \smallskip
    
    We next construct the automorphisms $\beta_1,\ldots,\beta_k$ of $A$. Let $1\le j \le k$. We divide the construction of $\beta_j$ into cases, depending on the type of the modular automorphism $\alpha_j \in \mathrm{Mod}_{\mathbb{R}_L}(L)$. 
    \begin{enumerate}
        \item $\alpha_j$ is a modular automorphism of $L$ of type (1), that is conjugation by some element $g\in L$. Let $g'\in A$ be such that $\theta_\infty(g')=g$ and set $\beta_j$ to be conjugation by $g'$. It is clear that the desired equality holds.
        
        \item $\alpha_j$ is a modular automorphism of $L$ of type (2), that is a Dehn twist by some element $c\in C$ which lies in an edge group $L_e$ of $\mathbb{R}_L$. By the manner in which the approximation $A$ was chosen, there is an element $c'\in A_e$ such that $\theta_\infty(c')=c$. Let $\beta_j$ be a Dehn twist by $c'$. Assume that by collapsing every edge except for $e$, $A$ splits as an amalgamated product over $A_e$, that is $A=A_1 \ast _{A_e} A_2$; the case where $A$ splits as an HNN extension is similar. It follows that $L$ splits as an amalgamated product over $L_e$; write $L=L_1 \ast_{L_e} L_2$. In addition, by the properties of the approximation $A$, one has $\theta_\infty(A_i) \subset L_i$ for $i\in \{1,2\}$. 
        
        \smallskip
        
        Let $g\in A$ and write $g$ as an alternating product of elements from $A_1$ and $A_2$, that is $g=a_1b_1a_2\cdots a_mb_m$ with $a_i\in A_1$ and $b_i\in A_2$ for $1\le i \le m$. We have that
        \begin{align*}
            \theta_\infty \circ \beta_j (g) & = \theta_\infty \circ \beta_j (a_1b_1a_2\cdots a_mb_m) \\
            & = \theta_\infty (a_1 (c' b_1 (c')^{-1}) a_2 \cdots a_m (c' b_m (c')^{-1}) \\
            & = \theta_\infty(a_1) (c \theta_\infty (b_1) c^{-1}) \theta_\infty(a_2) \cdots \theta_\infty(a_m) (c \theta_\infty(b_m) c^{-1}) \\
            & = \alpha_j(\theta_\infty(a_1)\theta_\infty(b_1)\theta_\infty(a_2)\cdots \theta_\infty(a_m)\theta_\infty(b_m)) \\
            & = \alpha_j \circ \theta_\infty (g).
        \end{align*}
        
        \item $\alpha_j$ is a modular automorphism of type (3), that is a natural extension of an automorphism $\alpha_v$ of a vertex group $L_v$ of $\mathbb{R}_L$ of Seifert-type, as described in Definition \ref{mod}. Let $e^v_1,\cdots,e^v_\ell\in \mathrm{E}(\mathbb{R}_L)$ be an enumeration of the edges of $\mathbb{R}_L$ which are adjacent to $v$ and recall that $\alpha_j$ restricts to conjugation by some $c_{e^v_i} \in L_v$ on $L_{e^v_i}$ for every $1\le i \le \ell$. In addition, by Corollary \ref{rips_app}, $\theta_\infty$ maps $A_v$ isomorphically to $L_v$. This implies that $\alpha_v$ is an isomorphism of $A_v$, and that there are elements $c'_{e^v_1},\ldots,c'_{e^v_\ell} \in A_v$ such that $\theta_\infty(c'_{e^v_i})=c_{e^v_i}$ and $\alpha_v$ restricts to conjugation by $c'_{e^v_i}$ on $A_{e^v_i}$ for every $1\le i \le \ell$. Let $\beta_j$ be the natural extension of $\alpha_v$ to $A$, with respect to the elements $c'_{e^v_1},\ldots,c'_{e^v_\ell}$. Now let $g\in A$ and write $g$ as a loop in the graph of groups $\mathbb{R}_A$, that is $g=[a_0,e_1,a_1,\ldots,e_k,a_k]$. Then $\beta_j(g)=[b_0,e_1,b_1,\ldots,e_k,b_k]$ where
        \begin{align*}
            b_i=\begin{cases} a_i & a_i \notin A_v \\ {c'_{e_i}}^{-1}\alpha_v(a_i)c'_{e_{i+1}} & a_i \in A_v \end{cases}.
        \end{align*}
        To finish, note that $\theta_\infty(g)$ can be written as a loop $[\theta_\infty(a_0),e_1,\theta_\infty(a_1),\ldots,e_k,\theta_\infty(a_k)]$ in $\mathbb{R}_L$, which implies that $\alpha_j\circ \theta_\infty(g)=[c_0,e_1,c_1,\ldots,e_k,c_k]$ where
        \begin{align*}
            c_i=\begin{cases} \theta_\infty(a_i) & \theta_\infty(a_i) \notin A_v \\ c_{e_i^{-1}}^{-1}\alpha_v(\theta_\infty(a_i)))c_{e_{i+1}} & a_i \in A_v \end{cases}.
        \end{align*}
        One easily sees that $\theta_\infty(b_i)=c_i$ for $1\le i \le k$, which implies that
        \begin{align*}
            \theta_\infty \circ \beta_j \circ \cdots \circ \beta_1  = \alpha_j \circ \cdots \circ \alpha_1 \circ \theta_\infty.
        \end{align*}
        \item $\alpha_j$ is a modular automorphism of $L$ of type (4), and which satisfies the condition appearing in Definition \ref{mod*}. In particular, $\alpha_j$ is a natural extension of an automorphism $\alpha_v\in \mathrm{Aut}^*(L_v)$ of an axial vertex group $L_v$ of $L$. Note that $\theta_\infty$ maps $A_v$ into $L_v$; we identify $A_v$ with its image in $L_v$. By the manner in which the set $C$ was defined, and by the observation following Definition \ref{mod*}, we have that the restriction of $\alpha_v$ to $A_v$ is an automorphism. One can continue as in (3) above, by taking a natural extension of ${\alpha_v}_{\vert _{A_v}}$ to $A$.
    \end{enumerate}
    Finally, let $\beta=\beta_k \circ \cdots \circ \beta_1$. The construction of the automorphisms $\beta_1,\ldots,\beta_k$ implies that the diagram appearing in the statement of this proposition does commute $\omega$-almost-surely. Lastly, let $h\in H$ be such that $\varphi_n(h)\ne 1$ $\omega$-almost-surely. It follows that $\theta_\infty \circ q(h)\ne 1$. Therefore $\alpha\circ \theta_\infty \circ q(h)\ne 1$ which implies that $\theta_\infty \circ \beta \circ q(h)\ne 1$. Hence $\theta_n\circ \beta \circ q(h)\ne 1$ $\omega$-almost-surely.
\end{proof}

\begin{rque}\label{elliptic_mod}
Note that since the action of $U$ on $T$ is elliptic, the modular automorphism $\alpha$ of $L$ restricts to conjugation on $U$. The proof of Proposition \ref{mod_approx} above implies that $\beta$ also restricts to conjugation on $U$.
\end{rque}

\begin{te}[The shortening argument]\label{shortening_arg} Suppose that $L$ does not split non-trivially over a finite subgroup of order $\le C$, then $\omega$-almost-surely the homomorphisms $\varphi_n$ are not short relative to $U$. More explicitly, denote by $\mathbb{R}_L$ the splitting of $L$ as a graph of actions outputted by the Rips machine. Then there is an $\mathbb{R}_L$-approximation $A$ of $L$ admitting a splitting $\mathbb{R}_A$, and an automorphism $\beta \in \mathrm{Aut}(A)$, for which the following holds: denote by $q$ the quotient map $H\twoheadrightarrow A$ and let $(\theta_n:A\rightarrow G)_{n \in \mathbb{N}}$ be such that $\varphi_n=\theta_n \circ q$ $\omega$-almost-surely. Then the sequence $(\phi_n=\theta_n\circ\beta\circ q:H\rightarrow G)_{n\in\mathbb{N}}$ satisfies the following:
    \begin{enumerate}
        \item ${\phi_n}_{\vert U}$ coincides with ${\varphi_n}_{\vert U}$ up to conjugation $\omega$-almost-surely;
        \item $\norm{\phi_n}<\norm{\varphi_n}$ $\omega$-almost-surely;
        \item for every $h\in H$ such that $\varphi_n(h)\ne 1$ $\omega$-almost-surely, $\phi_n(h)\ne 1$ $\omega$-almost-surely.
    \end{enumerate}
\end{te}

\begin{rque}
Note that conditions $\emph{(1)}$ and $\emph{(2)}$ above imply that $\varphi_n$ is not short relative to $U$ $\omega$-almost-surely. Furthermore, condition $\emph{(2)}$ can be equivalently phrased as follows: $\norm{\theta_n \circ \beta} < \norm{\theta_n}$ $\omega$-almost-surely, where the lengths are taken with respect to the set $q(S)$. Condition $\emph{(3)}$ is equivalent to each of the following two conditions:
\begin{enumerate}
    \item[(3')] $\underleftarrow{\ker}_\omega((\phi_n)_{n\in\mathbb{N}})\subset \underleftarrow{\ker}_\omega((\varphi_n)_{n\in\mathbb{N}})$;
    \item[(3'')] $\beta^{-1}(\underleftarrow{\ker}_\omega((\theta_n)_{n\in\mathbb{N}}))\subset \underleftarrow{\ker}_\omega((\theta_n)_{n\in\mathbb{N}})$.
\end{enumerate}
\end{rque}

\begin{rque}\label{single_ap}
The version of the shortening argument appearing in \cite[Theorem 5.29]{GH19} satisfies conditions $\emph{(1)}$ and $\emph{(2)}$ above, but does not necessarily satisfy condition $\emph{(3)}$. To obtain condition $\emph{(3)}$ we approximate a single modular automorphism of $L$, rather than a sequence of modular automorphisms, and the result follows from Lemma \ref{mod_approx}.
\end{rque}

The second version of the shortening argument that will be proved is a strengthened version of Theorem \ref{shortening_arg} which accommodates the use of JSJ decompositions of limit groups over finite groups of order less than $C$. Note that in this version of the shortening argument, we assume that a single vertex group of a JSJ decomposition of $L$ over $\leq C$ admits a splitting outputted by the Rips machine and we shorten the homomorphisms with respect to the generators of this vertex group.

\begin{te}\label{shortening_arg2}
    Let $\mathbb{J}_L$ be a reduced JSJ splitting of $L$ over finite groups of order $\le C$ relative to $U$ and let $u$ be the vertex fixed by $U$. Let $\mathbb{R}_{L_u}$ be the splitting of $L_u$ as a graph of actions outputted by the Rips machine, and let $\mathbb{RJ}_L$ be the splitting of $L$ obtained from $\mathbb{J}_L$ by replacing $u$ with $\mathbb{R}_{L_u}$. Let $H_u$ be a finitely generated subgroup of $H$ containing $U$ and such that $\varphi_{\infty}(H_u)=L_u$. Let $S_u$ be a finite generating of $H_u$. Then the following hold:
    \begin{enumerate}
        \item there exist an $\mathbb{RJ}_L$-approximation $A$ of $L$ admitting a splitting $\mathbb{RJ}_A$ as in Corollary \ref{comb_app} and a sequence $(\theta_n:A\rightarrow G)_{n \in \mathbb{N}}$ that satisfies $\varphi_n=\theta_n \circ q$ $\omega$-almost-surely (where $q$ is the quotient map $H\twoheadrightarrow A$);
        \item denote by $\mathbb{R}_{A_u}$ the subgraph of $\mathbb{RJ}_A$ corresponding to the subgraph $\mathbb{R}_{L_u}$ of $\mathbb{RJ}_L$. Denote by $A_u$ the fundamental group of $\mathbb{R}_{A_u}$. There exists an automorphism $\beta_u$ of $A_u$ that admits a natural extension $\beta \in \mathrm{Aut}(A)$, and such that the sequence $(\phi_n=\theta_n\circ\beta\circ q:H\rightarrow G)_{n\in\mathbb{N}}$ satisfies the following properties: 
        \begin{enumerate}
            \item ${\phi_n}_{\vert U}$ coincides with ${\varphi_n}_{\vert U}$ up to conjugation $\omega$-almost-surely;
            \item $\norm{\phi_{n \vert _{H_u}}}<\norm{\varphi_{n \vert _{H_u}}}$ $\omega$-almost-surely (where the lengths are taken with respect to the $S_u$);
            \item for every $h\in H$ such that $\varphi_n(h)\ne 1$ $\omega$-almost-surely, $\phi_n(h)\ne 1$ $\omega$-almost-surely.
        \end{enumerate}
    \end{enumerate}
\end{te}

\begin{rque}
As in Theorem \ref{shortening_arg}, the condition $\emph{(2)(b)}$ above is equivalent to \[\norm{(\theta_n \circ \beta)_{\vert _{q(H_u)}}} < \norm{\theta_{n\vert _{q(H_u)}}}.\]
\end{rque}

We need the following lemma in order to prove Theorems \ref{shortening_arg} and \ref{shortening_arg2}.

\begin{lemme}
    \label{center}
    Suppose that $L$ does not split non-trivially over a finite subgroup of order at most $C$, and that $e\subset T$ (where $T$ is the limiting tree on which $L$ acts) is an edge; denote its stabilizer by $L_e$. Then there is an element $\tilde{c}_e\in H$ whose image $c_e$ in $L$ is contained in  $Z(L_e)$ and such that $\varphi_n(\tilde{c}_e)$ is hyperbolic $\omega$-almost-surely.
\end{lemme}

To prove this lemma, we use the following result.

\begin{lemme}
    \cite[Lemma 3.5]{per08}
    \label{quasitrans}
    Let $(X,d)$ be a $\delta$-hyperbolic space and let $g:X\rightarrow X$ be an isometry. Suppose that for $x,y \in X$
    \begin{align*}
        d(x,g(x))+d(y,g(y))<2d(x,y)-4\delta.
    \end{align*}
    Then for some $\lambda\in \mathbb{R}$ such that $\abs{\lambda}\le \max \{d(x,g(x)),d(y,g(y))\}$ the isometry $g$ acts by $(\lambda,2\delta)$-\emph{quasi-translation} on a subgeodesic of $[x,y]$: for every $p\in[x,y]$ at distance greater than $\max \{d(x,g(x)),d(y,g(y))\}$ from both $x$ and $y$ (if such a point $p$ exists), one has
    \begin{align*}
        d(g(p),p_\lambda)<2\delta  \,\,\,\,\,  \text{ and } \,\,\,\,\, d(g^{-1}(p),p_{-\lambda})<2\delta
    \end{align*}
    where $p_\lambda$ and $p_{-\lambda}$ are the points on $[x,y]$ which lie at distance $\lambda$ from $p$.
\end{lemme}

\begin{proof}[Proof of Lemma \ref{center}]
    Since $L$ does not split non-trivially over a finite subgroup of order $\le C$, Lemma \ref{stabilitylemma} implies that there exists $c\in L_e$ of infinite order and such that $\varphi_n(\tilde{c})$ is hyperbolic $\omega$-almost-surely. Fix $\varepsilon=8\delta$ and let $N$ and $R$ be the corresponding acylindricity constants. We claim that $c_e=c^{N!}$ lies in $Z(L_e)$. Let $g\in L_e$, write $e=[x,y]$ and let $(x_n)_{n\in\mathbb{N}}$ and $(y_n)_{n\in\mathbb{N}}$ be approximating sequences for $x$ and $y$ respectively; let $\tilde{g}$ be a lift of $g$ to $H$. Since $g$ fixes both $x$ and $y$ in the limiting tree, $\varphi_n(\tilde{g})$ must displace both $x_n$ and $y_n$ by a distance which is significantly smaller than $d(x_n,y_n)$ $\omega$-almost-surely. More precisely, we have that
    \begin{align*}
        d(x_n,y_n) & > 100 \cdot R \\
        d(x_n,y_n) & > 100 \cdot \max \left\{ d(x_n,\varphi_n(\tilde{g})x_n),d(y_n,\varphi_n(\tilde{g})y_n) \right\} \\
        d(x_n,y_n) & > 100 \cdot \max \left\{ d(x_n,\varphi_n(\tilde{c})^jx_n),d(y_n,\varphi_n(\tilde{c})^jy_n) \right\}
    \end{align*}
    holds for all $j\in \llbracket 1,N+1 \rrbracket$ $\omega$-almost-surely. Choose two points $p_n,q_n\in e$ satisfying $d(x_n,p_n)<d(x_n,q_n)$, $d(p_n,q_n)>R$ and
    \begin{align*}
        \min \{ d(x_n,p_n), d(q_n,y_n) \} > 10 \cdot \left( \max_{h \in \{g,c,c^2,\ldots,c^{N+1}\}}  \left\{ d\left(x_n,\varphi_n(\tilde{h})x_n\right), d\left(y_n,\varphi_n(\tilde{h})y_n \right) \right\} \right)
    \end{align*}
    $\omega$-almost-surely. It follows from Lemma \ref{quasitrans} that each of $g,c,c^2,\ldots,c^{N+1}$ acts on a subsegment of $[x,y]$ which contains $p_n$ and $q_n$ by $2 \delta$-quasi-translation, and therefore both $d\left(p_n,\left[\varphi_n(\tilde{g}),\varphi_n(\tilde{c})^j\right]p_n\right) \leq 8 \delta$ and $d\left(q_n,\left[\varphi_n(\tilde{g}),\varphi_n(\tilde{c})^j\right]q_n\right) \leq 8 \delta$ hold for every $j \in \llbracket 1,N+1 \rrbracket$ $\omega$-almost-surely. The acylindricity condition implies that not all of the $N+1$ commutators can be distinct and there are $i,j\in\llbracket 1,N+1 \rrbracket$ for which
    \begin{align*}
        \left[\varphi_n(\tilde{g}),\varphi_n(\tilde{c})^j\right]  = &\varphi_n(\tilde{g}) \varphi_n(\tilde{c})^j\varphi_n(\tilde{g}) ^{-1} \varphi_n(\tilde{c})^{-j} \\ = & \varphi_n(\tilde{g}) \varphi_n(\tilde{c})^i\varphi_n(\tilde{g}) ^{-1} \varphi_n(\tilde{c})^{-i} \\ = &\left[\varphi_n(\tilde{g}),\varphi_n(\tilde{c})^i\right].
    \end{align*}
    In particular, $\varphi_n(\tilde{g})$ commutes with $\varphi_n(\tilde{c})^{j-i}$ $\omega$-almost-surely. Since $\abs{j-i} \le N$, the element $c_e=c^{N!}$ is a power of $c^{j-i}$ and $\varphi_n(\tilde{g})$ commutes with $\varphi_n(\tilde{c}_e)$ $\omega$-almost-surely; hence $g$ commutes with $c_e$.
\end{proof}

\begin{proof}[Proof of Theorem \ref{shortening_arg}]
    The proof of this theorem is divided in two: using results from \cite{RW14} and \cite{RS94} and explaining briefly how they adapt to our setting, we first find a suitable modular automorphism $\alpha \in \mathrm{Mod}^*_{\mathbb{R}_L}(L)$. Then, by means of Proposition \ref{mod_approx}, we find an $\mathbb{R}_L$-approximation $A$ of $L$ and an automorphism $\beta$ of $A$ which satisfy the desired properties.
    
    \smallskip
    
    The idea behind the proof amounts to finding a finite sequence of modular automorphisms of $L$, each of which shortens the actions of the generators $S$ of $H$ over $U$ with respect to the different vertex actions in the graph of actions decomposition $\mathbb{R}_L$ of $L$. The construction of the modular automorphism $\alpha \in \mathrm{Mod}^*_{\mathbb{R}_L}(L)$ relies on the proofs appearing in \cite{RW14} for the axial and Seifert-type cases, and on the proof appearing in \cite{RS94} for the simplicial case. We begin with vertex actions $L_v \curvearrowright T_v$ which admit dense orbits, namely axial and Seifert-type vertex actions. Recall that if $L_v$ is of Seifert-type, then the index of the $2$-orbifold subgroup of $L_v$ is at most $C$. Therefore, by \cite[Subsections 4.2.1 and 4.2.2]{RW14}, for every such vertex action and every finite subset $F\subset H$, there exists a modular automorphism $\alpha_v^F \in \mathrm{Mod}^*_{\mathbb{R}_L}(L)$ of type (3) or (4) which satisfies the following: denote by $o=(o_n)_{n\in\mathbb{N}}$ the base point of $T$, then for every $f \in F$,
    \begin{align*}
        d_T(o,\alpha_v^F(\varphi_\infty(f))o)<d_T(o,\varphi_\infty(f)o)
    \end{align*}
    whenever $[o,\varphi_\infty(f)o]$ has a non-degenerate intersection with a translate of $T_v$ in $T$, and $d_T(o,\alpha_v^F(\varphi_\infty(f))o)=d_T(o,\varphi_\infty(f)o)$ otherwise.
    
    \smallskip
    
    This allows us to shorten the actions (on the real tree $T$) of all the generators which intersect (a translate of) an axial or a Seifert-type component of $T$ non-degenerately: let $v_1,\cdots,v_m$ be an enumeration of the axial and Seifert-type vertices of $\mathbb{R}_L$; we define a sequence $(\alpha_1,\ldots,\alpha_m)\in \mathrm{Mod}^*_{\mathbb{R}_L}(L)^m$ iteratively. Let $\alpha_1=\alpha_{v_1}^{\varphi_\infty(S)}$, and after $\alpha_1,\ldots,\alpha_{i}$ were defined let
    \begin{align*}
        \alpha_{i+1}=\alpha_{v_{i+1}}^{\alpha_i\circ \cdots \alpha_1 (S)}.
    \end{align*}
    Now note that for every $s\in S$ such that $[o,\varphi_\infty(f)o]$ has a non-degenerate intersection with (a translate of) an axial or a Seifert-type component of $T$, we have that
    \begin{align*}
        d_T(o,\alpha_m\circ \cdots \circ \alpha_1 \circ \varphi_\infty(s)o) < d_T (o, \varphi_\infty(s)o).
    \end{align*}
    
    \smallskip
    
    Bring to mind that the modular automorphism $\alpha_m \circ \cdots \circ \alpha_1$ of $L$ does not necessarily shorten the actions of $S$ on $T$: it could be that for every $s \in S$, $[o,\varphi_\infty(s)o]$ is contained entirely in the discrete part of $T$, that is it does not intersect non-degenerately (translates of) axial and Seifert-type components of $T$. We therefore adapt \cite[Theorem 6.1]{RS94} to our settings. This theorem states that for every finite set $F \subset H$ there is a modular automorphism $\alpha^F_\mathrm{sim}\in \mathrm{Mod}^*_{\mathbb{R}_L}(L)$ of $L$, which can be written as a composition of Dehn twists about elements which lie in the edge groups of $T$, and which satisfies the following: for every $f\in F$ which does not fix $o$, and such that $[o,\varphi_\infty(f)o]$ lies entirely in the discrete part of $T$, let $f_{\alpha^F_\mathrm{sim}}$ be a lift of $\alpha^F_{\mathrm{sim}}(\varphi_\infty(f))$ to $H$; then
    \begin{align*}
        d\left(o_n,\varphi_n\left(f_{\alpha^F_\mathrm{sim}}\right)o_n\right) < d(o_n, \varphi_n(f)o_n),
    \end{align*}
    $\omega$-almost-surely. In addition, for every $f \in F$,
    \begin{align*}
        d_T(o, \alpha^F_\mathrm{sim}(\varphi_\infty(f))o)=d_T(o, \varphi_\infty(f)o).
    \end{align*}
    
    Note that in this case the modular automorphism $\alpha^F_\mathrm{sim}$ does not shorten the actions of the elements in $F$ on $T$, but rather shortens the actions of the elements in $F$ direclty on the spaces $X_n$.
    
    \smallskip
    
    In the proof of Theorem \cite[Theorem 6.1]{RS94}, one finds Dehn twists over the edge groups of $\mathbb{R}_L$, where each Dehn twist does not affect the displacement of $o$ by the elements of $F$ in $T$, but does affect the displacement of $o_n$ by $\varphi_n(F)$ in $X_n$. The construction of a Dehn twist over the stabilizer of an edge $e$ of $\mathbb{R}_L$ is divided into three cases:
        \begin{enumerate}
        \item $o$ lies in the interior of $e$ and $L$ splits as an amalgamated product over the stabilizer of $e$;
        \item $o$ lies in the interior of $e$ and $L$ splits as an HNN extension over the stabilizer of $e$;
        \item $o$ does not lie in the interior of $e$, and is one of its vertices.
    \end{enumerate}
    In the last case, one considers all edges in $T$ which are adjacent to $o$ and shortens the action of the generators with respect to all of these edges simultaneously. The proof appearing in \cite{RS94} can be transitioned almost seamlessly to our setting. In \cite{RS94}, the stabilizers of edges in the limiting tree are cyclic, whereas in our case they are virtually abelian. Therefore, in our case we can not take Dehn twists by any element in an edge group $L_e$ of $\mathbb{R}_L$, and take Dehn twists by a power of an element $c_e\in Z(L_e)$, of infinite order, which exists by Lemma \ref{center}. The only other parts which do not carry over to our settings are Lemmas 6.2, 6.5, 6.8 and 6.11 in \cite{RS94} which assert that there are elements of infinite order in the edge groups of $\mathbb{R}_L$ which satisfy the following: let $\tilde{c}\in H$ be a lift of such an element, then $\varphi_n(\tilde{c})$ displaces certain points in $X_n$ by a distance bounded from below by $10 \delta$ or $20 \delta$ $\omega$-almost-surely. These are the elements by which one takes the Dehn twists. We can easily overcome this: by \cite{Bow08}, there is $\eta > 0$ such that
    \begin{align*}
        \eta < \ell (g) = \lim_{n\rightarrow \infty} \frac{1}{n}d(g^no,o)
    \end{align*}
    for every hyperbolic $g \in G$. In addition, denote $\norm{g}=\inf_{x \in X} \{d(x,gx)\}$ and by \cite[Lemma 10.6.4]{CDP90} we have that $\ell(g)\leq \norm{g}$ and clearly $\ell(g^n)=n \ell(g)$. Therefore, given $D>0$ there exists $N$ such that
    \begin{align*}
        D < N\eta < N \ell (g) = \ell (g^N) \leq \norm{g^N} = \inf_{x \in X} \{d(x,g^N x)\}
    \end{align*}
    for every hyperbolic $g \in G$. We also remark that one might have to enlarge the constant $C_0$ appearing in \cite{RS94} to accommodate with the choice of $N$ above. This implies that the proof of \cite[Theorem 6.1]{RS94} can be carried out in our setting. 
    
    \smallskip
    Now let
    \begin{align*}
        \alpha=\alpha_{\mathrm{sim}}^{\alpha_m\circ\cdots\circ\alpha_1\circ\varphi_\infty(S)}\circ\alpha_m\circ\cdots\circ\alpha_1 \in \mathrm{Mod}^*_{\mathbb{R}_L}(L).
    \end{align*}
    For every $s\in S$, denote by $s_\alpha \in H$ a lift of $\alpha(\varphi_\infty(s))$ to $H$. Since $d_T(o,\alpha(\varphi_\infty(s))o)<d_T(o,\varphi_\infty(s)o)$ whenever $[o,\varphi_\infty(s)o]$ intersects a translate of an axial or a Seifert type component of $T$ non-degenerately, the following holds $\omega$-almost-surely:
    \begin{align*}
        d(o_n,\varphi_n(s_\alpha)o_n)<d(o_n,\varphi_n(s)o_n).
    \end{align*}
    By Proposition \ref{mod_approx} there exists an $\mathbb{R}_L$-approximation $A$ of $L$ and $\beta\in\mathrm{Aut}(A)$ such that the following diagram commutes $\omega$-almost-surely.
    \begin{align*}
        \xymatrix{H \ar[rrrrrr]^{\varphi_n} \ar@{->>}[drr]^q \ar@{->>}[ddrr]_{\varphi_\infty} && && && G \\
        && A \ar[rr]_\beta \ar@{->>}[d]^{\theta_\infty} \ar[urrrr]^{\theta_n}&& A \ar@{->>}[d]^{\theta_\infty} && \\
        && L \ar[rr]^\alpha && L &&
        }
    \end{align*}
    We claim that the approximation $A$ and its automorphism $\beta$ satisfy the properties described in the theorem.
    
    \smallskip
    
    By Remark \ref{elliptic_mod}, since the action of $U$ on $T$ is elliptic, $\beta$ restricts to conjugation on $U$. Hence condition $\emph{(1)}$ holds. For $\emph{(2)}$, note that for every $s\in S$, $\alpha(\varphi_\infty(s))$ and $\beta(q(s))$ share the same lift in $H$; setting $\varphi_n=\theta_n\circ \beta\circ q$, it follows that
    \begin{align*}
        d(o_n,\theta_n\circ\beta\circ q(s)o_n)=d(o_n,\phi_n(s)o_n)<d(o_n,\varphi_n(s)o_n),
    \end{align*}
    and in particular, since $o_n$ realizes the infimum $\inf_{x \in X} \max_{s \in S} d(x,\varphi_n(s)x)$,
    \begin{align*}
        \norm{\phi_n} &= \inf_{x \in X} \max_{s \in S} d(x,\phi_n(s)x) \\
        & \le \max_{s \in S} d(o_n,\phi_n(s)o_n) \\
        & < \max_{s \in S} d(o_n,\varphi_n(s)o_n) \\
        &= \norm{\varphi_n}
    \end{align*}
    $\omega$-almost-surely. Last, Property $\emph{(3)}$ follows directly from Proposition \ref{mod_approx}.
\end{proof}

The proof of Theorem \ref{shortening_arg2} is very similar to the proof of Theorem \ref{shortening_arg}.
\begin{proof}[Proof of Theorem \ref{shortening_arg2}]
    The proof is identical to that of Theorem \ref{shortening_arg}, with one change (applied to both Theorem \ref{shortening_arg} and Proposition \ref{mod_approx}): one has to take natural extensions of automorphisms of axial and Seifert-type vertex group with respect to the entire graph of groups decomposition $\mathbb{RJ}_L$ of $L$ (and not with respect to a graph of actions outputted by the Rips machine) instead of modular automorphisms of $L$ with respect to $\mathbb{R}_L$.
\end{proof}

\section{Test sequences}

The goal of this section is to define test sequences and prove important preliminary results about these sequences that will play a crucial role in the proof of Merzlyakov's theorem \ref{th0bis}.

\subsection{Transverse covering}

We will use the following definitions (see \cite{Gui04}, Definitions 4.6 and 4.8).

\begin{de}\label{transverse}Let $T$ be a real tree endowed with an action of a group $G$, and let $(Y_j)_{i\in J}$ be a $G$-invariant family of non-degenerate closed subtrees of $T$. We say that $(Y_j)_{j\in J}$ is a transverse covering of $T$ if the following two conditions hold.
\begin{itemize}
\item[$\bullet$]\emph{Transverse intersection:} if $Y_i\cap Y_j$ contains more than one point, then $Y_i=Y_j$.
\item[$\bullet$]\emph{Finiteness condition:} every arc of $T$ is covered by finitely many $Y_j$.
\end{itemize}
\end{de}

\begin{de}\label{squelette}Let $T$ be a real tree, and let $(Y_j)_{j\in J}$ be a transverse covering of $T$. The \emph{skeleton} of this transverse covering is the bipartite simplicial tree $S$ defined as follows:
\begin{enumerate}
\item $V(S)=V_0(S)\sqcup V_1(S)$ where $V_1(S)=\lbrace Y_j \ \vert \ j\in J\rbrace$ and $V_0(S)$ is the set of points $x\in T$ that belong to at least two distinct subtrees $Y_i$ and $Y_j$. 
\item There is an edge $\varepsilon=(Y_j,x)$ between $Y_j\in V_1(S)$ and $x\in V_0(S)$ if and only if $x$, viewed as a point of $T$, belongs to $Y_j$, viewed as a subtree of $T$.
\end{enumerate}
\end{de}

\begin{rque}The stabilizer of a vertex of $S$ is the stabilizer $G_{Y_i}$ or $G_e$ of the corresponding subtree or point of $T$. The stabilizer of an edge $\varepsilon=(Y_j,x)$ is $G_{Y_i}\cap G_x$. Moreover, the action of $G$ on $S$ is minimal provided that the action of $G$ on $T$ is minimal (see \cite{Gui04} Lemma 4.9).\end{rque}

\subsection{Bounding the number of branch points}

Recall (see Definition \ref{recall_definition_stable}) that a group action on a real tree is stable if any non-degenerate arc contains a non-degenerate stable subarc. We need to strengthen this definition.

\begin{de}\label{Msuper}An action on a real tree is $K$\emph{-superstable} if every arc whose pointwise stabilizer has order greater than $K$ is stable.\end{de}

Let $T$ be a real tree, and let $x$ be a point of $T$. A \emph{direction} at $x$ is a connected component of $T\setminus\lbrace x\rbrace$. We say that $x$ is a \emph{branch point} if there are at least three directions at $x$. The following result is a work in preparation by Guirardel and Levitt (improving \cite{Gui01}).

\begin{te}\label{guilev}Let $L$ be a group acting on a real tree $T_L$. Suppose that $L$ is finitely generated relative to a subgroup $G$ elliptic in $T_L$. Suppose that the action is $K$-superstable for some constant $K$, with finitely generated arc stabilizers. Then every point stabilizer is finitely generated relative to $G$, the number of orbits of branch points in $T_L$ is finite, the number of orbit of directions at branch points in $T_L$ is finite.\end{te}

\subsection{Small cancellation condition}\label{Coulon}Let $(X,d)$ be a $\delta$-hyperbolic simplicial graph, let $G$ be a group acting on $(X,d)$ by isometries, and let $g$ be an element of $G$. We define the \emph{translation length} of $g$ by $\vert\vert g\vert\vert=\inf_{x\in X}d(x,gx)$. If $g$ is hyperbolic, the \emph{quasi-axis} of $g$, denoted by $A(g)$, is the union of all geodesics joining $g^{-}$ and $g^{+}$. By Lemma 2.26 in \cite{Cou13}, the quasi-axis $A(g)$ is 11$\delta$-quasi-convex. If $g'$ is another hyperbolic element of $G$, one defines the \emph{fellow traveling constant} $\Delta(g,g')$ as follows: \[\Delta(g,g')=\mathrm{diam}\left(A(g)^{+100\delta}\cap A(g')^{+100\delta}\right)\in\mathbb{N}\cup\lbrace\infty\rbrace,\] where $A(g)^{+100\delta}$ is the $100\delta$-neighbourhood of $A(g)$ in $(X,d)$, and $A(g')^{+100\delta}$ is defined similarly. Recall that if $G$ acts acylindrically on $(X,d)$, then every hyperbolic element $g$ is contained in a unique maximal infinite virtually cyclic subgroup $\Lambda (g)$ of $G$ (see \cite{DGO17}, Lemma 6.5). Moreover, there exists a constant $ N(g) \geq 0 $ such that every element $h \in G $ satisfying $\Delta (g,hgh^{-1}) \geq N(g) $ belongs to $\Lambda (g)$ (see for example \cite{Cou16}). In addition, if $g$ and $h$ are hyperbolic, then either $\Lambda (g)=\Lambda (h)$ or $\Lambda (g)\cap \Lambda (h)$ is finite.

\begin{de}\label{SCC0}Let $\varepsilon>0$. We say that a hyperbolic element $g\in G$ satisfies the \emph{$\varepsilon$-small cancellation condition} if the following holds: for every $h\in G$, if \[\Delta(g,hgh^{-1}) > \varepsilon \vert \vert g\vert \vert,\] then $h$ and $g$ commute (so $h$ belongs to $\Lambda (g)$). In particular, $g$ is central in $\Lambda (g)$.\end{de}

\begin{de}\label{SCC}Let $\varepsilon>0$. We say that a tuple $(g_1,\ldots,g_p)\in G^p$ of hyperbolic elements satisfies the \emph{$\varepsilon$-small cancellation condition} if the following condition holds: for every $h\in G$, for every $(i,j)\in\llbracket 1,p\rrbracket^2$, if \[\Delta(g_i,hg_jh^{-1}) > \varepsilon \min(\vert \vert g_i\vert \vert,\vert\vert g_j\vert\vert),\] then $i=j$, and the elements $h$ and $g_i$ commute. In particular, $h$ belongs to $\Lambda (g_i)$. As a consequence, $g_i$ is central in $\Lambda (g_i)$ and for every $(i,j)\in\llbracket 1,p\rrbracket^2$, we have $\Lambda (g_i)\neq \Lambda (g_j)$, i.e.\ $\Lambda (g_i)\cap \Lambda (g_j)=E(G)$.\end{de}

\subsection{Preliminary lemmas}

Let $G$ be an acylindrically hyperbolic group. By \cite[Theorem 1.2]{Osi16}, there exists a generating set $S$ of $G$ such that the Cayley graph $X$ of $G$ with respect to $S$ is $\delta$-hyperbolic, for some $\delta\geq 0$, and the natural action of $G$ on $X$ is acylindrical and non-elementary. We denote by $d$ the word metric on $X$ associated with $S$, and by $1$ the neutral of $G$, viewed as a point of $X$. Given a hyperbolic element $g\in G$, recall that $\vert\vert g\vert\vert$ denotes the translation length of $g$, and that $A(g)$ denotes the quasi-axis of $g$.

\begin{de}\label{suitetest}Let $G$ be an acylindrically hyperbolic group, and let $\bm{a}$ be a tuple of elements of $G$. Fix a presentation $\langle \bm{a} \ \vert \ R(\bm{a})=1\rangle$ for the subgroup of $G$ generated by $\bm{a}$. Let $\Sigma(\bm{x},\bm{y},\bm{a})=1$ be a finite system of equations over $G$, where $\bm{x}$ and $\bm{y}$ are tuples of variables. Denote $G_{\Sigma}=\langle \bm{x},\bm{y},\bm{a} \ \vert \ R(\bm{a})=1, \Sigma(\bm{x},\bm{y},\bm{a})=1\rangle$. Let $p=\vert \bm{x}\vert$ be the arity of $\bm{x}$, and let $x_i$ denote the $i$th component of $\bm{x}$. Let $(\sigma_1,\ldots,\sigma_p)$ be a $p$-tuple of elements of $\mathrm{Aut}_G(E(G))$. A sequence of homomorphisms $(\varphi_n : G_{\Sigma} \rightarrow G)_{n\in\mathbb{N}}$ is called a $(\sigma_1,\ldots,\sigma_p)$-\emph{test sequence} if the following four conditions hold.
\begin{enumerate}
\item For every integer $n$, the morphism $\varphi_n$ coincides with a conjugation on $\bm{a}$.\smallskip
\item For every integer $i\in\llbracket 1,p\rrbracket$, the translation length $\vert \vert \varphi_n(x_i)\vert \vert$ of $\varphi_n(x_i)$ tends to infinity as $n$ tends to infinity. 
\smallskip
\item For every $(i,j)\in\llbracket 1,p\rrbracket^2$, there exists a real number $r_{i,j}\in [0,+\infty]$ such that the ratio \[\frac{\vert \vert \varphi_n(x_i)\vert \vert}{\vert \vert \varphi_n(x_j)\vert \vert}\]tends to $r_{i,j}$ as $n$ tends to infinity.
\smallskip
\item There exists a sequence of positive real numbers $(\varepsilon_n)_{n\in\mathbb{N}}$ converging to $0$ such that, for every integer $n$, the tuple $\varphi_n(\bm{x})$ satisfies the $\varepsilon_n$-small cancellation condition (see Definition \ref{SCC}), and the following equality holds, for every integer $1\leq i\leq p$: \[\Lambda (\varphi_n(x_i))=\langle \varphi_n(x_i), E(G) \ \vert \ \mathrm{ad}(\varphi_n(x_i))_{\vert E(G)}=\sigma_i\rangle.\]In particular, the image of $\varphi_n(x_i)$ in $\Lambda (\varphi_n(x_i))/E(G)$ has no roots.
\smallskip
\end{enumerate}
In the particular case where $(\sigma_1,\ldots,\sigma_p)=(\mathrm{id}_{E(G)},\ldots,\mathrm{id}_{E(G)})$, one simply says that $(\varphi_n)_{n\in\mathbb{N}}$ is a \emph{test sequence}.
\end{de}


\begin{rque}\label{impor}Let $(\varphi_n : G_{\Sigma} \rightarrow G)_{n\in\mathbb{N}}$ be a $(\sigma_1,\ldots,\sigma_p)$-test sequence. Let $U$ be the subgroup of $G_{\Sigma}$ generated by $\bm{x}$ and $\bm{a}$. Since the translation length $\vert\vert \cdot\vert\vert $ is constant on conjugacy classes, one easily sees that any sequence $(\theta_n : G_{\Sigma} \rightarrow G)_{n\in\mathbb{N}}$ such that $\theta_n$ coincides on $U$ with $\varphi_n$ up to conjugation is also a $(\sigma'_1,\ldots,\sigma'_p)$-test sequence for some $(\sigma'_1,\ldots,\sigma'_p)\in\mathrm{Aut}_G(E(G))^p$.\end{rque}

\begin{rque}Note that any subsequence of a $(\sigma_1,\ldots,\sigma_p)$-test sequence is a $(\sigma_1,\ldots,\sigma_p)$-test sequence as well.
\end{rque}

The following easy lemma will be useful in the sequel.

\begin{lemme}\label{intersection}Let $(\varphi_n)_{n\in\mathbb{N}}$ be a $(\sigma_1,\ldots,\sigma_p)$-test sequence. For every infinite subset $A\subset \mathbb{N}$ and every integer $1\leq i\leq p$, we have \[\bigcap_{n\in A}\Lambda (\varphi_n(x_i))=E(G).\] 
\end{lemme}

\begin{proof}Suppose that $g$ belongs to $\Lambda (\varphi_n(x_i))$ for every $n\in A$. Then, there exists an integer $k_n$ and an element $g_n\in E(G)$ such that $g=\varphi_n(x_i)^{k_n}g_n$, for every $n\in A$. Now, observe that $k_n$ must be equal to $0$ for every $n$ large enough, otherwise (up to extracting a subsequence) $\vert\vert \varphi_n(x_i)^{k_n}\vert\vert$ goes to infinity, and so does the constant $\vert\vert g\vert\vert$, which is a contradiction. It follows that $g$ belongs to $E(G)$.\end{proof}

Let $G$ be an acylindrically hyperbolic group. We keep the same notations as above. Let $(\varphi_n: G_{\Sigma} \rightarrow G)_{n\in\mathbb{N}}$ be a $(\sigma_1,\ldots,\sigma_p)$-test sequence. Let $L$ be the quotient of $G_{\Sigma}$ by the stable kernel of the sequence $(\varphi_n)_{n\in\mathbb{N}}$, and let $\varphi_{\infty} : G_{\Sigma}\twoheadrightarrow L$ be the corresponding epimorphism. Let $S$ be a finite generating set of $L$ containing the images of $\bm{x}$ and $\bm{a}$ in $L$ (still denoted by $\bm{x}$ and $\bm{a}$) and let $(X,d)$ be a hyperbolic Cayley graph of $G$ on which $G$ acts acylindrically and non-elementarily. Let $\lambda_n=\inf_{x\in X}\mathrm{max}_{s\in S} \ d(x,\varphi_n(s)x)$ be the displacement (or scaling factor) of $\varphi_n$. Let us consider the rescaled metric $d_n=d/\lambda_n$, let $\omega$ be a non-principal ultrafilter and let $(X_{\omega},d_{\omega})$ be the ultralimit of $((X,d_n))_{n\in\mathbb{N}}$. Classically, $X_{\omega}$ is a real tree and there exists a unique minimal $L$-invariant non-degenerate subtree $T_L\subset X_{\omega}$. Moreover, some subsequence of the sequence $((X,d_n))_{n\in\mathbb{N}}$ converges to $T_L$ in the equivariant Gromov-Hausdorff topology. Let $o$ denote the limit of the sequence $(1)_{n\in\mathbb{N}}$, where $1$ is the neutral of $G$ (viewed as a point of the Cayley graph $X$). This point $o$ is called the base point of $T_L$.

\begin{lemme}\label{fini}Define $K=\lbrace g\in G_{\Sigma} \ \vert \ \varphi_n(g)\in E(G) \ \omega\text{-surely}\rbrace$ and $F=\varphi_{\infty}(K)$. This group $F$ is finite. Moreover, it is the unique maximal finite subgroup of $L$ normalized by $\varphi_{\infty}(x_i)$, for every component $x_i$ of $\bm{x}$, with $1\leq i\leq p$.\end{lemme}

\begin{proof}First, we prove the existence of a unique maximal finite subgroup of $L$ normalized by $\varphi_{\infty}(x_i)$, denoted by $F_i$. Let $\lbrace F_{i,j}\rbrace_{j\in J}$ be the collection of all finite subgroups of $L$ that are normalized by $\varphi_{\infty}(x_i)$, let $A_i=\cup_{j\in J}F_{i,j}$ and let $F$ be the subgroup of $L$ generated by $A_i$. Let $K_i=\varphi_{\infty}^{-1}(F_i)$. We claim that $\varphi_n(K_i)$ is contained in $E(G)$ $\omega$-almost-surely. Let $k\in K_i$ be a preimage of an element of $A_i$. By definition of $A_i$, $\varphi_{\infty}(k)$ is contained in a finite subgroup of $L$ normalized by $\varphi_{\infty}(x_i)$. As a consequence, there exists an integer $m\geq 1$ such that $\varphi_{\infty}([k,x_i^m])=1$. It follows that $\varphi_n([k,x_i^m])$ is trivial $\omega$-almost-surely. By Lemma \ref{lemmefin}, this implies that $\varphi_n(k)$ belongs to $\Lambda (\varphi_n(x_i))$ $\omega$-almost-surely. Now, recall that $\Lambda (\varphi_n(x_i))$ is generated by $\varphi_n(x_i)$ and $E(G)$ by definition of a test sequence. In particular, $\Lambda (\varphi_n(x_i))$ is $E(G)$-by-$\mathbb{Z}$, which proves that $\varphi_n(k)$ is contained in $E(G)$, since $\varphi_n(k)$ has finite order, by definition of $A_i$. Hence,  $\varphi_n(K_i)$ is contained in $E(G)$ $\omega$-almost-surely. It follows that $F_i=\varphi_{\infty}(K_i)$ is finite. Moreover, by construction, $F_i$ is the unique maximal finite subgroup of $L$ normalized by $\varphi_{\infty}(x_i)$.

\smallskip

Let $K$ be the subgroup \[\lbrace g\in G_{\Sigma} \ \vert \ \varphi_n(g)\in E(G) \ \omega\text{-surely}\rbrace\] of $G_{\Sigma}$. We claim that $K_i$ and $K$ coincide (in particular, $K_i$ does not depend on $i$). First, note that the inclusion $K_i\subset K$ was proved in the previous paragraph. Then, let $k$ be an element of $K$, and let us prove that $k$ belongs to $K_i$. Let us define a subgroup $K'_i$ of $G_{\Sigma}$ as follows: $K'_i=\langle \lbrace x_i^{\ell}kx_i^{-\ell}, \ \ell\in\mathbb{N}\rbrace\rangle$. By definition of $K$, the element $\varphi_n(k)$ belongs to $E(G)$ $\omega$-almost-surely. It follows that $\varphi_n(K'_i)$ is a subgroup of $E(G)$ $\omega$-almost-surely. In addition, this subgroup is normalized by $\varphi_n(x_i)$ by construction. As a consequence, $\varphi_{\infty}(K'_i)$ is contained in $F_i$. In particular, $\varphi_{\infty}(k)$ belongs to $F_i$, which implies that $k$ belongs to $K_i$. Hence, one has $K_i=K$ for any integer $1\leq i\leq p$. As a consequence, the finite groups $F_1, \ldots, F_p$ are equal.\end{proof}

In the sequel, we abuse notation and denote by $x_i$ both the element of $G_{\Sigma}$ and its image in $L$ under $\varphi_{\infty}$.

\begin{lemme}\label{lemme2}We keep the same notations and assumptions as above. Let $S\subset L$ be the stabilizer of the base point $o$. Note that $S$ contains each element of the tuple $\bm{a}$. Indeed, each $\varphi_n$ restricts to a conjugation on $\bm{a}$. Suppose that the subgroup $\Gamma:=\langle \bm{x}\cup S\rangle$ of $L$ does not fix a point in $T_L$. Then the minimal subtree $ T_{\Gamma} \subset T_L$ of $\Gamma$ is simplicial, and $\Gamma$ admits a splitting of the form \[\Gamma=\langle \bm{x},S \ \vert \ \mathrm{ad}(x_i)_{\vert F}=\alpha_i, \forall i\in\llbracket 1,p\rrbracket\rangle,\]where $F$ denotes the finite subgroup of $L$ defined in the previous lemma, and $\alpha_i$ denotes the automorphism of $F$ induced by the action of $x_i$ on $F$.\end{lemme}


\begin{proof}Suppose that $\Gamma$ does not fix a point of $T_L$. Let $ T_{\Gamma} \subset T_L$ be the minimal subtree of $\Gamma$. Note that there exists an integer $1\leq i\leq p$ such that $\vert \vert \varphi_n(x_i)\vert\vert/\lambda_n$ does not approach $0$ as $n$ goes to infinity. Otherwise $\Gamma$, which is generated by $S$ and $\bm{x}$, would be elliptic in $T_L$. 

\smallskip

Moreover, by the third condition of Definition \ref{suitetest}, $\vert \vert \varphi_n(x_i)\vert\vert/\vert \vert \varphi_n(x_j)\vert\vert$ tends to a real number $r_{i,j}\in [0,+\infty]$, for every integers $1\leq i,j\leq p$. Consequently, each $x_i$ acts hyperbolically on $T_{\Gamma}$; we denote by $\ell_i$ the limit of the sequence $(\vert \vert \varphi_n(x_i)\vert\vert/\lambda_n)_{n\in\mathbb{N}}$, for every $1\leq i\leq p$. Note that one has $0<\ell_i<+\infty$.


\smallskip

\textbf{Claim 1:} for every integers $1\leq i,j\leq p$ and for every $s\in S$, if the intersection of the axes of $x_j$ and $sx_is^{-1}$ contains two distinct points $v$ and $w$, then $i=j$ and $s$ belongs to $F$.

\smallskip

Let us prove this claim. By assumption, the segment $[v,w]$ is contained in the intersection of the axes of $sx_is^{-1}$ and $x_j$. Let $\eta$ be the length of $[v,w]$ in the limiting tree $T_{\Gamma}$. Let $\bar{s}$ be a preimage of $s$ in $G_{\Sigma}$. The overlap $\Delta(\varphi_n(\bar{s}x_i\bar{s}^{-1}),\varphi_n(x_j))$ is comparable to $\eta\lambda_n$ when $n$ is large. As a consequence, $\omega$-almost-surely, the following inequality holds: \[\Delta(\varphi_n(\bar{s}x_i\bar{s}^{-1}),\varphi_n(x_j))\geq \eta\lambda_n/2 .\]Moreover, the translation length $\vert\vert \varphi_n(x_i)\vert\vert$ is equivalent to $\ell_i\lambda_n$. Thus, $\omega$-almost-surely, one has \[\Delta(\varphi_n(\bar{s}x_i\bar{s}^{-1}),\varphi_n(x_j))\geq \frac{\eta}{4\ell_i} \vert\vert \varphi_n(x_i)\vert\vert.\]According to the fourth condition of Definition \ref{suitetest}, the tuple $(\varphi_n(x_1),\ldots,\varphi_n(x_p))$ satisfies the $\varepsilon_n$-small cancellation condition for some sequence of positive real numbers $(\varepsilon_n)_{n\in\mathbb{N}}$ converging to $0$; $\omega$-almost-surely, $\varepsilon_n$ is smaller than $\eta/(4\ell_i)$. As a consequence, the integers $i$ and $j$ are equal, and $\varphi_n(\bar{s})$ is contained in $\Lambda (\varphi_n(x_i))=E(G)\rtimes \langle\varphi_n(x_i)\rangle$. But since $s$ belongs to $S$, the translation length of $\varphi_n(\overline{s})$ is negligible compared to that of $\varphi_n(x_i)$ $\omega$-almost-surely, so $s=\varphi_{\infty}(\bar{s})$ belongs to $F=\varphi_{\infty}(H)$, where $H$ is the subgroup from Lemma \ref{fini}.

\smallskip

\textbf{Claim 2:} for every $1\leq i\leq p$, the pointwise stabilizer $Z_i$ of the axis $A(x_i)$ of $x_i$ in $T_{\Gamma}$ is equal to $\langle x_i,F\rangle$. Moreover, $A(x_i)$ is transverse to its translates, which means that for every $\gamma\in\Gamma\setminus Z_i$, the intersection of $A(x_i)$ with $A(\gamma x_i\gamma^{-1})$ is empty or reduced to a point.

\smallskip

Let us prove the second claim. Consider an element $\gamma\in\Gamma$ such that the intersection of $A(x_i)$ with $A(\gamma x_i\gamma^{-1})$ contains two distinct points $v$ and $w$. Let $\bar{\gamma}$ be a preimage of $\gamma$ in $G_{\Sigma}$. As in the proof of Claim 1, $\varphi_n(\bar{\gamma})$ is contained in $\Lambda (\varphi_n(x_i))=E(G)\rtimes \langle\varphi_n(x_i)\rangle$. Hence, for every $n$, there exists an integer $p_n$ such that $\varphi_n(\bar{\gamma}x_i^{p_n})$ belongs to $E(G)$ $\omega$-almost-surely. Note that the sequence $(p_n)_{n\in\mathbb{N}}$ is bounded $\omega$-almost-surely, since $\vert \vert \varphi_n(x_i)\vert\vert/\lambda_n$ stays away from $0$ $\omega$-almost-surely. In particular, $p_n$ is constant $\omega$-almost-surely, equal to a certain integer $p$. It follows that $\gamma x_i^p$ fixes the basepoint $o$. In other words, $\gamma$ is equal to $sx_i^{-p}$ for some element $s\in S$. Now, observe that one has $\gamma x_i\gamma^{-1}=sx_is^{-1}$. Since the intersection of $A(x_i)$ with $A(\gamma x_i\gamma^{-1})$ is neither empty nor reduced to a point, Claim 1 implies that $s$ belongs to $F$. This shows that $Z_i=\langle x_i,F\rangle$ and that $A(x_i)$ is transverse to its translates, which completes the proof of Claim 2.

\smallskip

Now, let us consider the union $T$ of all the $\Gamma$-translates of the axes of $x_1,\dots,x_p$. For every integer $i$, the intersection $x_iT\cap T$ is non-empty by definition of $T$. It follows that $T$ is connected, i.e.\ $T$ is a subtree of $T_{\Gamma}$. In addition, $T$ is $\Gamma$-invariant by definition. By minimality of $T_{\Gamma}$, one has $T=T_{\Gamma}$.

\smallskip

Let us prove that $T_{\Gamma}$ is a simplicial tree. Let $v$ be a point on the axis of $x_i$ in $T_{\Gamma}$, and let $e=[v,x_iv]$ be the segment of $T_{\Gamma}$ with endpoints $v$ and $x_iv$. Let us prove that there are only finitely many branch points on $e$ in $T_{\Gamma}$. By Lemma 4.7 and Theorem 4.18 of \cite{GH19}, the hypotheses of Theorem \ref{guilev} are satisfied. It follows from this theorem that the number of orbits of directions at branch points in $T_{\Gamma}$ is finite. Now, assume towards a contradiction that there are infinitely many branch points on $e$. Then there exist necessarily two non-degenerate subsegments $I$ and $J$ in $e$, with $I\cap J=\varnothing$, and an element $\gamma\in \Gamma$ such that $\gamma I=J$. But we just proved that the axis of $x_i$ in $T_{\Gamma}$ is transverse to its translates (Claim 2 above). Thus, $\gamma$ belongs to the stabilizer of the axis of $x_i$, namely $\langle x_i, F\rangle$. This is a contradiction since $F$ fixes the axis of $x_i$ pointwise, and since the intersection of $e$ with $x_ie$ is reduced to the endpoint $x_iv$ of $e$.

\smallskip

It follows from the previous description of $T_{\Gamma}$ that there exist some conjugates $\gamma_1x_1\gamma_1^{-1}$, $\ldots,$ $\gamma_px_p\gamma_p^{-1}$ of $x_1,\ldots,x_p$ whose axes intersect at the basepoint $o$. Last, since the pointwise stabilizer of the axis of $x_i$ is $F$ for every $1\leq i\leq p$, and since $ \bm{x}$ and $S$ generate $\Gamma$ by definition, $\Gamma$ admits a splitting of the form \[\Gamma=\langle \bm{x},S \ \vert \ \mathrm{ad}(x_i)_{\vert F}=\alpha_i, \forall i\in\llbracket 1,p\rrbracket\rangle,\]where $\alpha_i$ denotes the automorphism of $F$ induced by the action of $x_i$ on $F$ by conjugation.\end{proof}

\begin{co}\label{lemme22}With the same notations and the same hypotheses as in Lemma \ref{lemme2} above, the tree $T_{\Gamma} $ is transverse to its translates, i.e.\ for every element $ h \in L \setminus \Gamma $, the intersection $hT_{\Gamma}\cap T_{\Gamma}$ is at most one point. In addition, if $e$ is an edge of $T_{\Gamma}$, there are only finitely many branch points on $e$ in $T_L$.\end{co}

\begin{proof}Let $h$ be an element of $L$ such that $hT_{\Gamma}\cap T_{\Gamma} $ is non-degenerate. As a consequence of the description of $ T_{\Gamma} $ above, we can find two elements $u, v \in \Gamma $ such that the axes of $ ux_iu^{-1} $ and $ h(vx_jv^{-1})h^{-1} $ have a non-trivial overlap in the limiting tree $T_L$, for some $1\leq i,j\leq p$, possibly equal. We denote by $\bar{u}, \bar{v},\bar{h}$ three preimages of $u,v,h$ in $G_{\Sigma}$. As in the previous proof, we have \[\Delta(\varphi_n(x_i),{\varphi_n(\bar{u}^{-1}\bar{h}\bar{v})}\varphi_n(x_j){\varphi_n(\bar{u}^{-1}\bar{h}\bar{v})}^{-1})\geq \varepsilon_n \min(\vert \vert \varphi_n(x_i)\vert \vert,\vert \vert \varphi_n(x_j)\vert \vert)\] $\omega$-almost-surely. Hence, the integers $i$ and $j$ are equal, and $ \varphi_n (\bar{u}^{- 1} \bar{h}\bar{v}) $ belongs to the group $\Lambda (\varphi_n(x_i)) = E(G) \times \langle \varphi_n(x_i) \rangle $. So, for every $ n $, there is an integer $ p_n $ such that $\varphi_n(\bar{u}^{-1}\bar{h}\bar{v}x_i^{p_n})$ belongs to $E(G)$. On the other hand, since $x_i$ acts hyperbolically on $T_{\Gamma}$, the integer $ p_n $ is bounded by a constant that does not depend on $ n $. Otherwise, $\vert\vert\varphi_n(x_i)\vert\vert/\vert\vert \varphi_n(u^{-1}hv)\vert\vert$ $\omega$-tends to $0$. Hence, since $\varphi_n(u^{-1}hv)/\lambda_n$ is bounded, $\vert\vert \varphi_n(x_i)\vert\vert /\lambda_n$ $\omega$-tends to $0$, contradicting that $x_i$ is hyperbolic. As a consequence, one can assume that $ p_n = p $ for all $ n $. Thus, the image $\varphi_n(\bar{u}^{-1}\bar{h}\bar{v}x_i^p)$ belongs to $E(G)$ $\omega$-almost-surely, i.e.\ $\bar{u}^{-1}\bar{h}\bar{v}x_i^p$ belongs to $H$ and $u^{-1}hvx_i^p$ belongs to $\varphi_{\infty}(H)=F$. Hence, there is an element $f\in F$ such that $u^{-1}hvx_i^p=f$, that is $h=ufx_i^{-p}v^{-1}$. This element belongs to $\Gamma$ since $u$, $v$, $f$ and $x_i$ belong to $\Gamma$.

\smallskip

Last, let $e$ be an edge of $T_{\Gamma}$. Let us prove that there are only finitely many branch points on $e$ in $T_L$. By Lemma 4.7 and Theorem 4.18 of \cite{GH19}, the hypotheses of Theorem \ref{guilev} are satisfied. It follows from this theorem that the number of orbits of directions at branch points in $T_L$ is finite. Now, assume towards a contradiction that there are infinitely many branch points on $e$. Then there exist necessarily two non-degenerate subsegments $I$ and $J$ in $e$, with $I\cap J=\varnothing$, and an element $g\in G$ such that $gI=J$. But we just proved that $T_{\Gamma}$ is transverse to its translates, so $g$ belongs to $\Gamma$. This is a contradiction since $T_{\Gamma}$ is a simplicial tree (by the previous lemma).\end{proof}

\section{Merzlyakov's theorem}\label{Merz_section}

The main theorem proved in this paper is the following generalisation of Merzlyakov's theorem. Note that Theorem \ref{th0bis} stated in the introduction corresponds to the case where $\ell=1$ in the result below.

\begin{te}\label{th0bis3}Let $G$ be an acylindrically hyperbolic group, and let $\bm{a}$ be a tuple of elements of $G$ (called constants). Fix a presentation $\langle \bm{a} \ \vert \ R(\bm{a})=1\rangle$ for the subgroup of $G$ generated by $\bm{a}$. Let \[\theta(\bm{x},\bm{y},\bm{a}):\bigvee_{k=1}^{\ell}(\Sigma_k(\bm{x},\bm{y},\bm{a})=1 \ \wedge \ \Psi_k(\bm{x},\bm{y},\bm{a})\neq 1)\] be a finite disjunction of finite systems of equations and inequations over $G$, where $\bm{x}$ and $\bm{y}$ are two tuples of variables. For every $1\leq k\leq \ell$, let $G_{\Sigma_k}$ denote the following group, finitely presented relative to $\langle \bm{a} \ \vert \ R(\bm{a})=1\rangle$: \[\langle \bm{x},\bm{y},\bm{a} \ \vert \ R(\bm{a})=1, \ \Sigma_k(\bm{x},\bm{y},\bm{a})=1\rangle.\]Let $p=\vert \bm{x}\vert$ be the arity of $\bm{x}$, and let $x_i$ denote the $i$th component of $\bm{x}$. Suppose that $G$ satisfies the following first-order sentence: \[\forall \bm{x} \ \exists \bm{y} \ \bigvee_{k=1}^{\ell}(\Sigma_k(\bm{x},\bm{y},\bm{a})=1 \ \wedge \ \Psi_k(\bm{x},\bm{y},\bm{a})\neq 1).\]
Then, for every $p$-tuple $\bm{\sigma}=(\sigma_1,\ldots,\sigma_p)\in \mathrm{Aut}_G(E(G))^p$, there exist an integer $1\leq k\leq \ell$ and a morphism \[\pi_{\bm{\sigma}} : G_{\Sigma_k}\rightarrow G_{\bm{\sigma}}=G\ast_{E(G)}\left\langle\bm{x},E(G) \ \vert \ \mathrm{ad}(x_i)_{\vert E(G)}={\sigma_i}, \ \forall i\in \llbracket 1,p\rrbracket\right\rangle\] such that the following hold:
\begin{enumerate}
    \item $\pi_{\bm{\sigma}}(\bm{x})=\bm{x}$,
    \item $\pi_{\bm{\sigma}}(\bm{a})=\bm{a}$,
    \item $\Psi(\bm{x},\pi_{\bm{\sigma}}(\bm{y}),\bm{a})\neq 1$.
\end{enumerate}
Moreover, the image of $\pi_{\bm{\sigma}}$ is a subgroup of $G_{\bm{\sigma}}$ of the form \[\left\langle \bm{g},\bm{a}\right\rangle\ast_{E(G)}\left\langle \bm{x}, E(G) \ \vert \ \mathrm{ad}(x_i)_{\vert E(G)}={\sigma_i}, \ \forall i\in \llbracket 1,p\rrbracket\right\rangle\] for some tuple $\bm{g}$ of elements of $G$.\end{te}

In fact, as shown by Lemma \ref{weaker_a_priori} below, it is enough to prove the following result, which is \emph{a priori} weaker than Theorem \ref{th0bis3} since the group $E(G)$ is replaced with a subgroup $E\subset E(G)$ that may be proper.

\begin{te}\label{th0bis2}Let $G$ be an acylindrically hyperbolic group, and let $\bm{a}$ be a tuple of elements of $G$ (called constants). Fix a presentation $\langle \bm{a} \ \vert \ R(\bm{a})=1\rangle$ for the subgroup of $G$ generated by $\bm{a}$. Let \[\theta(\bm{x},\bm{y},\bm{a}):\bigvee_{k=1}^{\ell}(\Sigma_k(\bm{x},\bm{y},\bm{a})=1 \ \wedge \ \Psi_k(\bm{x},\bm{y},\bm{a})\neq 1)\] be a finite disjunction of finite systems of equations and inequations over $G$, where $\bm{x}$ and $\bm{y}$ are two tuples of variables. For every $1\leq k\leq \ell$, let $G_{\Sigma_k}$ denote the following group, finitely presented relative to $\langle \bm{a} \ \vert \ R(\bm{a})=1\rangle$: \[\langle \bm{x},\bm{y},\bm{a} \ \vert \ R(\bm{a})=1, \ \Sigma_k(\bm{x},\bm{y},\bm{a})=1\rangle.\]Let $p=\vert \bm{x}\vert$ be the arity of $\bm{x}$, and let $x_i$ denote the $i$th component of $\bm{x}$. Suppose that $G$ satisfies the following first-order sentence: \[\forall \bm{x} \ \exists \bm{y} \ \bigvee_{k=1}^{\ell}(\Sigma_k(\bm{x},\bm{y},\bm{a})=1 \ \wedge \ \Psi_k(\bm{x},\bm{y},\bm{a})\neq 1).\]
Then, for every $p$-tuple $\bm{\sigma}=(\sigma_1,\ldots,\sigma_p)\in \mathrm{Aut}_G(E(G))^p$, there exist an integer $1\leq k\leq \ell$, a finite subgroup $E$ of $E(G)$, and a morphism \[\pi_{\bm{\sigma}} : G_{\Sigma_k}\rightarrow G_{\bm{\sigma}}=\left\langle G, \bm{x} \ \vert \ \mathrm{ad}(x_i)_{\vert E}={\sigma_i}, \ \forall i\in \llbracket 1,p\rrbracket\right\rangle\] such that the following hold:
   \begin{enumerate}
    \item $\pi_{\bm{\sigma}}(\bm{x})=\bm{x}$,
    \item $\pi_{\bm{\sigma}}(\bm{a})=\bm{a}$,
    \item $\Psi(\bm{x},\pi_{\bm{\sigma}}(\bm{y}),\bm{a})\neq 1$.
\end{enumerate}
Moreover, the image of $\pi_{\bm{\sigma}}$ is a subgroup of $G_{\bm{\sigma}}$ of the form \[\left\langle \bm{g},\bm{a}\right\rangle\ast_{E}\left\langle \bm{x}, E \ \vert \ \mathrm{ad}(x_i)_{\vert E}={\sigma_i}, \ \forall i\in \llbracket 1,p\rrbracket\right\rangle\] for some tuple $\bm{g}$ of elements of $G$.\end{te}

\begin{lemme}\label{weaker_a_priori}Theorems \ref{th0bis3} and \ref{th0bis2} are equivalent.\end{lemme}

\begin{proof}Theorem \ref{th0bis2} follows immediately from Theorem \ref{th0bis3}. Let us prove the converse. The proof consists in slightly modifying the first-order formula $\forall \bm{x} \ \exists \bm{y} \ \theta(\bm{x},\bm{y},\bm{a})$. Let $\bm{b}$ be the tuple of elements of $G$ composed of $\bm{a}$ and $E(G)$, and let $\bm{\sigma}_1,\ldots,\bm{\sigma}_N$ be an enumeration of the elements of $\mathrm{Aut}_G(E(G))^p$. For $1\leq i\leq N$, let $\mu_i(\bm{x},\bm{y},\bm{b})$ be the quantifier-free formula saying "$\theta(\bm{x},\bm{y},\bm{a})$ is true and $\bm{x}$ acts on $E(G)$ as $\bm{\sigma}_i$". Since $\forall \bm{x} \ \exists \bm{y} \ \theta(\bm{x},\bm{y},\bm{a})$ holds in $G$, the following first-order sentence holds in $G$ as well:\[\forall \bm{x} \ \exists \bm{y} \ \bigvee_{i=1}^N \mu_i(\bm{x},\bm{y},\bm{b}).\]Theorem \ref{th0bis3} follows from Theorem \ref{th0bis2} applied to this new first-order sentence.\end{proof}

\section{Proof of Merzlyakov's theorem \ref{th0bis2} in a particular case}

In this section, we deal with the case where $(\sigma_1,\ldots,\sigma_p)=(\mathrm{id}_{E(G)},\ldots,\mathrm{id}_{E(G)})$. More precisely, we prove the following result, which is a partial version of Theorem \ref{th0bis2}.

\begin{te}\label{th0bispartial}Let $G$ be an acylindrically hyperbolic group, and let $\bm{a}$ be a tuple of elements of $G$ (called constants). Fix a presentation $\langle \bm{a} \ \vert \ R(\bm{a})=1\rangle$ for the subgroup of $G$ generated by $\bm{a}$. Let \[\bigvee_{k=1}^{\ell}(\Sigma_k(\bm{x},\bm{y},\bm{a})=1 \ \wedge \ \Psi_k(\bm{x},\bm{y},\bm{a})\neq 1)\] be a finite disjunction of finite system of equations and inequations over $G$, where $\bm{x}$ and $\bm{y}$ are two tuples of variables. For every $1\leq k\leq \ell$, let $G_{\Sigma_k}$ denote the following group, finitely presented relative to $\langle \bm{a} \ \vert \ R(\bm{a})=1\rangle$: \[\langle \bm{x},\bm{y},\bm{a} \ \vert \ R(\bm{a})=1, \ \Sigma_k(\bm{x},\bm{y},\bm{a})=1\rangle.\]Let $p=\vert \bm{x}\vert$ be the arity of $\bm{x}$, and let $x_i$ denote the $i$th component of $\bm{x}$. Suppose that $G$ satisfies the following first-order sentence: \[\forall \bm{x} \ \exists \bm{y} \ \bigvee_{k=1}^{\ell}(\Sigma_k(\bm{x},\bm{y},\bm{a})=1 \ \wedge \ \Psi_k(\bm{x},\bm{y},\bm{a})\neq 1).\]
Then there exist an integer $1\leq k\leq \ell$, a finite subgroup $E$ of $E(G)$, and a morphism \[\pi_{\bm{\sigma}} : G_{\Sigma_k}\rightarrow G_{\bm{\sigma}}=\left\langle G, \bm{x} \ \vert \ \mathrm{ad}(x_i)_{\vert E}=\mathrm{id}_{E}, \ \forall i\in \llbracket 1,p\rrbracket\right\rangle\] such that the following hold:
\begin{itemize}
    \item[$\bullet$] $\pi_{\bm{\sigma}}(\bm{x})=\bm{x}$,
    \item[$\bullet$] $\pi_{\bm{\sigma}}(\bm{a})=\bm{a}$,
    \item[$\bullet$] $\Psi(\bm{x},\pi_{\bm{\sigma}}(\bm{y}),\bm{a})\neq 1$.
\end{itemize}
Moreover, the image of $\pi_{\bm{\sigma}}$ is a subgroup of $G_{\bm{\sigma}}$ of the form \[\left\langle \bm{g},\bm{a}\right\rangle\ast_{E}\left\langle \bm{x}, E \ \vert \ \mathrm{ad}(x_i)_{\vert E}=\mathrm{id}_{E}, \ \forall i\in \llbracket 1,p\rrbracket\right\rangle\] for some tuple $\bm{g}$ of elements of $G$.\end{te}

Recall that a $(\mathrm{id}_{E(G)},\ldots,\mathrm{id}_{E(G)})$-test sequence is simply called a \emph{test sequence}. First, we build a test sequence enjoying two special properties.

\subsection{Construction of a test sequence}

The construction relies crucially on the existence of a quasi-isometrically embedded subgroup of $G$ of the form $F(a,b)\times E(G)$ (provided by \cite[Theorem 6.14]{DGO17} together with \cite[Lemma 3.1]{AMS13}), which will enable us to use small cancellation within $F(a,b)$.

\begin{prop}\label{jesaispas}Let $G$ be an acylindrically hyperbolic group, and let $\bm{a}$ be a tuple of elements of $G$. Fix a presentation $\langle \bm{a} \ \vert \ R(\bm{a})=1\rangle$ for the subgroup of $G$ generated by $\bm{a}$. Let \[\bigvee_{k=1}^{\ell}(\Sigma_k(\bm{x},\bm{y},\bm{a})=1 \ \wedge \ \Psi_k(\bm{x},\bm{y},\bm{a})\neq 1)\] be a finite disjunction of finite system of equations and inequations over $G$, where $\bm{x}$ and $\bm{y}$ are two tuples of variables. For every $1\leq k\leq \ell$, denote \[G_{\Sigma_k}=\langle \bm{x},\bm{y},\bm{a} \ \vert \ R(\bm{a})=1, \ \Sigma_k(\bm{x},\bm{y},\bm{a})=1\rangle.\]Suppose that $G$ satisfies the following first-order sentence:\[\forall \bm{x} \ \exists \bm{y} \ \bigvee_{k=1}^{\ell}(\Sigma_k(\bm{x},\bm{y},\bm{a})=1 \ \wedge \ \Psi_k(\bm{x},\bm{y},\bm{a})\neq 1).\]Then, there exists an integer $1\leq k\leq \ell$ and a test sequence $(\varphi_n : G_{\Sigma_k} \rightarrow G)_{n\in\mathbb{N}}$ satisfying the following two conditions $\omega$-almost-surely:
\begin{enumerate}
\item no component of the system of inequations $\Psi(\bm{x},\bm{y},\bm{a})$ is killed by $\varphi_n$, 
\item and the morphism $\varphi_n$ maps $\bm{a}$ to $\bm{a}$ (not only to a conjugate).
\end{enumerate}
\end{prop}

\begin{proof}By Theorem 6.14 in \cite{DGO17}, there exists a hyperbolically embedded subgroup $H\hookrightarrow_h G$ such that $H=F(a,b)\times E(G)$, where $F(a,b)$ denotes the free group on two generators $a$ and $b$, and the elements $a$ and $b$ are hyperbolic in $G$.

\smallskip

Up to replacing the generating set $S$ of $G$ with $S\cup\lbrace a,b\rbrace$, one can assume without loss of generality that $a$ and $b$ belong to $S$. Let $(X,d)$ be the Cayley graph of $G$ with respect to this enlarged set $S$.

\smallskip

Let $d'$ denote the metric in the free group $ \langle a, b \rangle $ for the generating set $ \lbrace a, b \rbrace $. By \cite[Lemma 3.1]{AMS13}, there exist two constants $q$ and $r$ such that \begin{equation}\label{ineg}d'(1,h)\leq q d(1,h)+r
\end{equation} for all $h\in \langle a,b\rangle$. 

\smallskip

Let $p$ denote the arity of $\bm{x}$. For any integers $1\leq i\leq p$ and $n\geq 0 $, we define $g_{i,n}=a^{(i-1)n+1} b a^{(i-1)n+2} b\cdots a^{in} b$. Let $\bm{g}_n$ be the $p$-tuple $(g_{1,n},\ldots,g_{p,n})$.

\smallskip

There exists an integer $1\leq k\leq \ell$ such that, for infinitely many integers $n$, there exists a tuple $\bm{h}_n$ of elements of $G$ such that $\Sigma_k(\bm{g}_n,\bm{h}_n,\bm{a})=1\wedge \Psi_k(\bm{g}_n,\bm{h}_n,\bm{a})\neq 1$. By passing to a subsequence and relabelling, one can assume without loss of generality that this system of equalities and inequalities holds for all integers $n$.

\smallskip

Let $\varphi_n: \Gamma\twoheadrightarrow G$ be the morphism defined by $\varphi_n (\bm{x}) = \bm{g}_n $, $\varphi_n(\bm{y})=\bm{h}_n$ and $\varphi_n (\bm{a}) = \bm{a} $. We will prove that $(\varphi_n)_{n\in\mathbb{N}}$ is a test sequence.

\smallskip

For $1\leq i\leq p$ and $n\geq 0$, let $ \tau_{i,n} $ be the path of $ X $ that links $ 1 $ to $ g_{i,n} $ and is labeled with the word $ g_{i,n} $ in $ a $ and $ b $, and consider the bi-infinite path $\overline{\tau}_{i,n}=\cup_{k\in\mathbb{Z}} g_{i,n}^k\tau_{i,n}$. This path $\overline{\tau}_{i,n}$ is a geodesic in the Cayley graph of the free group $\langle a,b\rangle$ equipped with the metric $d'$, and by the inequality \ref{ineg} this graph is quasi-isometrically embedded into $(X,d)$, thus $\overline{\tau}_{i,n}$ is a quasi-geodesic in $(X,d)$, for some constants that do not depend on $n$. Consequently, $\overline{\tau}_{i,n}$ lies in the $\lambda$-neighborhood of the quasi-axis $A(g_{i,n})$ of $g_{i,n}$ for some constant $\lambda\geq 0$ independent from $n$. Similarly, let $\alpha$ be the edge of $X$ linking $1$ to $a$, let $\overline{\alpha}$ denote the quasi-geodesic $\overline{\alpha}=\cup_{k\in\mathbb{Z}} a^k\alpha$ and let $\mu$ be a constant such that $\overline{\alpha}$ lies in the $\mu$-neighborhood of $A(a)$.

\smallskip

Since $g_{i,n}$ is cyclically reduced in $\langle a,b\rangle$, an easy calculation shows that $d'(1, g_{i,n}) \sim (i-1/2) n^2 $. Thus, the second and third conditions of Definition \ref{suitetest} hold. In addition, note that it follows from the inequality (\ref{ineg}) that there exists a constant $ R> 0 $ such that $ \vert\vert g_{i,n}\vert\vert \geq Rn^ 2 $ for all $ n $ large enough, and all $i\in\llbracket 1,p\rrbracket$.

\smallskip

It remains to prove the fourth condition of Definition \ref{suitetest}. Since $ a $ is hyperbolic, there exists a constant $ N \geq 0 $ such that, for every element $ g \in G $, if $ \Delta (a, gag^{-1}) \geq N $, then $ g $ belongs to $\Lambda (a)=\langle a\rangle\times E(G)$ (see paragraph \ref{Coulon}). Let $n_0$ be an integer such that $ 16qRn_0 $ is large compared to $N'=N+204\delta+2\lambda+2\mu$, where $q$ is the constant involved in the inequality (\ref{ineg}). We will show that for every $ n \geq n_0 $, the tuple $\varphi_n(\bm{x})=\bm{g}_n$ satisfies the $(16q/n)$-small cancellation condition \ref{SCC}. Let $ n $ be an integer greater than $ n_0 $. Consider an element $ g \in G $ such that \begin{equation}\label{ineg2}\Delta(g_{i,n},gg_{j,n}g^{-1})\geq 16q\min(\vert\vert g_{i,n}\vert\vert,\vert\vert g_{j,n}\vert\vert)/n
\end{equation}
for some $(i,j)\in\llbracket 1,p\rrbracket^2$. We will show that $i=j$ and that $ g $ belongs to the subgroup $\langle g_{i,n}\rangle\times E(G)$. One can suppose without loss of generality that $j$ is larger than $i$. Thus, $\omega$-almost-surely, one has $\min(\vert\vert g_{i,n}\vert\vert,\vert\vert g_{j,n}\vert\vert)=\vert\vert g_{i,n}\vert\vert$. 

\smallskip

We first show that $ g $ belongs to the subgroup $\langle a,b\rangle\times E(G)$. Since \[\Delta(g_{i,n},gg_{j,n}g^{-1})\geq 16q\vert\vert g_{i,n}\vert\vert/n\geq 16qRn\geq 16qRn_0 \gg N',\] we can choose two subpaths $\mu_{i,n}$ and $\mu_{j,n}$ of $\overline{\tau}_{i,n}$ and $g\overline{\tau}_{j,n}$ respectively, of length $ N' $ and labeled by $ a^{N'} $, such that $\mathrm{diam}((\mu_{i,n})^{+(100\delta+\lambda)}\cap(\mu_{j,n})^{+(100\delta+\lambda)})\geq N'$. Denoting by $ o_{i,n} $ and $ o_{j,n} $ the initial points of $ \mu_{i,n} $ and $ \mu_{j,n} $ respectively, we have \[\mathrm{diam}(o_{i,n}\overline{\alpha}^{+(100\delta+\lambda)}\cap o_{j,n}\overline{\alpha}^{+(100\delta+\lambda)})\geq N'.\]It follows that 
\[\mathrm{diam}(A(a)^{+(100\delta+\lambda+\mu)}\cap o_{i,n}^{-1}o_{j,n}A(a)^{+(100\delta+\lambda+\mu)})\geq N'.\] By Lemma 2.13 in \cite{Cou13}, we have:
{\small
\begin{align*}
\Delta(a,{(o_{i,n}^{-1}o_{j,n})}a{(o_{i,n}^{-1}o_{j,n})}^{-1})& \geq \mathrm{diam}(A(a)^{+(100\delta+\lambda+\mu)}\cap o_{i,n}^{-1}o_{j,n}A(a)^{+(100\delta+\lambda+\mu)}) - (204\delta+2\lambda+2\mu) \\
 & \geq N' - (204\delta+2\lambda+2\mu) =N.
\end{align*}}It follows from this inequality that the element $ o_{i,n}^{-1}o_{j,n} $ belongs to $\Lambda (a)=\langle a\rangle\times E(G)$. Now, observe that as $o_{i,n}$ lies in $\langle a,b\rangle$, it is on the quasi-geodesic $ \overline {\tau}_{i,n} $. Similarly, $ o_{j,n} $ can be written as $ o_{j,n} = gw_{j,n} $ with $ w_{j,n} $ a word in $ a $ and $ b $. It follows that $ g $ belongs to the subgroup $ \langle a, b \rangle \times E(G) $.

\smallskip

Up to replacing $g$ with $gc$ for some $c \in E(G)$, we can now assume that $g$ belongs to the free group $\langle a,b\rangle$. This does not affect the inequality $\Delta(g_{i,n},gg_{j,n}g^{-1})\geq 16q \vert\vert g_{i,n}\vert\vert/n$; indeed, $gcg_{j,n}{(gc)}^{-1}$ is equal to $gg_{j,n}g^{-1}$ since $ g_{j,n} $ centralizes $ E(G) $, as an element of $\langle a,b\rangle$. 

\smallskip

Let $ Y $ be the Cayley graph of the free group $ \langle a, b \rangle $ equipped with the distance $ d'$. The following inequality can be easily deduced from the inequalities (\ref{ineg}) and (\ref{ineg2}):\[\mathrm{diam}\left((\overline{\tau}_{i,n})^{+(q(100\delta+r)+1)}\cap (g\overline{\tau}_{j,n})^{+(q(100\delta+r)+1)}\right)\geq 16q d'(1,g_{i,n})/(2qn)=8d'(1,g_{i,n})/n.\] Since $Y$ is a tree, this inequality tells us that the axes of $ g_{i,n} $ and $ gg_{j,n}g^{-1} $ have an overlap of length larger than $8d'(1,g_{i,n})/n$ in this tree. Then, recall that $d'(1,g_{i,n})$ is asymptotically equivalent to $(i-1/2) n^2$. Thus, $8d'(1,g_{i,n})/n$ is asymptotically equivalent to $8(i-1/2) n$. Therefore, $\omega$-almost-surely, the axes of $ g_{i,n} $ and $ gg_{j,n}g^{-1} $ have an overlap of length larger than $4(i-1/2) n$ in the tree $Y$.

\smallskip

To conclude, let us observe that $4(i-1/2)n>2in-2$, and that two distinct cyclic conjugates of $g_{i,n}$ and $g_{j,n}$ have at most their first $ 2in-2 $ letters in common (recall that $j$ is larger than $i$ by assumption). Thus, if the axes of $ g_{i,n} $ and $ gg_{j,n}g^{-1} $ have a common subsegment in $ Y $ of length strictly larger than $2in-2 $, then $ g_{i,n} $ and $ gg_{j,n}g^{-1} $ have the same axis. It follows that $i=j$ and that $ g_{i,n} $ and $ g $ have a common root. Last, note that $ g_{i,n} $ has no root. It follows that $ g $ is a power of $ g_{i,n} $, which concludes the proof.\end{proof}

\subsection{Proof of Theorem \ref{th0bispartial}}

Theorem \ref{th0bispartial} is an immediate consequence of Proposition \ref{jesaispas} and Proposition \ref{jesaispas2} below (applied with $(\sigma_1,\ldots,\sigma_p)=(\mathrm{id}_{E(G)},\ldots,\mathrm{id}_{E(G)})$).

\begin{prop}\label{jesaispas2}Let $G$ be an acylindrically hyperbolic group, and let $\bm{a}$ be a tuple of elements of $G$. Fix a presentation $\langle \bm{a} \ \vert \ R(\bm{a})=1\rangle$ for the subgroup of $G$ generated by $\bm{a}$. Let \[\Sigma(\bm{x},\bm{y},\bm{a})=1 \ \wedge \ \Psi(\bm{x},\bm{y},\bm{a})\neq 1\] be a finite system of equations and inequations over $G$, where $\bm{x}$ and $\bm{y}$ are tuples of variables. Suppose that there exists a $(\sigma_1,\ldots,\sigma_p)$-test sequence $(\varphi_n : G_{\Sigma} \rightarrow G)_{n\in\mathbb{N}}$ satisfying the following two conditions $\omega$-almost-surely:
\begin{enumerate}
\item no component of the system of inequations $\Psi(\bm{x},\bm{y},\bm{a})$ is killed by $\varphi_n$, 
\item and the morphism $\varphi_n$ maps $\bm{a}$ to $\bm{a}$ (not only to a conjugate).
\end{enumerate}
Then there exist a finite subgroup $E$ of $E(G)$ and a morphism \[\pi_{\bm{\sigma}} : G_{\Sigma_k}\rightarrow G_{\bm{\sigma}}=\left\langle G, \bm{x} \ \vert \ \mathrm{ad}(x_i)_{\vert E}=\sigma_i, \ \forall i\in \llbracket 1,p\rrbracket\right\rangle\] such that the following hold:
\begin{itemize}
    \item[$\bullet$] $\pi_{\bm{\sigma}}(\bm{x})=\bm{x}$,
    \item[$\bullet$] $\pi_{\bm{\sigma}}(\bm{a})=\bm{a}$,
    \item[$\bullet$] no component of the tuple $\Psi(\bm{x},\bm{y},\bm{a})$ is killed by $\pi_{\sigma}$.
\end{itemize}
Moreover, the image of $\pi_{\bm{\sigma}}$ is a subgroup of $G_{\bm{\sigma}}$ of the form \[\left\langle \bm{g},\bm{a}\right\rangle\ast_{E}\left\langle \bm{x}, E \ \vert \ \mathrm{ad}(x_i)_{\vert E}=\sigma_i, \ \forall i\in \llbracket 1,p\rrbracket\right\rangle\] for some tuple $\bm{g}$ of elements of $G$.\end{prop}

\begin{proof}By assumption, there exists a $(\sigma_1,\ldots,\sigma_p)$-test sequence $(\varphi_n : G_{\Sigma} \rightarrow G)_{n\in\mathbb{N}}$ that satisfies the following two conditions $\omega$-almost-surely:
\begin{enumerate}
\item each component of $\varphi_n(\Psi(\bm{x},\bm{y},\bm{a}))$ is non-trivial, 
\item and the morphism $\varphi_n$ maps $\bm{a}$ to $\bm{a}$ (not only to a conjugate).
\end{enumerate}

\smallskip

Let $U$ be the subgroup of $G$ generated by $\bm{x}$ and $\bm{a}$. Since the Cayley graph of $G$ with respect to $S$ (on which $G$ acts acylindrically and non-elementarily) is discrete, the length of any morphism $G_{\Sigma}\rightarrow G$ belongs to $\mathbb{N}$. As a consequence, there exists a sequence of morphisms $(\theta_n : G_{\Sigma} \rightarrow G)_{n\in\mathbb{N}}$ that satisfies simultaneously the following three conditions $\omega$-almost-surely:
\begin{enumerate}
    \item $\theta_n$ coincides with $\varphi_n$ on $U$ up to conjugation,
    \item each component of $\theta_n(\Psi(\bm{x},\bm{y},\bm{a}))$ is non-trivial,
    \item and there is no morphism that satisfies simultaneously the conditions \emph{(1)} and \emph{(2)} above and that is stricly shorter than $\theta_n$.
\end{enumerate}
Note that, by Remark \ref{impor}, the sequence $(\theta_n)_{n\in\mathbb{N}}$ is a $(\sigma'_1,\ldots,\sigma'_p)$-test sequence for some $(\sigma'_1,\ldots,\sigma'_p)\in\mathrm{Aut}_G(E(G))^p$. However, one cannot guarantee that $\sigma'_i$ coincides with $\sigma_i$. Moreover, $\theta_n$ maps $\bm{a}$ to a conjugate of $\bm{a}$, not necessarily to $\bm{a}$ itself. 


\smallskip

Let $L=G_{\Sigma}/\underleftarrow{\ker}_\omega((\theta_n)_{n\in\mathbb{N}})$, and let $\theta_{\infty} : G_{\Sigma}\twoheadrightarrow L$ be the corresponding epimorphism. Note that $\theta_{\infty}$ is injective on $U$. In the proof below, we abuse notation and denote by $U$ the isomorphic image of $U$ in the successive quotients of $G_{\Sigma}$ involved in the construction of the formal solution $\pi_{\bm{\sigma}}$.

\smallskip

In the rest of the proof, $C$ denotes the constant defined in the Stability Lemma \ref{stabilitylemma}.


\medskip

\textbf{A particular case.} For presentation purposes, we first present a proof of Proposition \ref{jesaispas2} in the particular case where $L$ does not split non-trivially over a finite group of order less than $C$. Under this assumption, if one assumes (towards a contradiction) that the group $U$ is elliptic in the limiting tree of the test sequence $({\theta_n} : G_{\Sigma} \rightarrow G)_{n\in\mathbb{N}}$, then by Theorem \ref{shortening_arg} there exists a sequence of homomorphisms $(\rho_n: G_{\Sigma} \rightarrow G)_{n\in\mathbb{N}}$ satisfying the following three conditions $\omega$-almost-surely: 
\begin{enumerate}
    \item $\rho_n$ coincides with $\theta_n$ (and therefore with $\varphi_n$) on $U$ up to conjugation,
    \item $\rho_n$ kills no component of the tuple $\Psi(\bm{x},\bm{y},\bm{a})$,
    \item and $\rho_n$ is stricly shorter than $\theta_n$ relative to $H$.
\end{enumerate}
This contradicts the definition of $\theta_n$ as the shortest morphism satisfying both conditions (1) and (2) $\omega$-almost-surely. Hence, $U$ is not elliptic in the limiting tree of the test sequence $({\theta_n} : G_{\Sigma} \rightarrow G)_{n\in\mathbb{N}}$. The conclusion now follows from the following technical lemma, whose proof is postponed (see Lemma \ref{abovelemma} for a more general version).

\begin{lemme}Let $F$ be the finite subgroup of $L$ defined in Lemma \ref{fini}. If $U$ is not elliptic in the limiting tree, then the group $L$ admits a splitting $\mathbb{S}_L$ with exactly two vertex groups $\langle \bm{\ell}, \bm{a}\rangle$ (for some tuple $\bm{\ell}$ of elements of $L$) and $\langle \bm{x},F\rangle$, and one edge group $F$. Let $A$ be an $\mathbb{S}_L$-approximation of $L$ as in Proposition \ref{approx}. There exist a finite subgroup $E$ of $E(G)$ and an epimorphism $r$ from $A$ onto a group of the form \[\langle \bm{g},\bm{a}\rangle\ast_E\langle \bm{x},E \ \vert \ \mathrm{ad}(x_i)_{\vert E}={\sigma_i}_{\vert E}, \ \forall i\in \llbracket 1,p\rrbracket\rangle,\] where $\bm{g}$ denotes a tuple of elements of $G$, such that $r(\bm{x})=\bm{x}$, $r(\bm{a})=\bm{a}$ and $r$ kills no component of the image in $A$ of the tuple $\Psi(\bm{x},\bm{y},\bm{a})$.\end{lemme}

Last, one defines the formal solution $\pi_{\bm{\sigma}} : G_{\Sigma}\rightarrow G$ by $\pi_{\bm{\sigma}}=r\circ q$ where $q$ denotes the natural epimorphism from $G_{\Sigma}$ onto $A$. This concludes the proof of Proposition \ref{jesaispas2} in the particular case where $L$ does not split non-trivially over a finite group of order less than $C$. In general, however, this hypothesis is not satisfied and one has to deal with complications arising from splittings over finite subgroups. In particular, one needs a strengthened version of the relative shortening argument Theorem \ref{shortening_arg}, namely Theorem \ref{shortening_arg2}.

\medskip

\textbf{General case.} Since we are going to describe an iterative process, let us rename $\theta_n$ to $\theta_n^0$, and $L$ to $L_0$. For any $G$-limit group $L_i$ that appears in the proof, we denote by $L_i^{U}$ the vertex group containing $U$ in a reduced JSJ decomposition $\mathbb{J}_i$ of $L_i$, relative to $U$, over finite groups of order less than $C$. 

\smallskip

The proof of Proposition \ref{jesaispas2} consists in constructing the following commuting diagrams $\omega$-almost-surely (the objects appearing in this diagram are defined below):

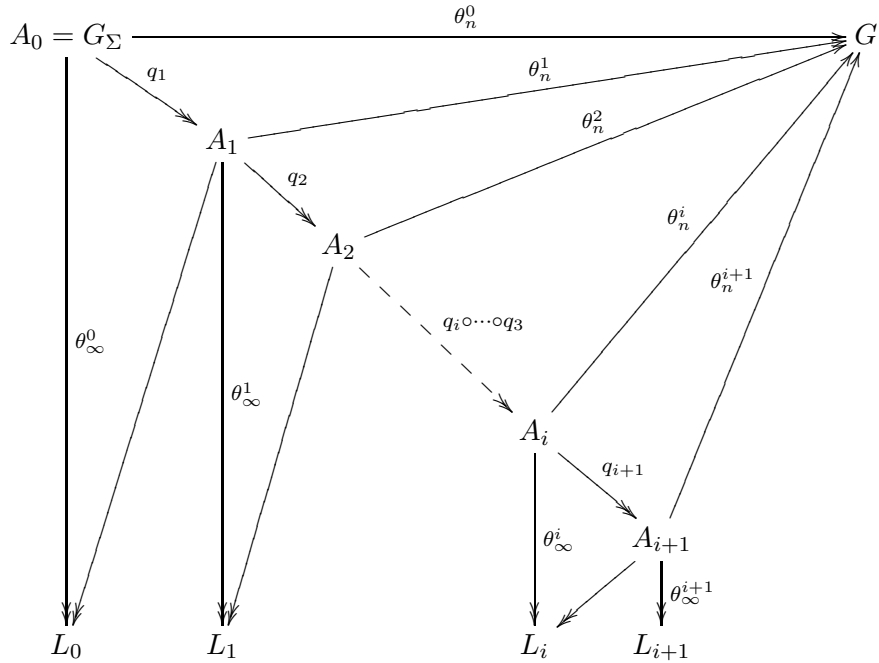
\begin{align*}
    \xymatrix{
    A_0=G_\Sigma \ar@{->>}[dr]^{q_1} \ar@{->>}[dddddd]^{\theta^0_\infty} \ar[rrrrrrr]^{\theta^0_n} && && &&& G \\
    & A_1 \ar@{->>}[dddddl] \ar@{->>}[dr]^{q_2} \ar@{->>}[ddddd]^{\theta^1_{\infty}} \ar[rrrrrru]^{\theta_n^1} \\ 
    & & A_2 \ar@{->>}[ddddl] \ar@{-->>}[ddrr]^{q_{i} \circ \cdots \circ q_3} \ar[rrrrruu]^{\theta^2_n}\\
    \\
    & & & & A_{i} \ar@{->>}[dd]^{\theta^{i}_\infty} \ar@{->>}[dr]^{q_{i+1}} \ar[uuuurrr]^{\theta^{i}_n} \\
    & & & & & A_{i+1} \ar@{->>}[d]^{\theta^{i+1}_\infty} \ar@{->>}[dl] \ar[uuuuurr]^{\theta^{i+1}_n}\\
    L_0 & L_1 & & & L_{i} & L_{i+1}\\
    }
\end{align*}

This diagram is built iteratively, as follows: given the sequence $(\theta_n^i : A_i \rightarrow G)_{n\in\mathbb{N}}$, one defines $L_i$ by $L_i=A_i/\underleftarrow{\ker}_\omega((\theta_n^i)_{n\in\mathbb{N}})$. Let $\mathbb{J}_i$ be a reduced JSJ spltting of $L_i$ over finite groups of order less than $C$, let $\mathbb{R}_i$ be the splitting of the vertex group $L_i^U$ as a graph of actions outputted by the Rips machine and let $\mathbb{RJ}_i$ be the splitting of $L$ obtained from $\mathbb{J}_i$ by replacing the vertex fixed by $L_i^U$ with the graph of groups $\mathbb{R}_i$. Let $A_{i+1}$ be an $\mathbb{RJ}_i$-approximation of $L_i$ given by Proposition \ref{approx} and Corollary \ref{comb_app}, and let $\rho_n^{i+1}: A_{i+1}\rightarrow G$ be the factorization of $\theta_n^i : A_i\rightarrow G$ through the natural epimorphism $q_{i+1} : A_i \twoheadrightarrow A_{i+1}$. 

\smallskip

 Since $A_{i+1}$ is an $\mathbb{RJ}_i$-approximation of $L_i$, it is also a $\mathbb{J}_i$-approximation of $L_i$ (indeed, one can collapse to a point the subgraph corresponding to $\mathbb{R}_i$). We denote by $A_{i+1}^U$ the vertex group of the splitting of $A_{i+1}$ corresponding to $L_i^U$. Note that $A_{i+1}^U$, unlike $L_i^U$, may split non-trivially relative to $U$ over finite subgroups of order less than $C$. It remains to define the sequence $(\theta_n^{i+1} : A_{i+1} \rightarrow G)_{n\in\mathbb{N}}$.
 
\smallskip


If $U$ is elliptic in the limiting tree of the sequence $({\rho_n^{i+1}})_{n\in\mathbb{N}}$, then by Theorem \ref{shortening_arg2} there exists a sequence of homomorphisms $(\theta_n^{i+1}: A_{i+1} \rightarrow G)_{n\in\mathbb{N}}$ satisfying the following three conditions $\omega$-almost-surely: 
\begin{enumerate}
    \item $\theta_n^{i+1}$ coincides with $\rho_n^{i+1}$ (and therefore with $\varphi_n$) on $U$ up to conjugation,
    \item $\theta_n^{i+1}$ kills no component of the image of the tuple $\Psi(\bm{x},\bm{y},\bm{a})$ in $A_{i+1}$,
    \item and the restriction ${\theta_n^{i+1}}_{\vert A_{i+1}^U}$ is stricly shorter than the restriction ${\rho_n^{i+1}}_{\vert A_{i+1}^U}$, relative to $U$.
\end{enumerate}
In addition, since the length of $\theta_n^{i+1}$ belongs to $\mathbb{N}$, one can assume without loss of generality that, $\omega$-almost-surely, $\theta_n^{i+1}$ is the shortest morphism from $A_{i+1}$ to $G$ that satisfies the first two conditions above.

\smallskip

Why does the iteration eventually terminate? We have to prove that there exists an integer $i$ such that $U$ is not elliptic in the limiting tree of the sequence $({\rho_n^{i+1}})_{n\in\mathbb{N}}$.


\smallskip

\textbf{Claim.} \emph{There exists an integer $i$ such that $q_{i+1}(A_i^{U})=A_{i+1}^{U}$. }

\smallskip

Before proving this claim, let us explain how to complete the proof of Proposition \ref{jesaispas2}. First, note that if $q_{i+1}(A_i^{U})$ is equal to $A_{i+1}^{U}$, then if one shortens the restriction of ${\theta_n^{i+1}}$ to $ A_{i+1}^U$, one automatically shortens the restriction of $\rho_n^{i}$ to $ A_i^U$, which is not possible by definition of ${\rho_n^{i}}$. As a consequence, if $q_{i+1}(A_i^{U})=A_{i+1}^{U}$, then $U$ cannot be elliptic in the limiting tree of the sequence $({\theta_n^{i+1}})_{n\in\mathbb{N}}$, otherwise one could get a contradiction by means of Theorem \ref{shortening_arg2}. In order to construct the formal solution, we will use the following lemma, whose proof is postponed.


\begin{lemme}\label{abovelemma}Let $F$ be the finite subgroup of $L_{i+1}$ defined in Lemma \ref{fini}. If $U$ is non-elliptic in the limiting tree of the sequence $({\theta_n^{i+1}})_{n\in\mathbb{N}}$, then the group $L_{i+1}$ admits a splitting $\mathbb{S}$ with exactly two vertex groups $\langle \bm{\ell}, \bm{a}\rangle$ (for some tuple $\bm{\ell}$ of elements of $L_{i+1}$) and $\langle \bm{x},F\rangle$, and one edge group $F$. Let $A$ be an $\mathbb{S}$-approximation of $L_{i+1}$ given by Proposition \ref{approx}. There exists a subgroup $E$ of $E(G)$ and an epimorphism $r$ from $A$ onto a group of the form \[\langle \bm{g},\bm{a}\rangle\ast_E \langle \bm{x},E \ \vert \ \mathrm{ad}(x_i)_{\vert E}={\sigma_i}_{\vert E}, \ \forall i\in \llbracket 1,p\rrbracket\rangle,\] where $\bm{g}$ denotes a tuple of elements of $G$, such that $r(\bm{x})=\bm{x}$, $r(\bm{a})=\bm{a}$ and $r$ kills no component of the image in $A$ of the tuple $\Psi(\bm{x},\bm{y},\bm{a})$.\end{lemme}

As a consequence of this lemma, if $U$ is not elliptic in the limiting tree of the sequence $({\rho_n^{i+1}})_{n\in\mathbb{N}}$, one can define the formal solution $\pi_{\bm{\sigma}} : G_{\Sigma}\rightarrow G$ by $\pi_{\bm{\sigma}}=r\circ q_{i+1}\circ q_i\circ \cdots\circ q_0$.

\smallskip

Therefore, in order to conclude the proof of Proposition \ref{jesaispas2}, we just have to prove the claim according to which there exists an integer $i$ such that $q_{i+1}(A_i^{U})=A_{i+1}^{U}$. Let us denote by $\eta_i$ the number of edges in a reduced JSJ splitting $\mathbb{J}_i$ of $L_i$ over finite groups of order less than $C$, relative to $U$. Let $E(\mathbb{J}_i)$ be the set of edges of $\mathbb{J}_i$. We make the following two observations.

\smallskip

\textbf{First observation:} using a folding sequence argument, Dunwoody proved in \cite{Dun98} that the sum $\sum_{e\in E(\mathbb{J}_i)}1/\vert {L_i}_{e}\vert$, where ${L_i}_{e}$ denotes the edge group of $e$, is smaller than the rank $\mathrm{rank}_U(L_i)$ of $L_i$ relative to $U$ (that is the minimal number of generators of $L_i$ relative to $U$). Therefore, for every integer $i$, one has $\eta_i\leq C\mathrm{rank}_U(L_i)$. In addition, one has $\mathrm{rank}_U(L_i)\leq \mathrm{rank}_U(G_{\Sigma})$ since $L_i$ is a quotient of $G_{\Sigma}$ relative to $U$. Thus, $\eta_i$ is bounded from above by $C\mathrm{rank}_U(G_{\Sigma})$. 

\smallskip

\textbf{Second observation:} we claim that $\eta_{i+1}$ is greater than $\eta_i$, with equality if and only if $q_{i+1}(A_i^{U})=A_{i+1}^{U}$. 

\smallskip

Let us prove this claim. Since $A_{i+1}$ is a $\mathbb{J}_i$-approximation of $L_i$, there exists by definition a splitting $\mathbb{J}'_{i}$ of $A_{i+1}$ with the same underlying graph as $\mathbb{J}_i$, and whose edge groups have the same order as the corresponding edge groups in $\mathbb{J}_i$. In particular, $\mathbb{J}'_i$ is a splitting over finite groups of order less than $C$, with $\eta_i$ edges. Moreover, by Proposition \ref{approx} and Lemma \ref{reduced_as_well}, the splitting $\mathbb{J}'_i$ is reduced since $\mathbb{J}_i$ is reduced.

\smallskip

In order to establish the inequality $\eta_{i+1}\geq \eta_i$, let us have a closer look at the defining sequence $(\theta_n^{i+1} : A_{i+1} \rightarrow G)_{n\in\mathbb{N}}$ of $L_{i+1}=A_{i+1}/\underleftarrow{\ker}_\omega((\rho_n^{i+1})_{n\in\mathbb{N}}$. In the proof of Theorem \ref{shortening_arg2}, each morphism $\theta_n^{i+1}$ is obtained by precomposing $\rho_n^{i+1}$ by an automorphism $\alpha$ of $A_{i+1}$ (independent from $n$) whose restriction to $A_{i+1}^U$ is a modular automorphism (or, to be more precise, a lift of a modular automorphism of $L_i^U$), and whose restriction to any other vertex group of $\mathbb{J}'_i$ is a conjugation. This automorphism $\alpha$ is obtained by means of Lemma \ref{extension_lemma}, using the fact that modular automorphisms coincide with the identity up to conjugation on finite subgroups of $A_{i+1}^U$. As a consequence, $L_{i+1}$ admits a splitting $\mathbb{J}''_i$ with $\eta_i$ edge groups, over finite groups of order less than $C$, obtained from the splitting $\mathbb{J}'_i$ of $A_{i+1}$ by replacing each vertex group by its image by the quotient map $\theta_{\infty}^{i+1}$. This shows that a reduced JSJ splitting of $L_{i+1}$ has at least $\eta_{i}$ edges; in other words, one has $\eta_{i+1}\geq \eta_i$. 

\smallskip

Now, suppose that $\eta_i=\eta_{i+1}$, and let us prove that $\mathbb{J}''_i$ is a reduced JSJ splitting of $L_{i+1}$ over finite groups of order less than $C$. Since we already know that $\mathbb{J}''_i$ is a splitting of $L_{i+1}$ over finite groups of order less than $C$ with $\eta_{i+1}$ edges, we just have to prove that $\mathbb{J}''_i$ is reduced. To this end, let us verify that the conditions of Lemma \ref{reduced_as_well} are satisfied. By definition of $\mathbb{J}''_i$, the natural epimorphism $\theta_{\infty}^{i+1}$ from $A_{i+1}$ onto $L_{i+1}$ maps each vertex group of $\mathbb{J}'_i$ onto the corresponding vertex group of $\mathbb{J}''_i$. We will prove the following two facts:
\begin{enumerate}
    \item $\theta_{\infty}^{i+1}$ is injective on finite vertex groups,
    \item and $\theta_{\infty}^{i+1}$ maps infinite vertex groups onto infinite vertex groups.
\end{enumerate}
Let us consider the following diagram, where $\pi_i$ denotes the natural epimorphism from $A_{i+1}$ onto $L_i$:

\begin{align*}
    \xymatrix{
     & A_{i+1} \ar@{->>}[dr]^{\theta^{i+1}_\infty} \ar@{->>}[dl]_{\pi_i} &\\
     L_{i} & & L_{i+1}.\\}
\end{align*}
Let us make the following observation: for each vertex group $V\subset A_{i+1}$ of $\mathbb{J}'_i$ that does not contain $U$, the kernel of the restriction of $\pi_i$ to $V$ coincides with the kernel of the restriction of $\theta_{\infty}^{i+1}$ to $V$. Indeed, recall that $\theta_n^{i+1}$ is obtained by precomposing $\rho_n^{i+1}$ by an automorphism $\alpha$ of $A_{i+1}$ whose restriction to $V$ is a conjugation. As a consequence, the vertex groups $\pi_i(V)$ and $\theta^{i+1}_\infty(V)$ are isomorphic. But we know that $\pi_i(V)$ is infinite if and only if $V$ is infinite, and that in addition $\pi_i(V)$ and $V$ are isomorphic if they are finite, by construction of $\mathbb{J}'_i$ and $A_{i+1}$ (see Proposition \ref{approx} and Corollary \ref{rips_app}). Therefore $\theta^{i+1}_\infty(V)$ is infinite if and only if $V$ is infinite. In addition, $\theta^{i+1}_\infty(V)$ and $V$ are isomorphic if they are finite. Last, note that the image by $\theta_{\infty}^{i+1}$ of the vertex group of $\mathbb{J}'_i$ containing $U$ is infinite since $U$ is infinite and since $\theta_{\infty}^{i+1}$ is injective on $U$. Hence, the conditions (1) and (2) above are satisfied. Thus Lemma \ref{reduced_as_well} applies and tells us that $\mathbb{J}''_i$ is reduced.


\smallskip

Hence, if $\eta_i=\eta_{i+1}$, then $\mathbb{J}''_i$ is a reduced JSJ splitting of $L_{i+1}$ over finite groups of order less than $C$. It follows that the image of the vertex group $A_{i+1}^U$ in $L_{i+1}$ coincides with $L_{i+1}^U$. Therefore, one has $A_{i+1}^{U}=q_{i+1}(A_i^{U})$.\end{proof}

\begin{lemme}\label{reduced_as_well}
Let $G$ and $H$ be two groups, with two splittings $\mathbb{S}_G$ and $\mathbb{S}_H$ over finite groups. Let $T_G$ and $T_H$ denote the Bass-Serre trees of these splittings. Suppose that there exists an epimorphism $\theta : G \twoheadrightarrow H$ and a $\theta$-equivariant bijection $f : T_G \rightarrow T_H$ such that $\theta$ is injective on finite vertex groups and maps infinite vertex groups onto infinite vertex groups. Then, the following implication holds: if $T_G$ is reduced, then $T_H$ is reduced.\end{lemme}

\begin{proof}Suppose that $T_G$ is reduced. Let $\varepsilon=[v,w]$ be an edge of $T_H$ such that $H_v=H_{\varepsilon}=H_w$. We have to prove that $w$ is a translate of $v$, i.e.\ that there exists an element $h\in H$ such that $w=hv$. Let $e=[x,y]$ be a preimage of $\varepsilon$ by $f$. Since $H_{\varepsilon}$ is finite, $H_v$ and $H_w$ are finite, thus $G_x$ and $G_y$ are finite (indeed, by assumption, $\theta$ maps infinite vertex groups onto infinite vertex groups). Moreover, $\theta$ being injective on finite vertex groups, one has $G_x=G_{e}=G_y$. It follows that $y=gx$ for some $g\in G$. One has $f(y)=w$ and, since $f$ is $\theta$-equivariant, $f(gx)=\theta(g)f(x)=\theta(g)v$. Hence, $w=\theta(g)v$.\end{proof}

It remains to prove Lemma \ref{abovelemma}, whose statement is recalled below (for the sake of readability, the index $i+1$ is replaced with $i$).

\begin{lemme2}Let $F$ be the finite subgroup of $L_{i}$ defined in Lemma \ref{fini}. If $U$ is non-elliptic in the limiting tree of the sequence $({\theta_n^{i}})_{n\in\mathbb{N}}$, then the group $L_{i}$ admits a splitting $\mathbb{S}$ with exactly two vertex groups $\langle \bm{\ell}, \bm{a}\rangle$ (for some tuple $\bm{\ell}$ of elements of $L_{i}$) and $\langle \bm{x},F\rangle$, and one edge group $F$. Let $A$ be an $\mathbb{S}$-approximation of $L_{i}$ given by Proposition \ref{approx}. There exists a subgroup $E$ of $E(G)$ and an epimorphism $r$ from $A$ onto a group of the form \[\langle \bm{g},\bm{a}\rangle\ast_E \langle \bm{x},E \ \vert \ \mathrm{ad}(x_i)_{\vert E}={\sigma_i}_{\vert E}, \ \forall i\in \llbracket 1,p\rrbracket\rangle,\] where $\bm{g}$ denotes a tuple of elements of $G$, such that $r(\bm{x})=\bm{x}$, $r(\bm{a})=\bm{a}$ and $r$ kills no component of the image in $A$ of the tuple $\Psi(\bm{x},\bm{y},\bm{a})$.\end{lemme2}

\begin{proof}By assumption, the group $U$ is non-elliptic in the limiting tree $T:=T_{L_i^{U}}$ associated with the divergent sequence $({\theta_n^{i}}_{\vert A_{i}^U})_{n\in\mathbb{N}}$. First, we aim to construct a splitting $\mathbb{S}$ of $L_i$ with exactly two vertex groups $\langle \bm{x},F\rangle$ and $\langle \bm{\ell},\bm{a}\rangle$ (for some tuple $\bm{\ell}$ of elements of $L_{i}$) and one edge group $F$. Let $S$ be the stabilizer of the base point $o$ in $T$. Let $\Gamma$ be the subgroup $\langle U,S\rangle$ of $L_i^{U}$. Since $U$ is contained in $\Gamma$, this group is non-elliptic in the limiting tree $T$ (otherwise $U$ should be elliptic as well). Let $T_{\Gamma}\subset T$ be the minimal invariant subtree of $\Gamma$. By Lemma \ref{lemme2}, the tree $T_{\Gamma}$ is simplicial and $\Gamma$ admits the following splitting: \[\Gamma=\langle \bm{x},S \ \vert \ \mathrm{ad}(x_i)_{\vert F}=\alpha_i, \ \forall i\in\llbracket 1,p\rrbracket\rangle,\]where $F$ denotes the finite subgroup of $L$ defined in Lemma \ref{fini}, and $\alpha_i$ denotes the automorphism of $F$ induced by the action of $x_i$.

\smallskip

Let $\sim$ be the relation on $T$ defined by $x\sim y$ if $[x,y]\cap u T_{\Gamma}$ contains at most one point, for every element $u\in L_i^{U}$. Note that $\sim$ is an equivalence relation. Let $(Y_j)_{j\in J}$ denote the equivalence classes that are not reduced to a point. Each $Y_j$ is a subtree of $T$. Let us prove that $(Y_j)_{j\in J}\cup \lbrace uT_{\Gamma} \ \vert \ u\in L_i^{U}/\Gamma\rbrace$ is a transverse covering of $T$, in the sense of Definition \ref{transverse}.

\smallskip

\begin{itemize}
\item[$\bullet$]\emph{Transverse intersection.} For every $i\neq j$, the intersection $Y_i\cap Y_j$ is clearly empty. For every $i$ and $u\in L_i^{U}$, $Y_i\cap uT_{\Gamma}$ contains at most one point by definition. For every $u,u'\in L_i^{U}$ such that $u'u^{-1}\notin \Gamma$, $\vert uT_{\Gamma}\cap u'T_{\Gamma}\vert \leq 1$ thanks to Lemma \ref{lemme22}.
\item[$\bullet$]\emph{Finiteness condition.} Let $x$ and $y$ be two points of $T$. By Lemma \ref{lemme22}, there exists a constant $\varepsilon > 0$ such that, for every $u\in U$, if the intersection $[x,y]\cap uT_{\Gamma}$ is non-degenerate, the length of $[x,y]\cap uT_{\Gamma}$ is bounded from below by $\varepsilon$. Consequently, the arc $[x,y]$ is covered by at most $\lfloor d(x,y)/\varepsilon\rfloor$ translates of $T_{\Gamma}$ and at most $\lfloor d(x,y)/\varepsilon\rfloor+1$ distinct subtrees $Y_j$.
\end{itemize}

\smallskip

Hence, the collection $(Y_j)_{j\in J}\cup \lbrace uT_{\Gamma} \ \vert \ u\in L_i^{U}\rbrace$ is a transverse covering of $T$. One can construct what Guirardel calls the skeleton of this transverse covering (see Definition \ref{squelette}), denoted by $T_c$. Since the action of $L_i^{U}$ on $T$ is minimal (by definition of $T$), the same holds for the action of $L_i^{U}$ on $T_c$, according to Lemma 4.9 of \cite{Gui04}. The question is now to understand the decomposition $\Delta_c=T_c/L_i^{U}$ of $L_i^{U}$ as a graph of groups.

\smallskip

We begin with a description of the stabilizer in $L_i^{U}$ of an edge $e$ of $T_{\Gamma}$. Let $u$ be an element of $L_i^{U}$ that fixes $e$. Then $e$ is contained in $T_{\Gamma}\cap u T_{\Gamma}$, so $u$ belongs to $\Gamma$, thanks to Lemma \ref{lemme22}. It follows that $u$ belongs to $F$, because the stabilizer of $e$ in $\Gamma$ is contained in $F$ (indeed, recall that $T_{\Gamma}$ is isometric to the Bass-Serre tree of the splitting $\Gamma=\langle \bm{x},\bm{a},S \ \vert \ \mathrm{ad}(x_i)_{\vert F}=\alpha_i, \forall i\in\llbracket 1,p\rrbracket\rangle$, by Lemma \ref{lemme2}). Thus, the stabilizer of $e$ in $L_i^{U}$ is equal to $F$. 

\smallskip

We now prove that if one of the subtrees of the covering other than $T_{\Gamma}$ intersects $T_{\Gamma}$ in a point, then this point is necessarily one of the extremities of a translate of the edge $e\in T_{\Gamma}$. Assume towards a contradiction that $Y_j$ or $uT_{\Gamma}$ with $u\notin \Gamma$ intersects $T_{\Gamma}$ in a point $x$ that is not one of the extremities of $e$. Then, $T_c$ contains an edge $\varepsilon=(x,T_{\Gamma})$ whose stabilizer is $\mathrm{Stab}(x) \cap \Gamma$ (where $\mathrm{Stab}(x)$ denotes the stabilizer of $x$ in $L_i^{U}$), which is contained in $F$ by the previous paragraph. So the splitting $\Delta_c$ of $L_i^{U}$ is a non-trivial splitting over the finite subgroup $F$, relative to $\Gamma$. This is impossible since $\vert F\vert \leq C$ (because $\varphi_n$ maps $F$ into $E(G)$ $\omega$-almost-surely) and $L_i^{U}$ does not split relative to $\Gamma$ over a finite subgroup of order $\leq C$ non-trivially, by definition of $L_i^{U}$. Hence, if $Y_j\cap T_{\Gamma}=\lbrace x\rbrace$ or $uT_{\Gamma}\cap T_{\Gamma}=\lbrace x\rbrace$ with $u\notin \Gamma$, then the point $x$ is one of the extremities of $e$ in $T_{\Gamma}$. As a consequence, $\mathrm{Stab}(x)$ is a conjugate of $S$ in $\Gamma$, and every edge adjacent to $T_{\Gamma}$ in $T_c$ is of the form $(\gamma x,T_{\Gamma})=\gamma\varepsilon$ with $\varepsilon=(x,T_{\Gamma})$. 

\smallskip

Therefore, $\varepsilon$ is the only edge adjacent to $T_{\Gamma}$ in the quotient graph $\Delta_c$. Its stabilizer is $S$. By collapsing all edges of $\Delta_c$ except $\varepsilon$, one gets a splitting of $L_i^{U}$ of the following form: $L_i^{U}=\Gamma\ast_S H$ for some subgroup $H\subset L_i^{U}$. Recall that $\Gamma=\langle S,\bm{x} \ \vert \ \mathrm{ad}(x_i)_{\vert F}=\alpha_i, \ \forall i\in\llbracket 1,p\rrbracket\rangle$. Hence, the previous splitting of $L_i^{U}$ can be written as \[L_i^{U}=\langle H, \bm{x} \ \vert \ \mathrm{ad}(x_i)_{\vert F}=\alpha_i, \ \forall i\in\llbracket 1,p\rrbracket\rangle.\]

\smallskip

Since every finite subgroup of $L_i^{U}$ is conjugate to a finite subgroup of $H$, the group $L_i$ splits as \[L_i=\langle K, \bm{x} \ \vert \ \mathrm{ad}(x_i)_{\vert F}=\alpha_i, \ \forall i\in\llbracket 1,p\rrbracket\rangle,\]for some subgroup $K$ of $L_{i}$ such that $\langle \bm{a}\rangle \subset S\subset H\subset K$. Last, one can rewrite this splitting in the following form: $L_i=K\ast_F\langle \bm{x},F\rangle$. Denote this splitting by $\mathbb{S}$.

\smallskip

If $\bm{a}$ were empty, one could just retract $L_i$ onto the free group $F(\bm{x})$ on $\bm{x}$. But $\bm{a}$ is not empty in general, which makes the construction of the retraction a little bit more involved.

\smallskip

Let $A$ be an $\mathbb{S}$-approximation of $L_{i}$ given by Proposition \ref{approx}, and let $\mathbb{S}_A$ be the corresponding splitting of $A$. By Remark \ref{trick}, one can assume that the components of $\Psi_k(\bm{x},\bm{y},\bm{a})$ in $L_i$ and $A$ have exactly the same normal forms when written in $\mathbb{S}$ and $\mathbb{S}_A$. The splitting $\mathbb{S}_A$ is of the form \[A=K'\ast_F\langle \bm{x},F\rangle=\langle K',\bm{x} \ \vert \ \mathrm{ad}(x_i)_{\vert F}=\alpha_i, \ \forall i\in\llbracket 1,p\rrbracket\rangle.\]Note that we abuse notation and still denote by $F$ a preimage of $F\subset L_i$ in $A$. Same comment about $\bm{x}$, and $\bm{a}$ (which is contained in $K'$).

\smallskip

We claim that there exists a subgroup $E$ of $E(G)$ and an epimorphism $r$ from $A$ onto a group of the form \[\langle\bm{g},\bm{a}\rangle\ast_E \langle \bm{x},E \ \vert \ \mathrm{ad}(x_i)_{\vert E}={\sigma_i}_{\vert E}, \ \forall i\in \llbracket 1,p\rrbracket\rangle,\] where $\bm{g}$ denotes a tuple of elements of $G$, such that $r(\bm{x})=\bm{x}$, $r(\bm{a})=\bm{a}$ and $r$ kills no component of the image in $A$ of the tuple $\Psi(\bm{x},\bm{y},\bm{a})$.

\smallskip

For every integer $n$, denote by $\psi_n : A \rightarrow G$ the factorization of the homomorphism $\theta_n^i : A_i \rightarrow G$ through the natural epimorphism from $A_i$ onto $A$. 

\begin{align*}
        \xymatrix{
        A_i \ar@{->>}[drr] \ar[rrrr]^{\theta_n^i} \ar@{->>}[ddrr]_{\theta^i_\infty} && && G \\
        && A \ar[rru]_{\psi_n} \ar@{->>}[d] \\
        && L_i}
    \end{align*}

This homomorphism $\psi_n$ restricts to a conjugation on $\langle \bm{x},\bm{a}\rangle$. Up to postcomposing $\psi_n$ with an inner automorphism of $G$, one can now assume without loss of generality that $\psi_n$ coincides with the identity on $\langle \bm{x},\bm{a}\rangle$. In particular, the inner automorphism $\mathrm{ad}(\psi_n(x_i))$ induces the same automorphism $\sigma_i$ of $E(G)$ as $\mathrm{ad}(\varphi_n(x_i))$, where $(\varphi_n : G_{\Sigma}\rightarrow G)_{n\in\mathbb{N}}$ denotes the initial $(\sigma_1,\ldots,\sigma_p)$-test sequence.

\smallskip

For every integer $n$, since $\psi_n$ is the identity on $\bm{a}$, the group $\psi_n(K')$ contains $\bm{a}$. Since $A$ is finitely presented relative to $U=\langle\bm{x},\bm{a}\rangle$, and since $F$ is a finite group, $K'$ and $\psi_n(K')$ are finitely generated relative to $\bm{a}$. Therefore, there exists a tuple $\bm{g}_n$ of elements of $G$ such that $\psi_n(K')=\langle\bm{g}_n,\bm{a}\rangle$. Let $E:=\psi_n(F)\subset E(G)$. Let's consider the following amalgamated product:\[Q_n=\langle \bm{g}_n,\bm{a}\rangle\ast_{E}\langle \bm{x}, E \ \vert \ \mathrm{ad}(x_i)_{\vert E}=\sigma_i, \ \forall i\in\llbracket 1,p\rrbracket\rangle.\]For every integer $n$, one can define a morphism $\pi_n$ from $A$ onto $Q_n$ by $\pi_n(x_i)=x_i$ and $\pi_n=\psi_n$ on $K'$. This morphism is well-defined. Indeed, for every integer $1\leq i\leq p$, as $x_i$ normalizes $F$, there exists an automorphism $\alpha_i$ of $F$ such that $x_ifx_i^{-1}=\alpha_i(f)$ for every $f\in F$. The following relation holds: \[\psi_n\circ \alpha_i=\sigma_i\circ \psi_n\] for every integer $n$. This relation shows that $\pi_n$ is well-defined. In addition, this morphism is surjective because its image contains $\bm{x}$ and $\psi_n(K')=\langle\bm{g}_n,\bm{a}\rangle$, which generate the group $Q_n$. It remains to prove that, $\omega$-almost-surely, $\pi_n$ kills no component of the image of $\Psi_k(\bm{x},\bm{y},\bm{a})$ in $A$.

\smallskip

Let $v$ be component of the image of $\Psi_k(\bm{x},\bm{y},\bm{a})$ in $A$. This element can be written in normal form in the splitting $\mathbb{S}_A$ as $v=k'_0t_1^{\varepsilon_1}k'_1t_2^{\varepsilon_2}k'_2\cdots t_q^{\varepsilon_q}k'_{q+1}$, with $k'_i\in K'$ and $t_j\in\lbrace x_1,\ldots,x_p\rbrace$ for every $1\leq j\leq q$. For every $j$, if $t_j=t_{j+1}=x_i$ and $\varepsilon_j=-\varepsilon_{j+1}$, then $k'_j\notin F$. By Remark \ref{trick}, the image of $v$ in $L_i$ can be written in normal form in a similar way, by replacing each $k'_i$ by an element $k_i$ that belongs to the subgroup $K$ of $L$. Therefore, for every $j$, if $t_j=t_{j+1}=x_i$ and $\varepsilon_j=-\varepsilon_{j+1}$, then $k_j$ does not belong to $F$. It follows that $\pi_n(k'_j)=\psi_n(k'_j)$ does not lie in $E$ $\omega$-almost-surely. Otherwise, if $\pi_n(k'_j)$ belonged to $E$ $\omega$-almost-surely, then $k_j$ would belong to $F$ $\omega$-almost-surely, contradicting the previous condition. Hence, for every $n$ large enough, the element $\pi_n(v)=\psi_n(k_0)t_1^{\varepsilon_1}\psi_n(k_1)t_2^{\varepsilon_2}\psi_n(k_2)\cdots t_q^{\varepsilon_q}\psi_n(k_{q+1})$ is non-trivial. Last, take $r=\pi_N$ for $N$ such that $\pi_N$ kills no component of $\Psi_k(\bm{x},\bm{y},\bm{a})$ in $A$.\end{proof}

\section{Proof of Theorem \ref{th11}}\label{proof_main_th}

In this section, we prove Theorem \ref{th11}. First, recall that this theorem says that every acylindrically hyperbolic group $G$ is $\exists\forall\exists$-embedded into the HNN extensions $G\dot{\ast}_{E(G)}=\langle G,t \ \vert \ [t,g]=1, \ \forall g\in E(G)\rangle$. In fact, we just have to prove that $G$ is $\forall\exists$-embedded into $G\dot{\ast}_{E(G)}$, in virtue of the following easy and general lemma, which has nothing to do with acylindrical hyperbolicity.

\begin{lemme}
Let $G'$ be a group, and let $G$ be a subgroup of $G$. If $G$ is $\forall\exists$-embedded into $G'$, then $G$ is $\exists\forall\exists$-embedded into $G'$.
\end{lemme}

\begin{proof}Suppose that $G$ is $\forall\exists$-embedded into $G'$. Let $\theta(\bm{t})$ be an $\exists\forall\exists$-formula with $m$ free variables. Suppose that there exists a tuple $\bm{g}\in G^m$ such that $\theta(\bm{g})$ holds in $G$, and prove that $\theta(\bm{g})$ holds in $\Gamma$.

\smallskip

The formula $\theta(\bm{t})$ can be written as $\exists \bm{x} \ \mu(\bm{t},\bm{x})$, where $\mu(\bm{t},\bm{x})$ denotes a $\forall\exists$-formula with $m+n$ free variables, where $n$ is the arity of $\bm{x}$. Since $\theta(\bm{g})$ holds in $G$, there exists a tuple $\bm{h}\in G^n$ such that $\mu(\bm{g},\bm{h})$ holds in $G$. But the formula $\mu(\bm{t},\bm{x})$ is $\forall\exists$, thus $\mu(\bm{g},\bm{h})$ holds in $\Gamma$. This concludes the proof of the lemma.\end{proof}

In order to prove Theorem \ref{th11}, it remains to prove that every acylindrically hyperbolic group $G$ is $\forall\exists$-embedded into $G\dot{\ast}_{E(G)}$. The proof of this result relies on Theorem \ref{th0bis}. 

\begin{te}\label{sacerdote2}Every acylindrically hyperbolic group $G$ is $\forall\exists$-embedded into $G\dot{\ast}_{E(G)}$.\end{te}

\begin{proof}Let \[\bigvee_{k=1}^{\ell}(\Sigma_k(\bm{x},\bm{y},\bm{g})=1 \ \wedge \ \Psi_k(\bm{x},\bm{y},\bm{g})\neq 1)\] be a finite disjunction of systems of equations and inequations in $\bm{x}$ and $\bm{y}$. Suppose that $G$ satisfies the following first-order sentence $\mu(\bm{g})$:\[\forall \bm{x} \ \exists \bm{y} \ \bigvee_{k=1}^{\ell}(\Sigma_k(\bm{x},\bm{y},\bm{g})=1 \ \wedge \ \Psi_k(\bm{x},\bm{y},\bm{g})\neq 1).\]Let $\bm{\gamma}$ be a tuple of elements of $\Gamma$ of the same arity as $\bm{x}$. We will prove that there exists a tuple $\bm{\gamma}'$ of elements of $\Gamma$ of the same arity as $\bm{y}$ such that the following holds in $\Gamma$: \[\bigvee_{k=1}^{\ell}(\Sigma_k(\bm{\gamma},\bm{\gamma'},\bm{g})=1 \ \wedge \ \Psi_k(\bm{\gamma},\bm{\gamma'},\bm{g})\neq 1).\]To this end, we would like to construct a retraction $\pi$ from the group $\langle \Gamma, \bm{y} \ \vert \ \Sigma_k(\bm{\gamma},\bm{y},\bm{g})=1\rangle$ onto $\Gamma$, for some $1\leq k\leq \ell$, such that $\pi$ kills no component of the system of inequations $\Psi_k(\bm{\gamma},\bm{y},\bm{g})\neq 1$.  Indeed, given such a retraction $\pi$, one can simply take $\bm{\gamma'}=\pi(\bm{y})$. We could construct this retraction by mimicking the proof of Theorem \ref{th0bis}, as sketched in the introduction, but in order to avoid unnecessary repetitions, we will appeal to Theorem \ref{th0bis}. However, before applying this result, one has to fix the following problem: Theorem \ref{th0bis} does not apply directly in the present situation since it only allows us to deal with constants from $G$, and $\bm{\gamma}$ is not a tuple of elements of $G$ in general. In order to be able to use Theorem \ref{th0bis}, we have first to slightly reformulate the problem.

\smallskip

Let $\bm{s}$ be a generating tuple of $G$, possibly infinite. For every integer $n\geq 1$, let $\bm{s}_n$ be the $n$-tuple composed of the first $n$ components of $\bm{s}$ and let $G_n$ be subgroup of $G$ generated by $\bm{s}_n$. For $n$ sufficiently large, the following two conditions are satisfied.
\begin{itemize}
    \item[$\bullet$]The subgroup $G_n$ of $G$ contains the finite subgroup $E(G)$. Therefore, there is a finite system of equations $\theta(\bm{s}_n,t)=1$ expressing the fact that the stable letter $t$ centralizes $E(G)$.
    \item[$\bullet$]The subgroup $\langle G_n,t\rangle$ of $\Gamma$ contains each component $\gamma_i$ of $\bm{\gamma}$. As a consequence, each $\gamma_i$ can be written as a word $w_i(\bm{s}_n,t)$.
\end{itemize}

\smallskip

Let $\bm{a}$ be the tuple of elements of $G$ obtained by concatenating $\bm{g}$ and $\bm{s}_n$. Let $\Sigma'_{k}(t,\bm{y},\bm{a})=1$ denote the finite system of equations \[\left(\Sigma_k((w_1(\bm{s}_n,t),\ldots,w_p(\bm{s}_n,t)),\bm{y},\bm{g})=1\right) \ \wedge \ \left(\theta(\bm{s}_n,t)=1\right),\] and let $\Psi'_{k}(t,\bm{y},\bm{a})\neq 1$ denote the finite system of inequations \[\Psi_{k}((w_1(\bm{s}_n,t),\ldots,w_p(\bm{s}_n,t)),\bm{y},\bm{g})\neq 1.\] By assumption, the group $G$ satisfies $\mu(\bm{g})$. Therefore, $G$ satisfies the following first-order sentence $\theta(\bm{a})$: \[\forall t \ \exists \bm{y} \ \bigvee_{k=1}^{\ell}(\Sigma'_{k}(t,\bm{y},\bm{a})=1 \ \wedge \ \Psi'_{k}(t,\bm{y},\bm{a}_n)\neq 1).\]By Theorem \ref{th0bis}, there exist an integer $1\leq k\leq \ell$, a subgroup $G'$ of $G$ containing $\langle \bm{a}\rangle$ and an epimorphism \[\pi : G_{\Sigma'_{k}}\twoheadrightarrow \Gamma':=(\langle t\rangle\times E(G))\ast_{E(G)} G'\] such that
\begin{enumerate}
    \item $\pi(t)=t$,
    \item $\pi(\bm{a})=\bm{a}$ (in particular $\pi(\bm{g})=\bm{g}$ and $\pi(\bm{s}_n)=\bm{s}_n$, and therefore $\pi(\bm{\gamma})=\bm{\gamma}$),
    \item and such that no component of the system of inequations $\Psi'_{k}(t,\bm{y},\bm{a})\neq 1$ is killed by $\pi$.
\end{enumerate}
As a consequence, the following system of equations and inequations holds in $\Gamma'$: \[\bigvee_{k=1}^{\ell}(\Sigma'_k(t,\pi(\bm{y}),\bm{a})=1 \ \wedge \ \Psi_k(t,\pi(\bm{y}),\bm{a})\neq 1).\]It follows that the following system of equations and inequations holds in $\Gamma'$:\[\bigvee_{k=1}^{\ell}(\Sigma_k(\bm{\gamma},\pi(\bm{y}),\bm{g})=1 \ \wedge \ \Psi_k(\bm{\gamma},\pi(\bm{y}),\bm{g})\neq 1).\]
Since $\Gamma'$ is a subgroup of $\Gamma$, this system holds in $\Gamma$ as well. One can take $\bm{\gamma}'=\pi(\bm{y})$.\end{proof}

\section{Proof of Merzlyakov's theorem \ref{th0bis2} in the general case}

\subsection{Reduction to an overgroup $G_{2p}$ of $G$}

As above, $G$ denotes an acylindrically hyperbolic group, and $p$ denotes the arity of $\bm{x}$ in the considered first-order sentence. In the proof of Proposition \ref{jesaispas}, for defining a test sequence (i.e.\ a $(\sigma_1,\ldots,\sigma_p)$-test sequence with $\sigma_i=\mathrm{id}_{E(G)}$ for every $1\leq i\leq p$), we used the fact that $G$ contains a quasi-convex non-abelian free subgroup $F_2$ that centralizes $E(G)$. It seems quite involved to adapt this construction in order to get a non-central prescribed action of $F_2$ by conjugation on $E(G)$. 

\smallskip

We shall circumvent this difficulty by means of Theorem \ref{th1} proved in the previous section. According to this result, the inclusion of $G$ into $G\ast_{E(G)}(E(G)\times \mathbb{Z})$ is an $\exists\forall\exists$-embedding. More generally, the inclusion of $G$ into $G_m:=G\ast_{E(G)}(E(G)\times F_m)$ is an $\exists\forall\exists$-embedding, for any integer $m$. Take $m=2p$, and let $t_1,...,t_{2p}$ be a basis of $F_{2p}$. Let $(\sigma_1,\ldots,\sigma_p)$ be a $p$-tuple of elements of $\mathrm{Aut}_G(E(G))$. For every $1\leq i\leq p$, there exists an element $g_i\in G$ such that $\sigma_i=\mathrm{ad}(g_i)_{\vert E(G)}$, by definition of $\mathrm{Aut}_G(E(G))$. Note that $\sigma_i=\mathrm{ad}(g_it_i)_{\vert E(G)}$, since $t_i$ centralizes $E(G)$. Let $\alpha_i$ be the automorphim of $G_{2p}$ that coincides with $\mathrm{ad}(g_i)$ on $G$ and that maps $t_i$ to $g_it_i$ and $t_j$ to $t_j$ for $j\neq i$. The composition $\alpha_1\circ \cdots\circ \alpha_p$ is an automorphism of $G_{2p}$ that coincides with the conjugacy by $g_1\cdots g_p$ on $G$ and maps $t_i$ to $g_it_i$ for $1\leq i\leq p$, and fixes $t_i$ for $p+1\leq i\leq 2p$. As a consequence, up to replacing $t_i$ by $g_it_i$, one can assume without loss of generality that $\mathrm{ad}(t_i)_{\vert E(G)}=\sigma_i$ for every $1\leq i\leq p$, and $G_{2p}$ has the following presentation:\[G_{2p}=G\ast_{E(G)}\Biggl\langle 
       \begin{array}{l|cl}
                        & \mathrm{ad}(t_i)_{\vert E(G)}=\sigma_i \text{ for $1\leq i\leq p$} \\
            E(G), t_1,\ldots,t_{2p}&    \\
                        & \mathrm{ad}(t_i)_{\vert E(G)}=\mathrm{id} \text{ for $p+1\leq i\leq 2p$}                                             
        \end{array}
     \Biggr\rangle.\]

\subsection{Construction of a $(\sigma_1,\ldots,\sigma_p)$-test sequence}

We now build a $(\sigma_1,\ldots,\sigma_p)$-test sequence from $G_{\Sigma_k}$ to $G_{2p}$, for any $(\sigma_1,\ldots,\sigma_p)$ in $\mathrm{Aut}_G(E(G))^p$.

\begin{prop}\label{jesaispasgeneral}Let $G$ be an acylindrically hyperbolic group, and let $\bm{a}$ be a tuple of elements of $G$. Fix a presentation $\langle \bm{a} \ \vert \ R(\bm{a})=1\rangle$ for the subgroup of $G$ generated by $\bm{a}$. Let \[\bigvee_{k=1}^{\ell}(\Sigma_k(\bm{x},\bm{y},\bm{a})=1 \ \wedge \ \Psi_k(\bm{x},\bm{y},\bm{a})\neq 1)\] be a finite disjunction of finite system of equations and inequations over $G$, where $\bm{x}$ and $\bm{y}$ are tuples of variables. For every $1\leq k\leq \ell$, denote \[G_{\Sigma_k}=\langle \bm{x},\bm{y},\bm{a} \ \vert \ R(\bm{a})=1, \ \Sigma_k(\bm{x},\bm{y},\bm{a})=1\rangle.\]Let $p=\vert \bm{x}\vert$ be the arity of $\bm{x}$, and let $(\sigma_1,\ldots,\sigma_p)$ be a $p$-tuple of elements of $\mathrm{Aut}_G(E(G))$. Suppose that $G$ satisfies the following first-order sentence:\[\theta:\forall \bm{x} \ \exists \bm{y} \ \bigvee_{k=1}^{\ell}(\Sigma_k(\bm{x},\bm{y},\bm{a})=1 \ \wedge \ \Psi_k(\bm{x},\bm{y},\bm{a})\neq 1).\]Then there exist an integer $1\leq k\leq \ell$ and a $(\sigma_1,\ldots,\sigma_p)$-test sequence $(\varphi_n : G_{\Sigma_k} \rightarrow G_{2p})_{n\in\mathbb{N}}$ such that $\varphi_n(\Psi_k(\bm{x},\bm{y},\bm{a}))$ is non-trivial for every $n$ sufficiently large.\end{prop}

\begin{proof}Recall that $G_{2p}$ has the following presentation:
\[G_{2p}=G\ast_{E(G)}\Biggl\langle 
       \begin{array}{l|cl}
                        & \mathrm{ad}(t_i)_{\vert E(G)}=\sigma_i \text{ for $1\leq i\leq p$} \\
            E(G), t_1,\ldots,t_{2p}&    \\
                        & \mathrm{ad}(t_i)_{\vert E(G)}=\mathrm{id} \text{ for $p+1\leq i\leq 2p$}                                             
        \end{array}
     \Biggr\rangle.\]
     
     \smallskip

For every $1\leq i\leq p$ and for every integer $n\geq 1$, let $o_i$ be the order of $\sigma_i$, let $r_n$ be the remainder of the division of $n$ by $o_i$, and let $q_n=o_i+1-r_n$. Note that one has $2\leq q_n\leq o_i+1$ and that $n+q_n=1\mod o_i$. Let us define an element $g_{i,n}$ of $G$ by \[g_{i,n}=t_{i+p}^nt_it_{i+p}^{n+1}t_i\cdots t_{i+p}^{2n}t_i^{q_n}.\]Observe that $\mathrm{ad}(g_{i,n})_{\vert E(G)}=\sigma_i$, thanks to our choice of $q_n$. 

\smallskip

Since the inclusion of $G$ into $G_{2p}$ is an $\exists\forall\exists$-embedding, the group $G_{2p}$ also satisfies the first-order sentence $\theta$. By the pigeonhole principle, there exists an integer $1\leq k\leq \ell$ and an infinite set $A\subset\mathbb{N}$ such that for every integer $n\in A$, the group $G_{2p}$ satisfies the following existential sentence: \[\exists \bm{y}_n \ \Sigma_k((g_{1,n},\ldots,g_{p,n}),\bm{y}_n,\bm{a})=1\wedge \Psi_k((g_{1,n},\ldots,g_{p,n}),\bm{y}_n,\bm{a})\neq 1.\]
Let us define $\varphi_n : G_{\Sigma}\rightarrow G_{2p}$ by $\varphi_n(x_i)=g_{i,n}$, $\varphi_n(\bm{a})=\bm{a}$ and $\varphi_n(\bm{y})=\bm{y}_n$. One can check that the sequence $(\varphi_n : G_{\Sigma_k}\rightarrow G_{2p})_{n\in\mathbb{N}}$ is a $(\sigma_1,\ldots,\sigma_p)$-test sequence.\end{proof}

\subsection{Proof of Merzlyakov's theorem \ref{th0bis}}

Theorem \ref{th0bis} is an immediate consequence of Proposition \ref{jesaispasgeneral} and Proposition \ref{jesaispas2} applied to $G_{2p}$ instead of $G$.

\section{Trivial positive theory and verbal subgroups}\label{sectionpos}

In this section, we give two proofs of Corollary \ref{pos}, which claims that acylindrically hyperbolic groups have trivial positive theory. We also deduce Corollary \ref{BBF} about verbal subgroups of acylindrically hyperbolic groups.

\smallskip

The first proof of Corollary \ref{pos} relies on Theorem \ref{th1}.

\begin{proof}Let $G$ be an acylindrically hyperbolic group. By Theorem 6.3 of \cite{CGN19}, if a group satisfies a non-trivial positive sentence, then it also satisfies a non-trivial positive $\forall\exists$-sentence. As a consequence, in order to prove that $G$ has trivial positive theory, it suffices to prove that $G$ has trivial positive $\forall\exists$-theory. Let $\theta$ be a positive $\forall\exists$-sentence satisfied by $G$. Let $E(G)$ denote the maximal finite normal subgroup of $G$. It follows from Theorem \ref{th1} that the groups $G$ and $\Gamma=\langle G, x,y \ \vert \ [x,g]=[y,g]=1, \forall g\in E(G)\rangle$ have the same $\forall\exists$-theory. As a consequence, $\theta$ is satisfied by $\Gamma$. Now, observe that $\Gamma$ maps onto the free group $\langle x,y\rangle\simeq F_2$. Since positive sentences are preserved under epimorphisms, $\theta$ is satisfied by $F_2$. It follows that $\theta$ is satisfied by all free groups. Therefore, $\theta$ holds in all groups.\end{proof}

The second proof relies on Theorem \ref{th0bis}.

\begin{proof}Let $\theta$ be a positive $\forall\exists$-sentence satisfied by $G$. Classically, $\theta$ is equivalent to a first-order sentence of the form $\forall \bm{x} \ \exists \bm{y} \ \bigvee_{k=1}^{\ell}\Sigma_k(\bm{x},\bm{y})=1$. It follows easily from Theorem \ref{th0bis2} that there exists an integer $1\leq k\leq \ell$ and an epimorphism $\pi : G_{\Sigma_k}\twoheadrightarrow F(\bm{x})$. The image $\pi(\bm{y})$ of $\bm{y}$ can be written as a word $w(\bm{x})$ in the free group $F(\bm{x})$, and the following equality holds in $F(\bm{x})$: $\Sigma_k(\bm{x},w(\bm{x}))=1$. Now, for any group $H$ and any tuple $\bm{h}$ of elements of $H$ of the same arity as $\bm{x}$, the evaluation map $\phi_{\bm{h}} : F(\bm{x})\rightarrow H : \bm{x}\mapsto \bm{h}$ maps $\Sigma_k(\bm{x},w(\bm{x}))=1$ to $\Sigma_k(\bm{h},w(\bm{h}))=1$ (in other words, $\phi_{\bm{h}}$ extends to a morphism $\bar{\phi}_{\bm{h}} : G_{\Sigma_k}\rightarrow H$ defined by $\bar{\phi}_{\bm{h}}(\bm{y})=w(\bm{h})$). This concludes the proof.\end{proof}

Last, prove Corollary \ref{BBF}. Let $w$ be a non-trivial element of the free group $F(x_1,\ldots,x_k)$. Let us denote by $e_i$ the sum of the exponents of $x_i$ in $w$. If they are all $0$, define $d(w) = 0$. Otherwise, let $d(w)$ be their greatest common divisor. Since acylindrically hyperbolic groups have trivial positive theory, Corollary \ref{BBF} is an immediate consequence of the following lemma.

\begin{lemme}\label{lemmeposintro}Let $G$ be a group, let $k\geq 1$ be an integer and let $w$ be an element of $F_k$. If $G$ has trivial positive theory, then $w(G)$ has infinite width, except if $w$ is trivial or $d(w)=1$ (in which cases the width is equal to $1$).\end{lemme}

\begin{proof}
If $w$ is trivial, then the width of $w(G)$ is $1$. Now, suppose that $w$ is non-trivial and that $d(w)=1$. Then there exist $k$ integers $a_1,\ldots ,a_k$ such that $a_1e_1+\cdots + a_ke_k=1$, and one has $w(g^{a_1},\ldots,g^{a_k})=g$. Hence $w(G)$ is equal to $G$, and its width is equal to $1$.

\smallskip

Assume towards a contradiction that there exists a non-trivial element $w\in F_k$ with $d(w)\neq 1$ and such that $w(G)$ has finite width $\ell$. Then, $G$ satisfies the following positive first-order ($\forall\exists$)-sentence $\phi_{n}$, for every integer $n\geq 1$: every element of $g$ that can be represented as a product of $n$ elements of $\lbrace w(\bm{g})^{\pm 1}, \ \bm{g}\in G^k\rbrace$ can be represented as a product of $\ell$ elements of $\lbrace w(\bm{g})^{\pm 1}, \ \bm{g}\in G^k\rbrace$. Since $G$ has trivial positive theory, this sentence $\phi_n$ is satisfied by all groups. In particular, $\phi_n$ is true in the free group $F_2$, for every $n$. Thus, $w(F_2)$ has finite width (equal to $\ell$). It follows from Lemma 3.1.1 and Theorem 3.1.2 in \cite{Seg09} (inspired from \cite{Rhem68}) that either $w$ is trivial or $d(w)=1$, contradicting our assumption.\end{proof}


\section{Questions and comments}\label{comments_section}

In \cite{Sel10}, Sela asked the following intriguing question.

\begin{qu}Which (algebraic, first-order) properties are satisfied by groups $G$ such that $G$ and $G\ast\mathbb{Z}$ are elementarily equivalent?\end{qu}

If the answer to the generalised Tarski's problem \ref{question} is `Yes', then every acylindrically hyperbolic group $G$ with trivial finite radical $E(G)$ is elementarily equivalent to $G\ast\mathbb{Z}$. As far as we are aware, no examples are known of finitely generated groups that have this property but are not acylindrically hyperbolic. This raises the following question.

\begin{qu}\label{question3}Is there a finitely generated group $G$ that is not acylindrically hyperbolic but is such that $G$ and $G\ast\mathbb{Z}$ are elementarily equivalent (or at least have the same $\forall\exists$-theory)?\end{qu}

This question is closely related to the following one (see Proposition \ref{relatedbelow}).

\begin{qu}\label{question2}Is acylindrical hyperbolicity preserved under elementary equivalence among finitely generated groups?\end{qu} 

In \cite{And18}, the first author proved that the property of being a hyperbolic group is preserved under elementary equivalence among finitely generated groups (this result was proved by Sela in \cite{Sel09} for torsion-free groups). Since acylindrically hyperbolic groups are not supposed to be finitely generated, Question \ref{question2} makes sense without assuming finite generation; however, the answer to this question is negative in general, even among countable groups. We refer the reader to \cite{And20} for further details.

\smallskip

The following result shows that a positive answer to Question \ref{question3} implies a negative answer to Question \ref{question2}, and that the converse is true under the assumption that the answer to the generalised Tarski's problem \ref{question} is `Yes'.

\begin{prop}\label{relatedbelow}If there exists a finitely generated non-acylindrically hyperbolic group $G$ such that $G$ and $G\ast\mathbb{Z}$ are elementarily equivalent, then acylindrical hyperbolicity is not preserved under elementary equivalence among finitely generated groups. Conversely, under the assumption that Question \ref{question} admits a positive answer, the following implication holds: if acylindrical hyperbolicity is not preserved under elementary equivalence among finitely generated groups, then there exists a finitely generated non-acylindrically hyperbolic group $G$ such that $G$ and $G\ast\mathbb{Z}$ are elementarily equivalent.\end{prop}

\begin{proof}If there exists a finitely generated group $G$ such that $G$ and $G\ast\mathbb{Z}$ are elementarily equivalent, and $G$ is not acylindrically hyperbolic, then acylindrical hyperbolicity is not preserved under elementary equivalence among finitely generated groups since $G\ast\mathbb{Z}$ is acylindrically hyperbolic.

\smallskip

Now, assume that the answer to Question \ref{question2} is negative, namely that there exist two elementarily equivalent finitely generated groups $G$ and $H$ such that $G$ is not acylindrically hyperbolic and $H$ is acylindrically hyperbolic. Observe that the maximal normal finite subgroup $E(H)$ coincides with the definable set $D_N(H)=\lbrace h'\in H \ \vert \ [h^N,h']=1, \ \forall h\in H\rbrace$ for $N=\vert \mathrm{Aut}(E(H))\vert$: indeed, the fact that $E(H)$ is contained in $D_N(H)$ is obvious since any element of $H$ induces an automorphism of $E(H)$ by conjugacy; conversely, by \cite[Theorem 6.14]{DGO17}, $E(H)$ is the intersection of all maximal virtually cyclic subgroups $\Lambda (h)$, where $h$ runs through all hyperbolic elements of $H$, and thus it follows from Lemma \ref{lemmefin} that the set $D_N(H)$ is contained in $E(H)$. Since this set is definable, $D_N(G)$ is isomorphic to $D_N(H)=E(H)$, and the quotients $G'=G/D_N(G)$ and $H'=H/D_N(H)$ are elementarily equivalent. Note that $H'$ is acylindrically hyperbolic since $H$ is acylindrically hyperbolic (see \cite[Lemma 3.9]{MO15}). In addition, by \cite{Sel10}, $G'\ast\mathbb{Z}$ and $H'\ast\mathbb{Z}$ are elementarily equivalent. Now, if the generalised Tarski's problem \ref{question} admits a positive answer, then $H'\ast\mathbb{Z}$ is elementarily equivalent to $H'$, which is elementarily equivalent to $G'$. Hence, $G'$ and $G'\ast\mathbb{Z}$ are elementarily equivalent. But $G'$ is not acylindrically hyperbolic, otherwise $G$ would be acylindrically hyperbolic as well, as a finite extension of $G$. Thus, the answer to Question \ref{question3} is `Yes'.\end{proof}

Last, it is worth mentioning the following partial answer to Question \ref{question2}, following from Theorem \ref{th11} together with a theorem of Minasyan and Osin that gives a sufficient condition under which a group $H=A\ast_C B$ or $H=A\ast_C$ is acylindrically hyperbolic (Corollaries 2.2 and 2.3 in \cite{MO15}).

\begin{prop}\label{acyl_preservation}
Let $G$ be an acylindrically hyperbolic group, and let $H$ be a group that admits a non-trivial splitting over a virtually abelian group. Suppose that $G$ and $H$ are elementarily equivalent (or simply that they have the same $\exists\forall\exists$-theory). Then, $H$ is acylindrically hyperbolic.
\end{prop}

A subgroup $C$ of $H$ is said to be \emph{weakly malnormal} in $H$ if there exists an element $h\in H$ such that $hCh^{-1}\cap C$ is finite.

\begin{proof}First, note that the group $H$ is not virtually cyclic since it has the same first-order theory as $G$, which contains a non-abelian free subgroup.

\smallskip

As a first step, let us assume that the radical $E(G)$ is trivial and that $C$ is abelian. Let us fix a non-trivial element $c$ of $C$. Assume towards a contradiction that $H$ is not acylindrically hyperbolic. Then, by \cite[Corollaries 2.2 and 2.3]{MO15}, the subgroup $C$ is not weakly malnormal. Hence, for all $h\in H$, the intersection of $hCh^{-1}$ and $C$ is infinite. In particular, this intersection contains a non-trivial element $z$. Since $C$ is abelian, this element $z$ commutes both with $c$ and $hch^{-1}$. Therefore, the following $\exists\forall\exists$-sentence is satisfied by $H$:\[\theta:\exists c\neq 1 \ \forall h  \ \exists z\neq 1 \ ([c,z]=1 \ \wedge \ [hch^{-1},z]=1).\]Since $G$ and $H$ have the same $\exists\forall\exists$-theory, the sentence $\theta$ is satisfied by $G$ as well. By Theorem \ref{th11}, the sentence $\theta$ is satisfied by $G\ast\langle t\rangle$, with $t$ of infinite order. This is a contradiction since no non-trivial element of $G\ast\langle t\rangle$ commutes both with $c$ and $tct^{-1}$; indeed, by writing the elements of $G\ast\langle t\rangle$ in normal form, one easily sees that the centralizer of $c$ in $G\ast\langle t\rangle$ is contained in $G$, and that the only element of $G$ that commutes with $tct^{-1}$ is the neutral element.

\smallskip

If $E(G)$ is non-trivial of order $N\geq 2$ and $C$ contains an abelian subgroup of index $d$, one has to modify the sentence $\theta$ a little bit in order to ensure that the elements $c$ and $z$ do not belong to $E(G)$ and belong to the abelian subgroup of $C$. For that we just replace the conditions "$\exists c\neq 1$" and "$\exists z\neq 1$" with the conditions "there exist $N+1$ pairwise distinct elements $c^d_1,\ldots,c^d_{N+1}$" and "there exist $N+1$ pairwise distinct elements $z^d_1,\ldots,z^d_{N+1}$".\end{proof}

\begin{rque}
More generally, if one assumes that $C$ virtually satisfies a law, the same proof works modulo some adjustments.
\end{rque}

\begin{rque}Note that the sentence $\theta$ given in the previous proof shows in particular that Baumslag-Solitar groups do not satisfy the conclusion of Theorem \ref{th11}: $BS(m,n)=\langle a,t \ \vert \ ta^mt^{-1}=a^m\rangle$ is not $\exists\forall\exists$-embedded into $BS(m,n)\ast\mathbb{Z}$. This observation is interesting because the main result of \cite{CGN19} applies to non-solvable Baumslag-Solitar groups (and shows that these groups have trivial positive theory); hence, the weak small cancellation conditions used in \cite{CGN19} for dealing with positive theory are not sufficient if one wants to deal with inequations.\end{rque}


\vspace{5mm}
\renewcommand{\refname}{Bibliography}

\begin{thebibliography}{10}

\bibitem{And18}
S.~Andr\'e.
\newblock Hyperbolicity and cubulability are preserved under elementary
  equivalence.
\newblock arXiv:1801.09411, 2018.

\bibitem{And19a}
S.~Andr\'e.
\newblock On {T}arski's problem for virtually free groups.
\newblock arXiv:1910.08464, 2019.

\bibitem{And20}
S.~Andr\'e.
\newblock Acylindrical hyperbolicity and existential closeness.
\newblock arXiv:2005.07220, 2020.

\bibitem{AMS13}
Y.~Antolin, A.~Minasyan, and A.~Sisto.
\newblock Commensurating endomorphisms of acylindrically hyperbolic groups and
  applications.
\newblock {\em Groups Geom. Dyn. 10 (2016), no. 4, pp. 1149-1210}, 2013.

\bibitem{AM18}
N.~Avni and C.~Meiri.
\newblock Words have bounded width in $\mathrm{SL}(n,\mathbb{Z})$.
\newblock 2018.

\bibitem{BBF19}
M.~Bestvina, K.~Bromberg, and K.~Fujiwara.
\newblock The verbal width of acylindrically hyperbolic groups.
\newblock {\em Algebraic \& Geometric Topology}, 19(1):477--489, 2019.

\bibitem{Bow08}
B.~Bowditch.
\newblock Tight geodesics in the curve complex.
\newblock {\em Invent. Math.}, 171(2):281--300, 2008.

\bibitem{CGN19}
M.~Casals-Ruiz, A.~Garreta, and J.~de~la Nuez~Gonz\'alez.
\newblock On the positive theory of groups acting on trees.
\newblock arXiv:1910.09000, 2019.

\bibitem{CGKN19}
M.~Casals-Ruiz, A.~Garreta, I.~Kazachkov, and J.~de~la Nuez~González.
\newblock Simple groups with infinite verbal width and the same positive theory
  as free groups.
\newblock arXiv:1911.02117, 2019.

\bibitem{CDP90}
M.~Coornaert, T.~Delzant, and A.~Papadopoulos.
\newblock {\em G\'eom\'etrie et th\'eorie des groupes}, volume 1441 of {\em
  Lecture Notes in Mathematics}.
\newblock Springer-Verlag, Berlin, 1990.
\newblock Les groupes hyperboliques de Gromov. [Gromov hyperbolic groups], With
  an English summary.

\bibitem{Cou13}
R.~Coulon.
\newblock Small cancellation theory and {B}urnside problem.
\newblock arXiv:1302.6933, 2013.

\bibitem{Cou16}
R.~Coulon.
\newblock Th\'{e}orie de la petite simplification: une approche
  g\'{e}om\'{e}trique [d'apr\`es {F}. {D}ahmani, {V}. {G}uirardel, {D}. {O}sin
  et {S}. {C}antat, {S}. {L}amy].
\newblock {\em Ast\'{e}risque}, (380, S\'{e}minaire Bourbaki. Vol.
  2014/2015):Exp. No. 1089, 1--33, 2016.

\bibitem{DGO17}
F.~Dahmani, V.~Guirardel, and D.~Osin.
\newblock Hyperbolically embedded subgroups and rotating families in groups
  acting on hyperbolic spaces.
\newblock {\em Mem. Amer. Math. Soc.}, 245(1156):v+152, 2017.

\bibitem{Nuez17}
J.~de~la Nuez~Gonz{\'a}lez.
\newblock On expansions of non-abelian free groups by cosets of a finite index
  subgroup.
\newblock 2017.

\bibitem{Dun98}
M.~J. Dunwoody.
\newblock Folding sequences.
\newblock In {\em The {E}pstein birthday schrift}, volume~1 of {\em Geom.
  Topol. Monogr.}, pages 139--158. Geom. Topol. Publ., Coventry, 1998.

\bibitem{Fru19}
J.~Fruchter.
\newblock Formal solutions over torsion-free acylindrically hyperbolic groups,
  2019.

\bibitem{Gro05}
D.~Groves.
\newblock Limit groups for relatively hyperbolic groups. {II}.
  {M}akanin-{R}azborov diagrams.
\newblock {\em Geom. Topol.}, 9:2319--2358, 2005.

\bibitem{GH19}
D.~Groves and M.~Hull.
\newblock Homomorphisms to acylindrically hyperbolic groups i: Equationally
  noetherian groups and families.
\newblock {\em Trans. Amer. Math. Soc.}, 372(10):7141--7190, 2019.

\bibitem{Gui01}
V.~Guirardel.
\newblock Bounding the complexity of small actions on $\mathbb{R}$-trees.
\newblock 2001.

\bibitem{Gui04}
V.~Guirardel.
\newblock Limit groups and groups acting freely on {$\Bbb R^n$}-trees.
\newblock {\em Geom. Topol.}, 8:1427--1470, 2004.

\bibitem{Gui08}
V.~Guirardel.
\newblock Actions of finitely generated groups on {$\Bbb R$}-trees.
\newblock {\em Ann. Inst. Fourier (Grenoble)}, 58(1):159--211, 2008.

\bibitem{Hei18}
S.~Heil.
\newblock Test sequences and formal solutions over hyperbolic groups.
\newblock arXiv:1811.06430, 2018.

\bibitem{HO16}
M.~Hull and D.~Osin.
\newblock Transitivity degrees of countable groups and acylindrical
  hyperbolicity.
\newblock {\em Israel J. Math.}, 216(1):307--353, 2016.

\bibitem{KM06}
O.~Kharlampovich and A.~Myasnikov.
\newblock Elementary theory of free non-abelian groups.
\newblock {\em J. Algebra}, 302(2):451--552, 2006.

\bibitem{Mer66}
J.~Merzljakov.
\newblock Positive formulae on free groups.
\newblock {\em Algebra i Logika Sem.}, 5(4):25--42, 1966.

\bibitem{MO15}
A.~Minasyan and D.~Osin.
\newblock Acylindrical hyperbolicity of groups acting on trees.
\newblock {\em Math. Ann.}, 362(3-4):1055--1105, 2015.

\bibitem{Osi16}
D.~Osin.
\newblock Acylindrically hyperbolic groups.
\newblock {\em Trans. Amer. Math. Soc.}, 368(2):851--888, 2016.

\bibitem{per08}
C.~Perin.
\newblock {\em {Elementary embeddings in torsion-free hyperbolic groups}}.
\newblock Theses, {Universit{\'e} de Caen}, October 2008.
\newblock Th{\`e}se r{\'e}dig{\'e}e en anglais, avec une introduction
  d{\'e}taill{\'e}e en fran{\c c}ais.

\bibitem{RW14}
C.~Reinfeldt and R.~Weidmann.
\newblock Makanin-razborov diagrams for hyperbolic groups.
\newblock 2014.

\bibitem{Rhem68}
A.~H. Rhemtulla.
\newblock A problem of bounded expressibility in free products.
\newblock {\em Proc. Cambridge Philos. Soc.}, 64:573--584, 1968.

\bibitem{RS94}
E.~Rips and Z.~Sela.
\newblock Structure and rigidity in hyperbolic groups. {I}.
\newblock {\em Geom. Funct. Anal.}, 4(3):337--371, 1994.

\bibitem{Sac73}
G.~Sacerdote.
\newblock Elementary properties of free groups.
\newblock {\em Trans. Amer. Math. Soc.}, 178:127--138, 1973.

\bibitem{Seg09}
D.~Segal.
\newblock {\em Words: notes on verbal width in groups}, volume 361.
\newblock Cambridge University Press, 2009.

\bibitem{Sel01}
Z.~Sela.
\newblock {D}iophantine {G}eometry over {G}roups. {I}. {M}akanin-{R}azborov
  diagrams.
\newblock {\em Publ. Math. Inst. Hautes \'Etudes Sci.}, (93):31--105, 2001.

\bibitem{Sel03}
Z.~Sela.
\newblock {D}iophantine {G}eometry over {G}roups {II}. completions, closures
  and formal solutions.
\newblock {\em Israel J. Math.}, 134(1):173--254, 2003.

\bibitem{Sel06}
Z.~Sela.
\newblock {D}iophantine {G}eometry over {G}roups. {VI}. {T}he elementary theory
  of a free group.
\newblock {\em Geom. Funct. Anal.}, 16(3):707--730, 2006.

\bibitem{Sel09}
Z.~Sela.
\newblock {D}iophantine {G}eometry over {G}roups. {VII}. {T}he elementary
  theory of a hyperbolic group.
\newblock {\em Proc. Lond. Math. Soc. (3)}, 99(1):217--273, 2009.

\bibitem{Sel10}
Z.~Sela.
\newblock {D}iophantine {G}eometry over {G}roups {X}. {T}he elementary theory
  of free products of groups.
\newblock arXiv:1012.0044, 2010.

\end{thebibliography}
\def\cprime{$'$} \def\cprime{$'$}

\vspace{10mm}

\textbf{Simon André}

Vanderbilt University.

E-mail address: \href{mailto:simon.andre@vanderbilt.edu}{simon.andre@vanderbilt.edu}

\vspace{5mm}

\textbf{Jonathan Fruchter}

University of Oxford.

E-mail address: \href{mailto:fruchter@maths.ox.ac.uk}{fruchter@maths.ox.ac.uk}

\end{document}